\documentclass[11pt,twoside,a4paper]{article}
\usepackage{makeidx}
\usepackage[dvips]{graphicx}
\usepackage{amscd}
\usepackage{amsmath}
\usepackage{amssymb}
\usepackage{latexsym}
\usepackage[latin1]{inputenc}
\usepackage[T1]{fontenc}   
\usepackage[english]{babel}
\newcommand{\tun}{\begin{picture}(5,0)(-2,-1)
\put(0,0){\circle*{2}}
\end{picture}}

\newcommand{\tdeux}{\begin{picture}(7,7)(0,-1)
\put(3,0){\circle*{2}}
\put(3,0){\line(0,1){5}}
\put(3,5){\circle*{2}}
\end{picture}}

\newcommand{\ttroisun}{\begin{picture}(15,8)(-5,-1)
\put(3,0){\circle*{2}}
\put(-0.65,0){$\vee$}
\put(6,7){\circle*{2}}
\put(0,7){\circle*{2}}
\end{picture}}
\newcommand{\ttroisdeux}{\begin{picture}(5,12)(-2,-1)
\put(0,0){\circle*{2}}
\put(0,0){\line(0,1){5}}
\put(0,5){\circle*{2}}
\put(0,5){\line(0,1){5}}
\put(0,10){\circle*{2}}
\end{picture}}

\newcommand{\tquatreun}{\begin{picture}(15,12)(-5,-1)
\put(3,0){\circle*{2}}
\put(-0.65,0){$\vee$}
\put(6,7){\circle*{2}}
\put(0,7){\circle*{2}}
\put(3,7){\circle*{2}}
\put(3,0){\line(0,1){7}}
\end{picture}}
\newcommand{\tquatredeux}{\begin{picture}(15,18)(-5,-1)
\put(3,0){\circle*{2}}
\put(-0.65,0){$\vee$}
\put(6,7){\circle*{2}}
\put(0,7){\circle*{2}}
\put(0,14){\circle*{2}}
\put(0,7){\line(0,1){7}}
\end{picture}}

\newcommand{\tquatrequatre}{\begin{picture}(15,18)(-5,-1)
\put(3,5){\circle*{2}}
\put(-0.65,5){$\vee$}
\put(6,12){\circle*{2}}
\put(0,12){\circle*{2}}
\put(3,0){\circle*{2}}
\put(3,0){\line(0,1){5}}
\end{picture}}
\newcommand{\tquatrecinq}{\begin{picture}(9,19)(-2,-1)
\put(0,0){\circle*{2}}
\put(0,0){\line(0,1){5}}
\put(0,5){\circle*{2}}
\put(0,5){\line(0,1){5}}
\put(0,10){\circle*{2}}
\put(0,10){\line(0,1){5}}
\put(0,15){\circle*{2}}
\end{picture}}

\newcommand{\tcinqun}{\begin{picture}(20,8)(-5,-1)
\put(3,0){\circle*{2}}
\put(-0.5,0){$\vee$}
\put(6,7){\circle*{2}}
\put(0,7){\circle*{2}}
\put(3,0){\line(2,1){10}}
\put(3,0){\line(-2,1){10}}
\put(-7,5){\circle*{2}}
\put(13,5){\circle*{2}}
\end{picture}}
\newcommand{\tcinqdeux}{\begin{picture}(15,14)(-5,-1)
\put(3,0){\circle*{2}}
\put(-0.65,0){$\vee$}
\put(6,7){\circle*{2}}
\put(0,7){\circle*{2}}
\put(3,7){\circle*{2}}
\put(3,0){\line(0,1){7}}
\put(0,7){\line(0,1){7}}
\put(0,14){\circle*{2}}
\end{picture}}

\newcommand{\tcinqcinq}{\begin{picture}(15,19)(-5,-1)
\put(3,0){\circle*{2}}
\put(-0.65,0){$\vee$}
\put(6,7){\circle*{2}}
\put(0,7){\circle*{2}}
\put(6,14){\circle*{2}}
\put(6,7){\line(0,1){7}}
\put(0,14){\circle*{2}}
\put(0,7){\line(0,1){7}}
\end{picture}}

\newcommand{\tcinqsept}{\begin{picture}(15,8)(-5,-1)
\put(3,0){\circle*{2}}
\put(-0.65,0){$\vee$}
\put(6,7){\circle*{2}}
\put(0,7){\circle*{2}}
\put(2.35,7){$\vee$}
\put(3,14){\circle*{2}}
\put(9,14){\circle*{2}}
\end{picture}}

\newcommand{\tcinqneuf}{\begin{picture}(15,26)(-5,-1)
\put(3,0){\circle*{2}}
\put(-0.65,0){$\vee$}
\put(6,7){\circle*{2}}
\put(0,7){\circle*{2}}
\put(6,14){\circle*{2}}
\put(6,7){\line(0,1){7}}
\put(6,21){\circle*{2}}
\put(6,14){\line(0,1){7}}
\end{picture}}
\newcommand{\tcinqdix}{\begin{picture}(15,19)(-5,-1)
\put(3,5){\circle*{2}}
\put(-0.5,5){$\vee$}
\put(6,12){\circle*{2}}
\put(0,12){\circle*{2}}
\put(3,0){\circle*{2}}
\put(3,0){\line(0,1){12}}
\put(3,12){\circle*{2}}
\end{picture}}
\newcommand{\tcinqonze}{\begin{picture}(15,26)(-5,-1)
\put(3,5){\circle*{2}}
\put(-0.65,5){$\vee$}
\put(6,12){\circle*{2}}
\put(0,12){\circle*{2}}
\put(3,0){\circle*{2}}
\put(3,0){\line(0,1){5}}
\put(0,12){\line(0,1){7}}
\put(0,19){\circle*{2}}
\end{picture}}

\newcommand{\tcinqtreize}{\begin{picture}(5,26)(-2,-1)
\put(0,0){\circle*{2}}
\put(0,0){\line(0,1){7}}
\put(0,7){\circle*{2}}
\put(0,7){\line(0,1){7}}
\put(0,14){\circle*{2}}
\put(-3.65,14){$\vee$}
\put(-3,21){\circle*{2}}
\put(3,21){\circle*{2}}
\end{picture}}
\newcommand{\tcinqquatorze}{\begin{picture}(9,26)(-5,-1)
\put(0,0){\circle*{2}}
\put(0,0){\line(0,1){5}}
\put(0,5){\circle*{2}}
\put(0,5){\line(0,1){5}}
\put(0,10){\circle*{2}}
\put(0,10){\line(0,1){5}}
\put(0,15){\circle*{2}}
\put(0,15){\line(0,1){5}}
\put(0,20){\circle*{2}}
\end{picture}}

\newcommand{\tdun}[1]{\begin{picture}(10,5)(-2,-1)
\put(0,0){\circle*{2}}
\put(3,-2){\tiny #1}
\end{picture}}

\newcommand{\tddeux}[2]{\begin{picture}(12,5)(0,-1)
\put(3,0){\circle*{2}}
\put(3,0){\line(0,1){5}}
\put(3,5){\circle*{2}}
\put(6,-2){\tiny #1}
\put(6,3){\tiny #2}
\end{picture}}

\newcommand{\tdtroisun}[3]{\begin{picture}(20,12)(-5,-1)
\put(3,0){\circle*{2}}
\put(-0.65,0){$\vee$}
\put(6,7){\circle*{2}}
\put(0,7){\circle*{2}}
\put(5,-2){\tiny #1}
\put(9,5){\tiny #2}
\put(-5,5){\tiny #3}
\end{picture}}
\newcommand{\tdtroisdeux}[3]{\begin{picture}(12,12)(-2,-1)
\put(0,0){\circle*{2}}
\put(0,0){\line(0,1){5}}
\put(0,5){\circle*{2}}
\put(0,5){\line(0,1){5}}
\put(0,10){\circle*{2}}
\put(3,-2){\tiny #1}
\put(3,3){\tiny #2}
\put(3,9){\tiny #3}
\end{picture}}

\newcommand{\tdquatreun}[4]{\begin{picture}(20,12)(-5,-1)
\put(3,0){\circle*{2}}
\put(-0.6,0){$\vee$}
\put(6,7){\circle*{2}}
\put(0,7){\circle*{2}}
\put(3,7){\circle*{2}}
\put(3,0){\line(0,1){7}}
\put(5,-2){\tiny #1}
\put(8.5,5){\tiny #2}
\put(1,10){\tiny #3}
\put(-5,5){\tiny #4}
\end{picture}}
\newcommand{\tdquatredeux}[4]{\begin{picture}(20,20)(-5,-1)
\put(3,0){\circle*{2}}
\put(-.65,0){$\vee$}
\put(6,7){\circle*{2}}
\put(0,7){\circle*{2}}
\put(0,14){\circle*{2}}
\put(0,7){\line(0,1){7}}
\put(5,-2){\tiny #1}
\put(9,5){\tiny #2}
\put(-5,5){\tiny #3}
\put(-5,12){\tiny #4}
\end{picture}}
\newcommand{\tdquatretrois}[4]{\begin{picture}(20,20)(-5,-1)
\put(3,0){\circle*{2}}
\put(-.65,0){$\vee$}
\put(6,7){\circle*{2}}
\put(0,7){\circle*{2}}
\put(6,14){\circle*{2}}
\put(6,7){\line(0,1){7}}
\put(5,-2){\tiny #1}
\put(9,5){\tiny #2}
\put(-5,5){\tiny #4}
\put(9,12){\tiny #3}
\end{picture}}
\newcommand{\tdquatrequatre}[4]{\begin{picture}(20,14)(-5,-1)
\put(3,5){\circle*{2}}
\put(-.65,5){$\vee$}
\put(6,12){\circle*{2}}
\put(0,12){\circle*{2}}
\put(3,0){\circle*{2}}
\put(3,0){\line(0,1){5}}
\put(6,-3){\tiny #1}
\put(6,4){\tiny #2}
\put(9,12){\tiny #3}
\put(-5,12){\tiny #4}
\end{picture}}
\newcommand{\tdquatrecinq}[4]{\begin{picture}(12,19)(-2,-1)
\put(0,0){\circle*{2}}
\put(0,0){\line(0,1){5}}
\put(0,5){\circle*{2}}
\put(0,5){\line(0,1){5}}
\put(0,10){\circle*{2}}
\put(0,10){\line(0,1){5}}
\put(0,15){\circle*{2}}
\put(3,-2){\tiny #1}
\put(3,3){\tiny #2}
\put(3,9){\tiny #3}
\put(3,14){\tiny #4}
\end{picture}}

\newcommand{\tquatredeuxa}{\begin{picture}(15,18)(-5,-1)
\put(3,0){\circle*{2}}
\put(-0.2,0.2){$\vee$}
\put(6,7){\circle*{2}}
\put(0,7){\circle*{2}}
\put(0,14){\circle*{2}}
\put(0,7){\line(0,1){7}}
\put(-4,4){\line(1,0){7}}
\put(5,-2){\tiny $d$}
\put(9,5){\tiny $c$}
\put(-5,5){\tiny $b$}
\put(-5,12){\tiny $a$}
\end{picture}}

\newcommand{\tquatredeuxb}{\begin{picture}(15,18)(-5,-1)
\put(3,0){\circle*{2}}
\put(-0.2,0.2){$\vee$}
\put(6,7){\circle*{2}}
\put(0,7){\circle*{2}}
\put(0,14){\circle*{2}}
\put(0,7){\line(0,1){7}}
\put(-4,11){\line(1,0){7}}
\put(5,-2){\tiny $d$}
\put(9,5){\tiny $c$}
\put(-5,5){\tiny $b$}
\put(-5,12){\tiny $a$}
\end{picture}}

\newcommand{\tquatredeuxc}{\begin{picture}(15,18)(-5,-1)
\put(3,0){\circle*{2}}
\put(-0.2,0.2){$\vee$}
\put(6,7){\circle*{2}}
\put(0,7){\circle*{2}}
\put(0,14){\circle*{2}}
\put(0,7){\line(0,1){7}}
\put(3,4){\line(1,0){7}}
\put(5,-2){\tiny $d$}
\put(9,5){\tiny $c$}
\put(-5,5){\tiny $b$}
\put(-5,12){\tiny $a$}
\end{picture}}

\newcommand{\tquatredeuxd}{\begin{picture}(15,18)(-5,-1)
\put(3,0){\circle*{2}}
\put(-0.2,0.2){$\vee$}
\put(6,7){\circle*{2}}
\put(0,7){\circle*{2}}
\put(0,14){\circle*{2}}
\put(0,7){\line(0,1){7}}
\put(-4,4){\line(1,0){7}}
\put(-4,11){\line(1,0){7}}
\put(5,-2){\tiny $d$}
\put(9,5){\tiny $c$}
\put(-5,5){\tiny $b$}
\put(-5,12){\tiny $a$}
\end{picture}}

\newcommand{\tquatredeuxe}{\begin{picture}(15,18)(-5,-1)
\put(3,0){\circle*{2}}
\put(-0.2,0.2){$\vee$}
\put(6,7){\circle*{2}}
\put(0,7){\circle*{2}}
\put(0,14){\circle*{2}}
\put(0,7){\line(0,1){7}}
\put(-4,4){\line(1,0){7}}
\put(3,4){\line(1,0){7}}
\put(5,-2){\tiny $d$}
\put(9,5){\tiny $c$}
\put(-5,5){\tiny $b$}
\put(-5,12){\tiny $a$}
\end{picture}}

\newcommand{\tquatredeuxf}{\begin{picture}(15,18)(-5,-1)
\put(3,0){\circle*{2}}
\put(-0.2,0.2){$\vee$}
\put(6,7){\circle*{2}}
\put(0,7){\circle*{2}}
\put(0,14){\circle*{2}}
\put(0,7){\line(0,1){7}}
\put(-4,11){\line(1,0){7}}
\put(3,4){\line(1,0){7}}
\put(5,-2){\tiny $d$}
\put(9,5){\tiny $c$}
\put(-5,5){\tiny $b$}
\put(-5,12){\tiny $a$}
\end{picture}}

\newcommand{\tquatredeuxg}{\begin{picture}(15,18)(-5,-1)
\put(3,0){\circle*{2}}
\put(-0.2,0.2){$\vee$}
\put(6,7){\circle*{2}}
\put(0,7){\circle*{2}}
\put(0,14){\circle*{2}}
\put(0,7){\line(0,1){7}}
\put(-4,4){\line(1,0){7}}
\put(-4,11){\line(1,0){7}}
\put(3,4){\line(1,0){7}}
\put(5,-2){\tiny $d$}
\put(9,5){\tiny $c$}
\put(-5,5){\tiny $b$}
\put(-5,12){\tiny $a$}
\end{picture}}

\newcommand{\tdpartun}[4]{\begin{picture}(17,19)(-2,-1)
\put(0,0){\circle*{2}}
\put(0,0){\line(0,1){5}}
\put(0,5){\circle*{2}}
\put(-1.5,7){.}
\put(-1.5,9){.}
\put(-1.5,11){.}
\put(0,14){\circle*{2}}
\put(0,14){\line(0,1){5}}
\put(0,19){\circle*{2}}
\put(3,-2){\tiny #1}
\put(3,3){\tiny #2}
\put(3,13){\tiny #3}
\put(3,18){\tiny #4}
\end{picture}}

\newcommand{\tdpartdeux}[5]{\begin{picture}(17,24)(-2,-1)
\put(0,0){\circle*{2}}
\put(0,0){\line(0,1){5}}
\put(0,5){\circle*{2}}
\put(-1.5,7){.}
\put(-1.5,9){.}
\put(-1.5,11){.}
\put(0,14){\circle*{2}}
\put(0,14){\line(0,1){5}}
\put(0,19){\circle*{2}}
\put(0,19){\line(0,1){5}}
\put(0,24){\circle*{2}}
\put(3,-2){\tiny #1}
\put(3,3){\tiny #2}
\put(3,13){\tiny #3}
\put(3,18){\tiny #4}
\put(3,23){\tiny #5}
\end{picture}}

\newcommand{\tdparttrois}[5]{\begin{picture}(22,19)(-5,-1)
\put(0,0){\circle*{2}}
\put(0,0){\line(0,1){5}}
\put(0,5){\circle*{2}}
\put(-1.5,7){.}
\put(-1.5,9){.}
\put(-1.5,11){.}
\put(0,14){\circle*{2}}
\put(-3.45,14){$\vee$}
\put(3,21){\circle*{2}}
\put(-3,21){\circle*{2}}
\put(3,-2){\tiny #1}
\put(3,3){\tiny #2}
\put(3,13){\tiny #3}
\put(3,24){\tiny #4}
\put(-3,24){\tiny #5}
\end{picture}}

\newcommand{\tdpartquatre}[5]{\begin{picture}(27,19)(-5,-1)
\put(0,0){\circle*{2}}
\put(-1.5,2){.}
\put(-1.5,4){.}
\put(-1.5,6){.}
\put(0,9){\circle*{2}}
\put(-3.45,9){$\vee$}
\put(3,16){\circle*{2}}
\put(-3,16){\circle*{2}}
\put(1.5,18){.}
\put(1.5,20){.}
\put(1.5,22){.}
\put(3,25){\circle*{2}}
\put(3,-2){\tiny #1}
\put(3,8){\tiny #2}
\put(6,15){\tiny #3}
\put(6,24){\tiny #4}
\put(-3,19){\tiny #5}
\end{picture}}

\input{xy}
\xyoption{all}

\newcommand{\h}{\mathcal{H}}
\newcommand{\g}{\mathfrak{g}}
\newcommand{\lies}{\mathfrak{g}_{(S)}}
\newcommand{\gfdb}{\mathfrak{g}_{FdB}}
\newcommand{\hs}{\mathcal{H}_{(S)}}
\newcommand{\D}{\mathcal{D}}
\newcommand{\T}{\mathcal{T}}
\newcommand{\F}{\mathcal{F}}
\newcommand{\R}{\:\mathcal{R}\:}
\newcommand{\gs}{G_{(S)}}
\newcommand{\modulev}{\mathbb{V}}
\newcommand{\modulew}{\mathbb{W}}
\newcommand{\fleche}[1]{\stackrel{#1}{\longrightarrow}}
\newcommand{\Set}{\mathbf{Set}}
\newcommand{\Vect}{\mathbf{Vect}}
\newcommand{\Lie}{\mathbf{Lie}}

\title{Systems of Dyson-Schwinger equations}
\date{}
\author{Loïc Foissy\footnote{e-mail: loic.foissy@univ-reims.fr; webpage: http://loic.foissy.free.fr/pageperso/accueil.html}
\\
{\small{\it Laboratoire de Mathématiques - FRE3111, Université de Reims}}\\
\small{{\it Moulin de la Housse - BP 1039 - 51687 REIMS Cedex 2, France}}\\
}

\newtheorem{defi}{\indent Definition}
\newtheorem{lemma}[defi]{\indent Lemma}
\newtheorem{cor}[defi]{\indent Corollary}
\newtheorem{theo}[defi]{\indent Theorem}
\newtheorem{prop}[defi]{\indent Proposition}

\newenvironment{proof}{{\bf Proof.}}{\hfill $\Box$}

\begin{document}
\maketitle

ABSTRACT. We consider systems of combinatorial Dyson-Schwinger equations (briefly, SDSE) $X_1=B^+_1(F_1(X_1,\ldots,X_N))$,
$\ldots$,  $X_N=B^+_N(F_N(X_1,\ldots,X_N))$ in the  Connes-Kreimer Hopf algebra $\mathcal{H}_I$ of rooted trees decorated by $I=\{1,\ldots,N\}$, 
where $B^+_i$ is the operator of grafting on a root decorated by $i$, and $F_1,\ldots,F_N$ are non-constant formal series.
The unique solution $X=(X_1,\ldots,X_N)$ of this equation generates a graded subalgebra $\mathcal{H}_{(S)}$ of $\mathcal{H}_I$.
We characterise here all the families of formal series $(F_1,\ldots,F_N)$ such that $\mathcal{H}_{(S)}$ is a Hopf subalgebra.
More precisely, we define three operations on SDSE (change of variables, dilatation and extension)
and give two families of SDSE (cyclic and fundamental systems), and prove that any SDSE $(S)$ such that $\mathcal{H}_{(S)}$ is Hopf
is the concatenation of several fundamental or cyclic systems after the application of a change of variables, a dilatation and iterated extensions.

We also describe the Hopf algebra $\mathcal{H}_{(S)}$ as the dual of the  enveloping algebra of a Lie algebra $\mathfrak{g}_{(S)}$
of one of the following types:
\begin{enumerate}
\item $\mathfrak{g}_{(S)}$ is a Lie algebra of paths associated to a certain oriented graph.
\item Or $\mathfrak{g}_{(S)}$ is an iterated extension of the Faà di Bruno Lie algebra.
\item Or $\mathfrak{g}_{(S)}$ is an iterated extension of an abelian Lie algebra.\\
\end{enumerate}


MATHEMATICS SUBJECT CLASSIFICATION. Primary 16W30. Secondary 81T15, 81T18.

\tableofcontents

\section*{Introduction}

The Connes-Kreimer Hopf algebra of rooted trees is introduced in \cite{Kreimer1} and studied in 
\cite{Broadhurst,Chapoton1,Quevedo,Connes,Figueroa,Foissy1,Hoffman,Panaite}. This graded, commutative, non-cocommutative Hopf algebra
is generated by the set of rooted trees. We shall work here with a decorated version $\h_\D$ of this algebra, where $\D$ is a finite, non-empty set,
replacing rooted trees by rooted trees with vertices decorated by the elements of  $\D$. This algebra has a family of operators $(B^+_d)_{d\in \D}$ 
indexed by $\D$, where $B^+_d$ sends a forest $F$ to the rooted tree obtained by grafting the trees of $F$ on a common root decorated by $d$.
These operators satisfy the following equation: for all $x \in \h_\D$,
$$\Delta \circ B_d^+(x)=B_d^+(x) \otimes 1+(Id\otimes B_d^+)\circ \Delta(x).$$
As explained in \cite{Connes}, this means that $B_d^+$ is a $1$-cocycle for a certain cohomology of coalgebras, dual to the Hochschild cohomology.\\

We are interested here in systems of combinatorial Dyson-Schwinger equations (briefly, SDSE), that is to say, 
if the set of decorations is $\{1,\ldots,N\}$, a system $(S)$ of the form:
$$\left\{ \begin{array}{rcl}
X_1&=&B_1^+(F_1(X_1,\ldots,X_N)),\\
&\vdots&\\
X_N&=&B_N^+(F_N(X_1,\ldots,X_N)),
\end{array}\right.$$
where $F_1,\ldots,F_N\in K[[h_1,\ldots,h_N]] $ are formal series in $N$ indeterminates. These systems (in a Feynman graph version) are used 
in Quantum Field Theory, as it is explained in \cite{Bergbauer,Kreimer2,Kreimer3}. They possess a unique solution,  which is a family of 
$N$ formal series in rooted trees, or equivalently elements of a completion of $\h_\D$. The homogeneous components of these elements generate 
a subalgebra $\hs$ of $\h_\D$. Our problem here is to determine Hopf SDSE, that is to say SDSE $(S)$ such that  $\hs$
is a Hopf subalgebra of $\h_\D$. In the case of a single combinatorial Dyson-Schwinger equation, this question has been answered in \cite{Foissy2}.\\

In order to answer this, we first associate an oriented graph to any SDSE, reflecting the dependence of the different $X_i$'s;
more precisely, the vertices of $\gs$ are the elements of $I$, and there is an edge from $i$ to $j$ if $F_i$ depends on $h_j$.
We shall say that $(S)$ is connected if $\gs$ is connected. Noting that any SDSE is the disjoint union of several connected SDSE, 
we can restrict our study to connected SDSE. We introduce three operations on Hopf SDSE:
\begin{itemize}
\item Change of variables, which replaces $h_i$ by $\lambda_i h_i$ for all $i\in I$, where $\lambda_i \neq 0$ for all $i$.
This operation replaces $\hs$ by an isomorphic Hopf algebra and does not change $\gs$.
\item Dilatation, which replaces each vertex of $\gs$ by several vertices. This operation increases the number of vertices. For example, consider:
$$(S):\left\{ \begin{array}{rcl}
X_1&=&B^+_1(f(X_1,X_2)),\\
X_2&=&B^+_2(g(X_1,X_2)),
\end{array}\right.$$
where $f,g \in K[[h_1,h_2]]$; then the following SDSE is a dilatation of $(S)$:
$$(S'):\left\{ \begin{array}{rcl}
X_1&=&B^+_1(f(X_1+X_2+X_3,X_4+X_5)),\\
X_2&=&B^+_2(f(X_1+X_2+X_3,X_4+X_5)),\\
X_3&=&B^+_3(f(X_1+X_2+X_3,X_4+X_5)),\\
X_4&=&B^+_4(g(X_1+X_2+X_3,X_4+X_5)),\\
X_5&=&B^+_5(g(X_1+X_2+X_3,X_4+X_5)),
\end{array}\right.$$
\item Extension, which adds a vertex $0$ to $\gs$ with an affine formal series. This operation increases the number of vertices by $1$.
For example, consider:
$$(S): \left\{ \begin{array}{rcl}
X_1&=&B^+_1(f(X_1,X_2)),\\
X_2&=&B^+_2(f(X_1,X_2)),
\end{array}\right.$$
where $f \in K[[h_1,h_2]]$ and $a,b \in K$; then the following SDSE is an extension of $(S)$:
$$(S'):\left\{ \begin{array}{rcl}
X_0&=&B^+_0(1+aX_1+bX_2),\\
X_1&=&B^+_1(f(X_1,X_2)),\\
X_2&=&B^+_2(f(X_1,X_2)),
\end{array}\right. $$
\end{itemize}
We then introduce two families of Hopf SDSE:
\begin{itemize}
\item Cycles, which are SDSE such that the associated graph is an oriented graph and all the formal series of the system are affine; see theorem \ref{30}.
For example, the following system is a 4-cycle:
$$\left\{\begin{array}{rcl}
X_1&=&B^+_1(1+X_2),\\
X_2&=&B^+_2(1+X_3),\\
X_3&=&B^+_3(1+X_4),\\
X_4&=&B^+_4(1+X_1).
\end{array}\right.$$
The associated oriented graph is:
$$\xymatrix{1\ar[r]&2\ar[d]\\4\ar[u]&3\ar[l]}$$
\item Fundamental SDSE, described in theorem \ref{32}. Here is an example of a fundamental SDSE:
$$\left\{\begin{array}{rcl}
X_1&=&B^+_1\left(f_{\beta_1}(X_1)f_{\frac{\beta_2}{1+\beta_2}}((1+\beta_2)h_2)(1-h_3)^{-1}(1-h_4)^{-1}\right),\\[4mm]
X_2&=&B^+_2\left(f_{\frac{\beta_1}{1+\beta_1}}(X_1)f_{\beta_2}(h_2)(1-h_3)^{-1}(1-h_4)^{-1}\right),\\[4mm]
X_3&=&B^+_3\left(f_{\frac{\beta_1}{1+\beta_1}}((1+\beta_1)X_1)f_{\frac{\beta_2}{1+\beta_2}}((1+\beta_2)h_2)(1-h_4)^{-1}\right),\\[4mm]
X_4&=&B^+_4\left(f_{\frac{\beta_1}{1+\beta_1}}((1+\beta_1)X_1)f_{\frac{\beta_2}{1+\beta_2}}((1+\beta_2)h_2)(1-h_3)^{-1}\right),\\[4mm]
X_5&=&B^+_5\left(f_{\frac{\beta_1}{1+\beta_1}}((1+\beta_1)X_1)f_{\frac{\beta_2}{1+\beta_2}}((1+\beta_2)h_2)(1-h_3)^{-1}(1-h_4)^{-1}\right),
\end{array}\right.$$
where $\beta_1,\beta_2 \in K-\{-1\}$ and, for all $\beta \in K$, $f_\beta$ is the following formal series:
$$f_\beta(h)=\sum_{k=0}^{\infty} \frac{(1+\beta)\cdots(1+(k-1)\beta)}{k!} h^k.$$
The associated oriented graph is:
$$\xymatrix{1\ar@(ul,ur)\ar@{<->}[rr] \ar@{<->}[d] \ar@{<->}[drr]&&2\ar@(ul,ur)\ar@{<->}[d] \ar@{<->}[lld]\\
3\ar@{<->}[rr]&&4\\
&5\ar@/^3pc/[luu] \ar@/_3pc/[ruu] \ar[lu] \ar[ru]&}$$
\end{itemize}
The main result of this paper is theorem \ref{14}, which says that any connected Hopf SDSE is obtained by a dilatation and a finite
number of iterated extensions of a cycle or a fundamental SDSE.\\

Let us now give a few explanations on the way this result is obtained. An important tool is given by a family indexed by $I^2$ of scalar sequences 
$\left(\lambda_n^{(i,j)}\right)_{n\geq 1}$  associated to any Hopf SDSE. They allow to reconstruct the coefficients of the formal series of $(S)$, 
as explained in proposition \ref{19}. Particular cases of possible sequence $\left(\lambda_n^{(i,j)}\right)_{n\geq 1}$ are affine sequences, 
up to a finite number of terms: this leads to the notion of level of a vertex. It is shown that level decreases along the oriented paths of $\gs$ 
(proposition \ref{23}), and this implies the following alternative if $(S)$ is connected: any vertex is of finite level or no vertex is of finite level.
In particular, any vertex of a fundamental SDSE is of finite level, whereas no vertex of a cycle is of finite level.

We then consider two special families of SDSE:
\begin{itemize}
\item We first assume that the graph associated to $(S)$ does not contain any vertex related to itself. This case includes cycles and their dilatations
(called multicycles), and a special case of fundamental SDSE called quasi-complete SDSE. We show, using graph-theoretical considerations 
and the coefficients $\lambda_n^{(i,j)}$, that under an hypothesis of symmetry, they are the only possibilities.
\item We then assume that any vertex of $(S)$ has an ascendant related to itself. We then prove that $(S)$ is fundamental.
\end{itemize}
This results are then unified in corollary \ref{50}. It says that any Hopf SDSE with a connected graph contains a multicycle or a a fundamental SDSE 
$(S_0)$ and is obtained from $(S_0)$ by adding repeatedly a finite number of vertices. This result is precised for the multicycle case in theorem \ref{51}
and for the fundamental case in theorem \ref{52}. The compilation of these results then proves theorem \ref{14}.\\
 
The end of the paper is devoted to the description of the Hopf algebras $\hs$. By the Cartier-Quillen-Milnor-Moore theorem, they are dual of 
enveloping algebra $\mathcal{U}(\lies)$, and it turns out that $\lies$ is a pre-Lie algebra \cite{Chapoton3}, that is to say it has a bilinear product $\star$
such that for all $f,g,h \in \lies$:
$$(f\star g) \star h-f\star(g\star h)=(g\star f) \star h-g\star(f\star h).$$
This relation implies that the antisymmetrisation of $\star$ is a Lie bracket. In our case, $\lies$ has a basis $(f_i(k))_{i\in I,k\geq 1}$ and 
by proposition \ref{21} its pre-Lie product is given by:
$$f_j(l)\star f_i(k)=\lambda_k^{(i,j)} f_i(k+l).$$
The product $\star$ can be associative, for example in the multicyclic case. Then, up to a change of variables, $f_j(l)\star f_i(k)=f_i(k+l)$ 
if there is an oriented path of length $k$ from $i$ to $j$ in the oriented graph associated to $(S)$, or $0$ otherwise; see proposition \ref{57}.

The fundamental case is separated into two subcases. In the non-abelian case, the Lie algebra $\lies$ is described as an iterated
semi-direct product of the Faà di Bruno Lie algebra by infinite dimensional modules. Similarly, the character group of $\hs$ is an iterated
semi-direct product of the Faà di Bruno group of formal diffeomorphisms by modules of formal series:
$$Ch(\hs)=G_m\rtimes(G_{m-1}\rtimes(\cdots G_2 \rtimes (G_1 \rtimes G_0)\cdots),$$
where $G_0$ is the Faà di Bruno group and $G_1,\ldots,G_{m-1}$ are isomorphic to direct sums of $(tK[[t]],+)$ as groups; see theorem \ref{65}.
The second subcase is similar, replacing the Faà di Bruno Lie algebra by an abelian Lie algebra; see theorem \ref{72}.\\

This text is organised as follows: the first section gives some recalls on the structure of Hopf algebra of $\h_\D$ and on the pre-Lie product on 
$\lies=Prim\left(\hs^*\right)$. In the second section are given the definitions of SDSE and their different operations: change of variables, dilatation
and extension. The main theorem of the text is also stated in this section. The following section introduces the coefficients $\lambda_n^{(i,j)}$
and their properties, especially their link with the pre-Lie product of $\lies$. The level of a vertex is defined in the fourth section,
which also contains lemmas on vertices of level $0$, $1$ or $\geq 2$, before that fundamental and multicyclic SDSE are introduced in the fifth section.
The next section contains preliminary results about graphs with no self-dependent vertices or such that any vertex is the descendant of
a self-dependent vertex, and the main theorem is finally proved in the seventh section. The next three sections deals with the description of the 
Lie algebra $\lies$ and the group $Ch\left(\hs\right)$ when $\lies$ is associative, in the non-abelian, fundamental case and finally in the abelian, 
fundamental case. The last section gives a functorial way to characterise pre-Lie algebra from the operation of dilatations of Hopf SDSE. \\

{\bf Notations.} We denote by $K$ a commutative field of characteristic zero. All vector spaces, algebras, coalgebras, Hopf algebras, etc.
will be taken over $K$.

\section{Preliminaries}

\subsection{Decorated rooted trees}

\begin{defi} \textnormal{\cite{Stanley1,Stanley2}
\begin{enumerate}
\item A {\it rooted tree} $t$ is a finite graph, without loops, with a special vertex called the {\it root} of $t$. 
The {\it weight} of $t$ is the number of its vertices. The set of rooted trees will be denoted by $\T$.
\item Let $\D$ be a non-empty set. A {\it rooted tree decorated by $\D$} is a rooted tree
with an application from the set of its vertices into $\D$. 
The set of rooted trees decorated by $\D$ will be denoted by $\T_\D$.
\item Let $i\in \D$. The set of rooted trees decorated by $\D$ with root decorated by $i$ will be denoted by $\T_\D^{(i)}$.
\end{enumerate}} \end{defi}

{\bf Examples.} \begin{enumerate}
\item Rooted trees with weight smaller than $5$:
$$\tun;\tdeux;\ttroisun,\ttroisdeux;\tquatreun, \tquatredeux,\tquatrequatre,\tquatrecinq;
\tcinqun, \tcinqdeux,\tcinqcinq,\tcinqsept,\tcinqneuf,\tcinqdix,\tcinqonze,\tcinqtreize,\tcinqquatorze.$$
\item Rooted trees decorated by $\D$ with weight smaller than $4$:
$$\tdun{$a$};\: a\in \D,\hspace{1cm} \tddeux{$a$}{$b$}\: (a,b)\in \D^2;\hspace{1cm}
\tdtroisun{$a$}{$c$}{$b$}=\tdtroisun{$a$}{$b$}{$c$},\: \tdtroisdeux{$a$}{$b$}{$c$},\:(a,b,c)\in \D^3;$$
$$ \tdquatreun{$a$}{$d$}{$c$}{$b$}=\tdquatreun{$a$}{$c$}{$d$}{$b$}=\tdquatreun{$a$}{$d$}{$b$}{$c$}=\tdquatreun{$a$}{$b$}{$d$}{$c$}
=\tdquatreun{$a$}{$c$}{$b$}{$d$}=\tdquatreun{$a$}{$b$}{$c$}{$d$},\: \tdquatredeux{$a$}{$d$}{$b$}{$c$}=\tdquatretrois{$a$}{$b$}{$c$}{$d$},\:
 \tdquatrequatre{$a$}{$b$}{$d$}{$c$}= \tdquatrequatre{$a$}{$b$}{$c$}{$d$},\: \tdquatrecinq{$a$}{$b$}{$c$}{$d$},\:(a,b,c,d)\in \D^4.$$
\end{enumerate}

\begin{defi} \textnormal{\begin{enumerate}
\item We denote by $\h_\D$ the polynomial algebra generated by $\T_\D$.
\item Let $t_1,\ldots,t_n$ be elements of $\T_\D$ and let $d\in \D$. We denote by $B^+_d(t_1\ldots t_n)$ the rooted tree obtained by grafting 
$t_1,\ldots,t_n$ on a common root decorated by $d$. This map $B^+_d$ is extended in an operator from $\h_\D$ to $\h_\D$.
\end{enumerate}}\end{defi}

For example, $B^+_d(\tddeux{$a$}{$b$}\tdun{$c$})=\tdquatredeux{$d$}{$c$}{$a$}{$b$}$.

\subsection{Hopf algebras of decorated rooted trees}

In order to make $\h_\D$ a bialgebra, we now introduce the notion of cut of a tree $t\in \T_\D$. A {\it non-total cut} $c$ 
of a tree $t$ is a choice of edges of $t$. Deleting the chosen edges, the cut makes $t$ into a forest denoted by $W^c(t)$. The cut $c$ is {\it admissible} 
if any oriented path in the tree meets at most one cut edge. For such a cut, the tree of $W^c(t)$ which contains the root of $t$ is denoted by $R^c(t)$ 
and the product of the other trees of $W^c(t)$ is denoted by $P^c(t)$. We also add the total cut, which is by convention an admissible cut 
such that $R^c(t)=1$ and $P^c(t)=W^c(t)=t$. The set of admissible cuts of $t$ is denoted by $Adm_*(t)$. 
Note that the empty cut of $t$ is admissible; we put $Adm(t)=Adm_*(t)-\{\mbox{empty cut, total cut}\}$.

{\bf example}. Let $a,b,c,d \in \D$ and let us consider the rooted tree $t=\tdquatredeux{$d$}{$c$}{$b$}{$a$}$. As it as $3$ edges, it has $2^3$ non-total cuts.
$$\begin{array}{|c|c|c|c|c|c|c|c|c|c|}
\hline \mbox{cut }c&\tdquatredeux{$d$}{$c$}{$b$}{$a$}&\tquatredeuxa&\tquatredeuxb&\tquatredeuxc
&\tquatredeuxd&\tquatredeuxe&\tquatredeuxf&\tquatredeuxg&\mbox{total}\\
\hline \mbox{Admissible?}&\mbox{yes}&\mbox{yes}&\mbox{yes}&\mbox{yes}&\mbox{no}&\mbox{yes}&\mbox{yes}&\mbox{no}&\mbox{yes}\\
\hline W^c(t)&\tdquatredeux{$d$}{$c$}{$b$}{$a$}&\tddeux{$b$}{$a$}\tddeux{$d$}{$c$}&\tdun{$a$}\tdtroisun{$d$}{$c$}{$b$}
&\tdtroisdeux{$d$}{$b$}{$a$}\tdun{$c$}&\tdun{$a$}\tdun{$b$}\tddeux{$d$}{$c$}&\tddeux{$b$}{$a$}\tdun{$c$}\tdun{$d$}
&\tdun{$a$}\tddeux{$d$}{$b$}\tdun{$c$}&\tdun{$a$}\tdun{$b$}\tdun{$c$}\tdun{$d$}&\tdquatredeux{$d$}{$c$}{$b$}{$a$}\\
\hline R^c(t)&\tdquatredeux{$d$}{$c$}{$b$}{$a$}&\tddeux{$d$}{$c$}&\tdtroisun{$d$}{$c$}{$b$}&\tdtroisdeux{$d$}{$b$}{$a$}
&\times&\tdun{$d$}&\tddeux{$d$}{$b$}&\times&1\\
\hline P^c(t)&1&\tddeux{$b$}{$a$}&\tdun{$a$}&\tdun{$c$}&\times&\tddeux{$b$}{$a$}\tdun{$c$}&\tdun{$a$}\tdun{$c$}
&\times&\tdquatredeux{$d$}{$c$}{$b$}{$a$}\\
\hline \end{array}$$

The coproduct of $\h_\D$ is defined as the unique algebra morphism from $\h_\D$ to $\h_\D \otimes \h_\D$ such that for all rooted tree $t \in \T_\D$:
$$\Delta(t)=\sum_{c\in Adm_*(t)} P^c(t)\otimes R^c(t)=t\otimes 1+1\otimes t+\sum_{c \in Adm(t)} P^c(t) \otimes R^c(t).$$
As $\h_\D$ is the free associative commutative unitary algebra generated by $\T_\D$, this makes sense.
This coproduct makes $\h_\D$ a Hopf algebra. Although it won't play any role in this text, we recall that the antipode $S$ is the unique algebra automorphism
of $\h_\D$ such that for all $t\in \T_\D$:
$$S(t)=-\sum_{\mbox{\scriptsize $c$ cut of $t$}} (-1)^{n_c} W_c(t),$$
where $n_c$ is the number of cut edges of $c$. \\

{\bf Example}.
$$\Delta(\tdquatredeux{$d$}{$c$}{$b$}{$a$})=\tdquatredeux{$d$}{$c$}{$b$}{$a$} \otimes 1+1\otimes \tdquatredeux{$d$}{$c$}{$b$}{$a$}
+\tddeux{$b$}{$a$} \otimes \tddeux{$d$}{$c$}+\tdun{$a$}\otimes \tdtroisun{$d$}{$c$}{$b$}+
\tdun{$c$} \otimes \tdtroisdeux{$d$}{$b$}{$a$} +\tddeux{$b$}{$a$}\tdun{$c$}\otimes \tdun{$d$}+\tdun{$a$}\tdun{$c$}\otimes \tddeux{$d$}{$b$}.$$

A study of admissible cuts shows the following result:

\begin{prop}
For all $d\in \D$, for all $x \in \h_\D$:
$$\Delta \circ B^+_d(x)=B_d^+(x) \otimes 1+(Id \otimes B_d^+)\circ \Delta(x).$$
\end{prop}

{\bf Remarks}. \begin{enumerate}
\item In other words, $B_d^+$ is a $1$-cocycle for a certain cohomology of coalgebras, see \cite{Connes}.
\item If $t\in \T_\D^{(i)}$, then $\Delta(t)-t\otimes 1\in \h_\D\otimes \T_\D^{(i)}$.
\end{enumerate}

\subsection{Gradation of $\h_\D$ and completion}

We grade $\h_\D$ by declaring the forests with $n$ vertices homogeneous of degree $n$.
We denote by $\h_\D(n)$ the homogeneous component of $\h_\D$ of degree $n$. Then $\h_\D$ is a graded bialgebra, that is to say:
\begin{itemize}
\item For all $i,j \in \mathbb{N}$, $\h_\D(i)\h(j)\subseteq \h_\D(i+j)$.
\item For all $k\in \mathbb{N}$, $\displaystyle \Delta(\h_\D(k))\subseteq \sum_{i+j=k}\h_\D(i) \otimes \h_\D(j)$.
\end{itemize}

We define, for all $x \in \h_\D$:
$$val(x)=\displaystyle \max\left\{n\in \mathbb{N}\:|\: x\in \bigoplus_{k\geq n} \h_\D(k)\right\}.$$
We then put, for all $x,y\in \h_\D$, $d(x,y)=2^{-val(x-y)}$, with the convention $2^{-\infty}=0$. Then $d$ is a distance on $\h_\D$. 
The metric space $(\h_\D,d)$ is not complete; its completion will be denoted by $\widehat{\h_\D}$. As a vector space:
$$\widehat{\h_\D}=\prod_{n\in\mathbb{N}} \h_\D(n).$$
The elements of $\widehat{\h_\D}$ will be denoted by $\sum x_n$, where $x_n \in \h_\D(n)$ for all $n\in \mathbb{N}$.
The product $m:\h_\D\otimes \h_\D \longrightarrow \h_\D$ is homogeneous of degree $0$, so is continuous: it can be extended from 
$\widehat{\h_\D}\otimes \widehat{\h_\D}$ to $\widehat{\h_\D}$, which is then an associative, commutative algebra. 
Similarly, the coproduct of $\h_\D$ can be extended as a map:
$$\Delta:\widehat{\h_\D} \longrightarrow \h_\D \widehat{\otimes}\h_\D=\prod_{i,j\in \mathbb{N}} \h_\D(i)\otimes \h_\D(j).$$

Let $f(h)=\sum p_n h^n \in K[[h]]$ be any formal series, and let $X=\sum x_n \in \widehat{\h_\D}$, such that $x_0=0$.
The series of $\widehat{\h_\D}$ of terms $p_n X^n$ is Cauchy, so converges. Its limit will be denoted by $f(X)$. In other words, $f(X)=\sum y_n$, with:
$$\left\{\begin{array}{rcl}
y_0&=&p_0,\\
y_n&=&\displaystyle \sum_{k=1}^n \sum_{a_1+\cdots+a_k=n} p_k x_{a_1}\cdots x_{a_k} \mbox{ if }n\geq 1.
\end{array}\right.$$

\subsection{Pre-Lie structure on the dual of $\h_\D$}

By the Cartier-Quillen-Milnor-Moore theorem \cite{Milnor}, the graded dual $\h_\D^*$ of $\h_\D$ is an enveloping algebra. Its Lie algebra $Prim(\h_\D^*)$ 
has a basis $(f_t)_{t\in \T_\D}$ indexed by $\T_D$:
$$ f_t : \left\{ \begin{array}{rcl}
\h_\D&\longrightarrow & K\\
t_1\ldots t_n &\longrightarrow & 
\left\{ \begin{array}{l}
0 \mbox{ if }n \neq 1,\\
\delta_{t,t_1}\mbox{ if }n=1.
\end{array}\right.\end{array}\right.$$

Recall that a pre-Lie algebra (or equivalently a Vinberg algebra or a left-symmetric algebra) is a couple $(A,\star)$, where $\star$ is a bilinear product on $A$
such that for all $x,y,z \in A$:
$$(x \star y) \star z-x \star (y \star z)=(y \star x) \star z-y \star (x \star z).$$
Pre-Lie algebras are Lie algebras, with bracket given by $[x,y]=x \star y-y\star x$. \\

The Lie bracket of $Prim(\h_\D^*)$ is induced by a pre-Lie product $\star $ given in the following way: if $f,g \in Prim(\h_\D^*)$, $f \star g$ is 
the unique element of $Prim(\h_\D^*)$ such that for all $t \in \T_\D$,
$$(f\star g)(t)=(f \otimes g) \circ (\pi \otimes \pi)\circ \Delta(t),$$
where $\pi$ is the projection on $Vect(\T^{\D})$ which vanishes on the forests which are not trees.
In other words, if $t,t' \in \T_\D$:
$$f_t \star f_{t'}=\sum_{t'' \in \T_\D} n(t,t';t'') f_{t''},$$
where $n(t,t';t')$ is the number of admissible cuts $c$ of $t''$ such that $P^c(t'')=t$ and $R^c(t'')=t'$.
It is proved that $(prim(\h_\D^*), \star)$ is the free pre-Lie algebra generated by the $\tdun{$d$}$'s, $d \in \D$: see \cite{Chapoton1,Chapoton3}.
Note that $\h_\D^*$ is isomorphic to the Grossman-Larson Hopf algebra of rooted trees \cite{Grossman1,Grossman2,Grossman3}.

\section{Definitions and properties of SDSE}

\subsection{Unique solution of an SDSE}

\begin{defi}\textnormal{Let $I$ be a finite, non-empty set, and let $F_i \in K[[h_j,j\in I]]$ be a non-constant formal series for all $i\in I$.
The {\it system of Dyson-Schwinger combinatorial equations} (briefly, the SDSE) associated to $(F_i)_{i\in I}$ is:
$$\forall i\in I,\: X_i=B^+_i(f_i(X_j,j\in I)),$$
where $X_i \in \widehat{\h_I}$ for all $i\in I$.
}\end{defi}

In order to ease the notation, we shall often assume that $I=\{1,\ldots,N\}$ in the proofs, without loss of generality. \\

{\bf Notations.} We assume here that $I=\{1,\ldots,N\}$. 
\begin{enumerate}
\item Let $(S)$ be an SDSE. We shall denote, for all $i\in I$:
$$F_i=\sum_{p_1,\cdots,p_N}a_{(p_1,\cdots,p_N)}^{(i)} h_1^{p_1}\cdots h_N^{p_N}.$$
\item Let $1\leq j \leq N$. We put $\varepsilon_j=(0,\cdots,0,1,0,\cdots,0)$ where the $1$ is in position $j$. We shall denote, for all $i\in I$,
$a^{(i)}_j=a^{(i)}_{\varepsilon_j}$; for all $j,k\in I$, $a^{(i)}_{j,k}=a^{(i)}_{\varepsilon_j+\varepsilon_k}$, and so on.
\end{enumerate}

{\bf Remark.} We assume that there is no constant $F_i$. Indeed, if $F_i \in K$, then $X_i$ is a multiple of $\tdun{$i$}$.
We shall always avoid this degenerated case in all this text.

\begin{prop}
Let $(S)$ be an SDSE. Then it admits a unique solution $(X_i)_{i\in I} \in \left(\widehat{\h_I}\right)^I$.
\end{prop}

\begin{proof} We assume here that $I=\{1,\ldots,N\}$. If $(X_1,\cdots,X_N)$ is a solution of $S$, then $X_i$ is a linear (infinite) span of rooted trees 
with a root decorated by $i$. We denote:
$$X_i=\sum_{t\in \T_I^{(i)}} a_t t.$$
These coefficients are uniquely determined by the following formulas: if
$$t=B_i^+\left(t_{1,1}^{p_{1,1}}\cdots t_{1,q_1}^{p_{1,q_1}}\cdots t_{N,1}^{p_{N,1}}\cdots t_{N,q_N}^{p_{N,q_N}}\right),$$
where the $t_{i,j}$'s are different trees, such that the root of $t_{i,j}$ is decorated by $i$ for all $i\in I$,  $1\leq j \leq q_i$, then:
\begin{equation}\label{E1}
a_t=\left(\prod_{i=1}^N \frac{(p_{i,1}+\cdots+p_{i,q_i})!}{p_{i,1}!\cdots p_{i,q_i}!}\right)
a^{(i)}_{(p_{1,1}+\cdots+p_{1,q_1},\cdots,p_{N,1}+\cdots+p_{N,q_N})} a_{t_{1,1}}^{p_{1,1}}\cdots a_{t_{N,q_N}}^{p_{N,q_N}}.
\end{equation}
So $(S)$ has a unique solution. \end{proof}

\begin{defi}\textnormal{
Let $(S)$ be an SDSE and let $X=(X_i)_{i\in I}$ be its unique solution. The subalgebra of $\h_I$ generated by the homogeneous components 
$X_i(k)$'s of the $X_i$'s will be denoted by $\hs$. If $\hs$ is Hopf, the system $(S)$ will be said to be Hopf.
}\end{defi}

\subsection{Graph associated to an SDSE}

We associate a oriented graph to each SDSE in the following way:

\begin{defi}\textnormal{
Let $(S)$ be an SDSE. 
\begin{enumerate}
\item We construct an oriented graph $\gs$ associated to $(S)$ in the following way:
\begin{itemize}
\item The vertices of $\gs$ are the elements of $I$.
\item There is an edge from $i$ to $j$ if, and only if, $\displaystyle \frac{\partial F_i}{\partial h_j} \neq 0$.
\end{itemize}
\item If $\displaystyle \frac{\partial F_i}{\partial h_i} \neq 0$, the vertex $i$ will be said to be {\it self-dependent}.
In other words, if $i$ is self-dependent, there is a loop from $i$ to itself in $\gs$.
\item If $\gs$ is connected, we shall say that $(S)$ is {\it connected}.
\end{enumerate}}\end{defi}

{\bf Remark.} If $(S)$ is not connected, then $(S)$ is the union of SDSE $(S_1)$, $\cdots$, $(S_k)$ with disjoint sets of indeterminates , 
so $\hs\approx \h_{(S_1)}\otimes \cdots \otimes \h_{(S_k)}$. As a corollary, $(S)$ is Hopf if, and only if, for all $j$, $(S_j)$ is Hopf.\\

Let $(S)$ be an SDSE and let $\gs$ be the associated graph. Let $i$ and $j$ be two vertices of $\gs$. We shall say that $j$ is a direct
descendant of $i$ (or $i$ is a direct ascendant of $j$) if there is an oriented edge from $i$ to $j$; we shall say that $j$ is a descendant
of $i$ (or $i$ is an ascendant of $j$) if there is an oriented path from $i$ to $j$. We shall write "$i \longrightarrow j$" for "$j$ is a direct descendant of $i$".

\subsection{Operations on Hopf SDSE}

\begin{prop}[change of variables]
Let $(S)$ be the SDSE associated to $(F_i(h_j, \:j\in I))_{i\in I}$. Let $\lambda_i$ and $\mu_i$ be non-zero scalars for all $i\in I$. 
The system $(S)$ is Hopf if, and only if, the SDSE system  $(S')$ associated to $(\mu_i F_i(\lambda_j h_j,\: j\in J))_{i\in I}$ is Hopf.
\end{prop}

\begin{proof} We assume that $I=\{1,\ldots,N\}$. We consider the following morphism:
$$\phi : \left\{ \begin{array}{rcl}
\h_I&\longrightarrow & \h_I\\
F\in \F&\longrightarrow &(\mu_1\lambda_1)^{n_1(F)} \cdots (\mu_N\lambda_N)^{n_N(F)} F,
\end{array}\right.$$
where $n_i(F)$ is the number of vertices of $F$ decorated by $i$. Then $\phi$ is a Hopf algebra automorphism and for all $i$, 
$\phi \circ B^+_i=\mu_i\lambda_i B^+_i \circ \phi$. Moreover, if we put $Y_i= \frac{1}{\lambda_i}\phi(X_i)$ for all $i$:
\begin{eqnarray*}
Y_i&=&\frac{1}{\lambda_i} \phi \circ B^+_i(F_i(X_1,\cdots,X_N))\\
&=&\frac{1}{\lambda_i} \mu_i \lambda_i B_i^+(F_i(\phi(X_1),\cdots, \phi(X_N)))\\
&=&\mu_i B_i^+(F_i(\lambda_1 Y_1,\cdots,\lambda_N Y_N)).
\end{eqnarray*}
So $(Y_1,\cdots, Y_N)$ is the solution of the system $(S')$. Moreover, $\phi$ sends $\hs$ onto $\h_{(S')}$.
As $\phi$ is a Hopf algebra automorphism, $\h_{(S)}$ is a Hopf subalgebra of $\h_I$ if, and only if, $\h_{(S')}$ is. \end{proof}\\

{\bf Remark.} A change of variables does not change the graph associated to $(S)$.

\begin{prop}[restriction]
Let $(S)$ be the SDSE associated to $(F_i(h_j,\:j\in I))_{i\in I}$ and let $I' \subseteq I$, non-empty. 
Let $(S')$ be the SDSE associated to $\left(F_i(h_j,j\in I)_{\mid h_j=0, \: \forall j\notin I'}\right)_{i\in I'}$. If $(S)$ is Hopf, then $(S')$ also is.
\end{prop}

\begin{proof} We consider the epimorphism $\phi$ of Hopf algebras
from $\h_I$ to $\h_{I'}$, obtained by sending the forests with at least a vertex decorated by an element which is not in $I'$ to zero. 
Then $\phi$ sends $\hs$ to $\h_{(S')}$. As $\phi$ is a morphism of Hopf algebras, if $\h_{(S)}$ is a Hopf subalgebra of $\h_I$,
$\h_{(S')}$ is a Hopf subalgebra of $\h_{I'}$. \end{proof} \\

{\bf Remark.} The restriction to a subset of vertices $I'$ changes $\gs$ into the graph obtained by deleting all the vertices $j\notin I'$ and 
all the edges related to these vertices.

\begin{prop}[dilatation]
Let $(S)$ be the system associated to $(F_i)_{i\in I}$ and $(S')$ be a system associated to a family $(F'_j)_{j\in J}$, such that there exists a partition
$J=\displaystyle \bigcup_{i\in I} J_i$, with the following property: for all $i\in I$, for all $x \in I_i$,
$$F'_x=F_i\left(\sum_{y\in I_j} h_y, \: j\in I \right).$$
Then $(S)$ is Hopf, if, and only if, $(S')$ is Hopf. We shall say that $(S')$ is a dilatation of $(S)$.
\end{prop}

\begin{proof} We assume here that $I=\{1,\ldots,N\}$.

$\Longrightarrow$. Let us assume that $(S)$ is Hopf. For all $i\in I$, we can then write:
$$\Delta(X_i)=\sum_{n\geq 0} P_n^{(i)}(X_1,\cdots,X_N) \otimes X_i(n),$$
with the convention $X_i(0)=1$.
Let $\phi:\h_I \longrightarrow \h_{I'}$ be the morphism of Hopf algebras such that, for all $1\leq i \leq N$:
$$\phi \circ B^+_i= \sum_{j\in I_i} B^+_j \circ \phi.$$
Then, immediately, for all $1\leq i \leq N$:
$$\phi(X_i)=\sum_{j\in I_i} X_j'.$$
As a consequence:
$$\sum_{j\in I_i} \Delta(X'_j)=\sum_{j\in I_i} \sum_{n\geq 0} P_n^{(i)} \left(\sum_{k\in I_1} X'_k,\cdots, \sum_{k\in I_N} X'_k\right) \otimes X'_j(n).$$
Conserving the terms of the form $F \otimes t$, where $t$ is a tree with root decorated by $j$, for all $j \in I_i$:
$$\Delta(X'_j)=\sum_{n\geq 0} P_n^{(i)} \left(\sum_{k\in I_1} X'_k,\cdots, \sum_{k\in I_N} X'_k\right) \otimes X'_j(n).$$
So $(S')$ is Hopf.\\

$\Longleftarrow$. By restriction, choosing an element in each $I_i$, if $(S')$ is Hopf, then $(S)$ is Hopf. \end{proof}\\

{\bf Remark.} If $(S')$ is a dilatation of $(S)$, then the set of vertices $J$ of the graph $G_{(S')}$ associated to $(S')$  admits a partition 
indexed by the vertices of $\gs$, and there is an edge from $x \in J_i$ to $y \in J_j$ in $G_{(S')}$ if, and only if, there is an edge from $i$ to $j$ in $\gs$. \\

{\bf Example.} Let $f,g \in K[[h_1,h_2]]$. Let us consider the following SDSE:
\begin{eqnarray*}
&(S):& \left\{ \begin{array}{rcl}
X_1&=&B^+_1(f(X_1,X_2)),\\
X_2&=&B^+_2(g(X_1,X_2)),
\end{array}\right.\\ \\
&(S'):&\left\{ \begin{array}{rcl}
X_1&=&B^+_1(f(X_1+X_2+X_3,X_4+X_5)),\\
X_2&=&B^+_2(f(X_1+X_2+X_3,X_4+X_5)),\\
X_3&=&B^+_3(f(X_1+X_2+X_3,X_4+X_5)),\\
X_4&=&B^+_4(g(X_1+X_2+X_3,X_4+X_5)),\\
X_5&=&B^+_5(g(X_1+X_2+X_3,X_4+X_5)).
\end{array}\right. \end{eqnarray*}
Then $(S')$ is a dilatation of $(S)$.

\begin{prop}[extension] \label{11}
Let $(S)$ be the SDSE associated to $(F_i)_{i\in I}$. Let $0\notin I$ and let $(S')$ be associated to $(F_i)_{i\in I\cup\{0\}}$, with:
$$F_0=1+\sum_{i\in I} a^{(0)}_i h_i.$$
Then $(S')$ is Hopf if, and only if, the two following conditions hold:
\begin{enumerate}
\item $(S)$ is Hopf.
\item For all $i,j \in I^{(0)}=\left\{j\in I\:/\:a^{(0)}_j\neq 0\right\}$, $F_i=F_j$.
\end{enumerate}
If these two conditions hold, we shall say that $(S')$ is an extension of $(S)$.
 \end{prop}

\begin{proof} We assume here that $I=\{1,\ldots,N\}$.

$\Longrightarrow$. Let us assume that $(S')$ is Hopf. By restriction, $(S)$ is Hopf. Moreover:
$$X_0=B^+_0\left(1+\sum_{i=1}^N a^{(0)}_i X_i\right)=\tdun{$0$}+\sum_{i=1}^Na^{(0)}_i B^+_0 \circ B^+_i(f_i(X_1,\cdots,X_N)).$$
As $\h_{(S')}$ is a graded Hopf subalgebra, the projection on $\h_{\{0,\cdots,N\}} \otimes \h_{\{0,\cdots,N\}}(2)$ gives:
$$\sum_{i=1}^N a^{(0)}_i F_i(X_1,\cdots,X_N)\otimes \tddeux{$0$}{$i$}\in \h_{(S')} \widehat{\otimes} \h _{(S')}.$$
So this is of the form:
$$P \otimes X_0(2)=P\otimes \left( \sum_{i=1}^N a^{(0)}_i \tddeux{$0$}{$i$} \right),$$
for a certain $P\in \widehat{\h_{(S')}}$. As the $\tddeux{$0$}{$i$}$'s, $i\in I$, are linearly independent, we obtain that for all $i,j$,
$a^{(0)}_i F_i(X_1,\cdots,X_N)=a^{(0)}_i P$ for all $i$, and this implies the second item.\\

$\Longleftarrow$. As $(S)$ is Hopf, we can put for all $1\leq i \leq N$:
$$\Delta(X_i)=X_i\otimes 1+\sum_{k=1}^{+\infty}P_k^{(i)} \otimes X_i(k),$$
where $P_n^{(i)}$ is an element of the completion of $\hs$. By the second hypothesis, if $i,j \in I$, as $F_i=F_j$, $P_n^{(i)}=P_n^{(j)}$.
We then denote by $P_n$ the common value of $P_n^{(i)}$ for all $i\in I$. So:
\begin{eqnarray*}
\Delta(X_0)&=&\tdun{$0$} \otimes 1+1\otimes \tdun{$0$}+\sum_{i=1}^N a_i^{(0)} \Delta \circ B^+_0(X_i)\\
&=&X_0\otimes 1+1\otimes X_0+\sum_{i=1}^N a^{(0)}_i (1+X_i) \otimes \tdun{$0$}
+\sum_{i=1}^N\sum_{j=1}^{\infty} a_i^{(0)} P_j^{(i)} \otimes B^+_0(X_i(j))\\
&=&X_0\otimes 1+1\otimes X_0+\sum_{i=1}^N a^{(0)}_i (1+X_i) \otimes \tdun{$0$}
+\sum_{i=1}^N\sum_{j=1}^{\infty} a_i^{(0)} P_j \otimes B^+_0(X_i(j))\\
&=&X_0\otimes 1+1\otimes X_0+\sum_{i=1}^N a^{(0)}_i (1+X_i) \otimes \tdun{$0$}
+\sum_{i=1}^NP_j \otimes B^+_0\left(\sum_{j=1}^{\infty} a_i^{(0)} X_i(j)\right)\\
&=&X_0\otimes 1+1\otimes X_0+\sum_{i=1}^N a^{(0)}_i (1+X_i) \otimes \tdun{$0$}+\sum_{i=1}^NP_j \otimes X_0(j+1).
\end{eqnarray*}
This belongs to the completion of $\h_{(S')} \otimes \h_{(S')}$, so $(S')$ is Hopf. \end{proof}\\

{\bf Remarks.} \begin{enumerate}
\item If $(S)$ is an extension of $(S')$, then $\gs$ is obtained from $G_{(S')}$ by adding a non-self-dependent vertex with no ascendant.
\item If $I^{(0)}$ is reduced to a single element, then condition 2 is empty.
\end{enumerate}

\begin{defi}
\textnormal{Let $(S)$ a Hopf SDSE and let $i \in I$. We shall say that $i$ is an {\it extension vertex}
if, denoting by $J$ the set of descendants of $i$, the restriction of $(S)$ to $J\cup\{i\}$ is an extension of the restriction of $(S)$ to $J$.}
\end{defi}

\subsection{Constant terms of the formal series}

\begin{lemma}
Let $(S)$ be an Hopf SDSE. If $F_i(0,\cdots,0)=0$, then $X_i=0$.
\end{lemma}

\begin{proof}
If $F_i(0,\cdots,0)=0$, then the homogeneous component of degree $1$ of $X_i$ is zero, so $\tdun{$i$} \notin \hs$. 
Considering the terms of the form $F \otimes \tdun{$i$}$ in $\Delta(X_i)$, we obtain:
$$F_i(X_j,\:j\in I) \otimes \tdun{$i$} \in \hs \otimes \hs.$$
As $\tdun{$i$} \notin \hs$, necessarily $F_i(X_j,\:j\in I)=0$, so $X_i=0$.
\end{proof}\\

As a consequence, if $F_i(0,\cdots,0)=0$, then $\hs=\h_{(S')}$, where $(S')$ is the restriction of $(S)$ to $I-\{i\}$. Using a change of variables, 
we shall always suppose in the sequel that for all  $i$, $F_i(0,\cdots,0)=1$.

\subsection{Main theorem}

{\bf Notations.} For all $\beta \in K$, we put:
$$f_\beta(h)=\sum_{k=0}^{+\infty} \frac{(1+\beta)\cdots (1+\beta(k-1))}{k!}h^k=
 \left\{ \begin{array}{l}
(1-\beta h)^{-\frac{1}{\beta}} \mbox{ if } \beta\neq 0,\\
e^h \mbox{ if }\beta=0.
\end{array}\right.$$

The main aim of this text is to prove the following result:

\begin{theo} \label{14}
Let $(S)$ be a connected SDSE. It is Hopf if and only if one of the following assertion holds:
\begin{enumerate}
\item (Extended multicyclic SDSE). The set $I$ admits a partition $I=I_{\overline{1}}\cup\cdots \cup I_{\overline{N}}$ indexed by the elements
of $\mathbb{Z}/N\mathbb{Z}$, $N\geq 2$, with the following conditions:
\begin{itemize}
\item For all $i\in I_{\overline{k}}$:
$$F_i=1+\sum_{j\in I_{\overline{k+1}}} a^{(i)}_j h_j.$$
\item If $i$ and $i'$ have a common direct ascendant in $\gs$, then $F_i=F_{i'}$ (so $i$ and $i'$ have the same direct descendants).
\end{itemize}
\item (Extended fundamental SDSE). There exists a partition:
$$I=\left(\bigcup_{i \in I_0} J_i\right) \cup \left(\bigcup_{i \in J_0} J_i\right) \cup K_0 \cup I_1 \cup J_1 \cup I_2,$$
with the following conditions:
\begin{itemize}
\item $K_0$, $I_1$, $J_1$, $I_2$ can be empty.
\item The set of indices $I_0 \cup J_0$ is not empty.
\item For all $i \in I_0 \cup J_0$, $J_i$ is not empty.
\end{itemize} 
Up to a change of variables:
\begin{enumerate}
\item For all $i \in I_0$, there exists $\beta_i \in K$, such that for all $x \in J_i$:
$$F_x=f_{\beta_i}\left(\sum_{y\in J_i} h_y\right) 
\prod_{j\in I_0-\{i\}} f_{\frac{\beta_j}{1+\beta_j}}\left((1+\beta_j) \sum_{y\in J_j} h_y\right) \prod_{j\in J_0} f_1\left(\sum_{y\in J_j} h_y\right).$$
\item For all $i\in J_0$, for all $x \in J_i$:
$$F_x=\prod_{j\in I_0} f_{\frac{\beta_j}{1+\beta_j}}\left((1+\beta_j) \sum_{y\in J_j} h_y\right) \prod_{j\in J_0-\{i\}}  f_1\left(\sum_{y\in J_j} h_y\right).$$
\item For all $i\in K_0$:
$$F_i=\prod_{j\in I_0}f_{\frac{\beta_j}{1+\beta_j}}\left((1+\beta_j) \sum_{y\in J_j} h_y\right)\prod_{j\in J_0}f_1\left(\sum_{y\in J_j} h_y\right).$$
\item For all $i\in I_1$, there exist $\nu_i\in K$ and a family of scalars $\left(a_j^{(i)}\right)_{j\in I_0\cup J_0 \cup K_0}$, with 
$(\nu_i \neq 1)$ or $(\exists j\in I_0, \: a^{(i)}_j \neq 1+\beta_j)$ or $(\exists j\in J_0, \: a^{(i)}_j \neq 1)$ or $(\exists j\in K_0, \: a^{(i)}_j \neq 0)$.
Then, if $\nu_i \neq 0$:
$$F_i=\frac{1}{\nu_i} \prod_{j\in I_0} f_{\frac{\beta_j}{\nu_i a^{(i)}_j}}\left(\nu_i a^{(i)}_j\sum_{y\in J_j} h_y\right)
\prod_{j\in J_0} f_{\frac{1}{\nu_i a^{(i)}_j}}\left(\nu_i a^{(i)}_j\sum_{y\in J_j} h_y\right) \prod_{j\in K_0} f_0\left(\nu_i a^{(i)}_jh_j\right)+1-\frac{1}{\nu_i}.$$
If $\nu_i=0$:
$$F_i=-\sum_{j\in I_0} \frac{a^{(i)}_j}{\beta_j} \ln\left(1-\sum_{y\in J_j} h_y\right)-\sum_{j\in J_0} a^{(i)}_j \ln\left(1-\sum_{y\in J_j} h_y\right)
+\sum_{j\in K_0} a^{(i)}_j h_j+1.$$
\item For all $i \in J_1$, there exists $\nu_i\in K-\{0\}$ and a family of scalars $\left(a_j^{(i)}\right)_{j\in I_0\cup J_0 \cup K_0\cup I_1}$, 
with the three following conditions:
\begin{itemize}
\item $I_1^{(i)}=\{j\in I_1\:/\:a^{(i)}_j \neq 0\}$ is not empty.
\item For all $j\in I_1^{(i)}$, $\nu_j=1$.
\item For all $j,k \in I_1^{(i)}$, $F_j=F_k$. In particular, we put $b^{(i)}_t=a^{(j)}_t$ for any $j\in I_1^{(i)}$, for all $t\in I_0 \cup J_0 \cup K_0$.
\end{itemize}
Then:
\begin{eqnarray*}
F_i&=&\frac{1}{\nu_i} \prod_{j\in I_0} f_{\frac{\beta_j}{b^{(i)}_j-1-\beta_j}}\left(\left(b^{(i)}_j-1-\beta_j\right)\sum_{y\in J_j} h_y\right)
\prod_{j\in J_0} f_{\frac{\beta_j}{b^{(i)}_j-1}}\left(\left(b^{(i)}_j-1\right)\sum_{y\in J_j} h_y\right)\\
&&\prod_{j\in K_0} f_0\left(b^{(i)}_jh_j\right)+\sum_{j\in I_1^{(i)}} a^{(i)}_j h_1+1-\frac{1}{\nu_i}.
\end{eqnarray*}
\item $I_2=\{x_1,\ldots,x_m\}$ and for all $1\leq k \leq m$, there exist a set:
$$I^{(x_k)}\subseteq \left(\bigcup_{i \in I_0} J_i\right) \cup \left(\bigcup_{i \in J_0} J_i\right) \cup K_0 \cup I_1 \cup J_1 \cup \{x_1,\ldots,x_{k-1}\}$$
and a family of non-zero scalars $\left(a^{(x_k)}_j\right)_{j\in I^{(x_k)}}$ such that for all $i,j \in I^{(x_k)}$, $F_i=F_j$. Then:
$$F_{x_k}=1+\sum_{j\in I^{(x_k)}} a^{(x_k)}_j h_j.$$
\end{enumerate} \end{enumerate} \end{theo}

Here is the graph of a system of an extended multicyclic SDSE, with $N=5$. The different subset of the partition are indicated by the different colours.
the multicycle corresponds to the five boxes. An arrow between two boxes means that all vertices of the boxes are related by an arrow.
$$\includegraphics[height=7cm]{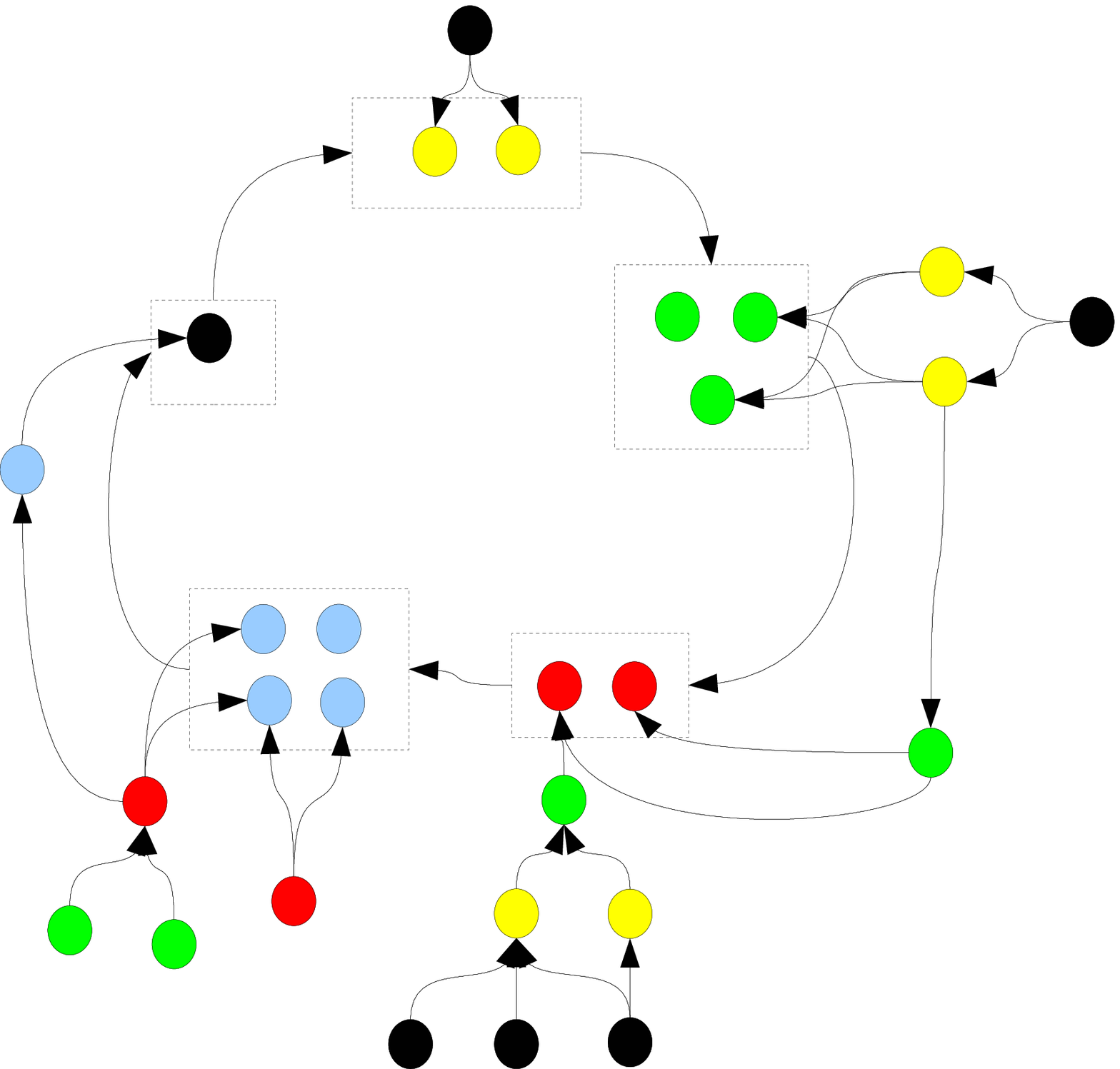}$$

Here is the graph of an extended fundamental SDSE. The vertices in $J_i$, with $i \in I_0$, are green. There are two elements in $I_0$, 
one with $\beta_i=-1$ (light green vertices) and one with $\beta_i \neq -1$ (dark green vertex). There are two elements in $J_0$, corresponding to 
light blue and dark blue vertices. The unique element of $K_0$ is red; the unique element of $I_1$ is yellow; the unique element of $J_1$ is orange;
the dark vertices are the elements of $I_2$. An arrow between two boxes means that all vertices of the boxes are related by an arrow.
$$\includegraphics[height=7.5cm]{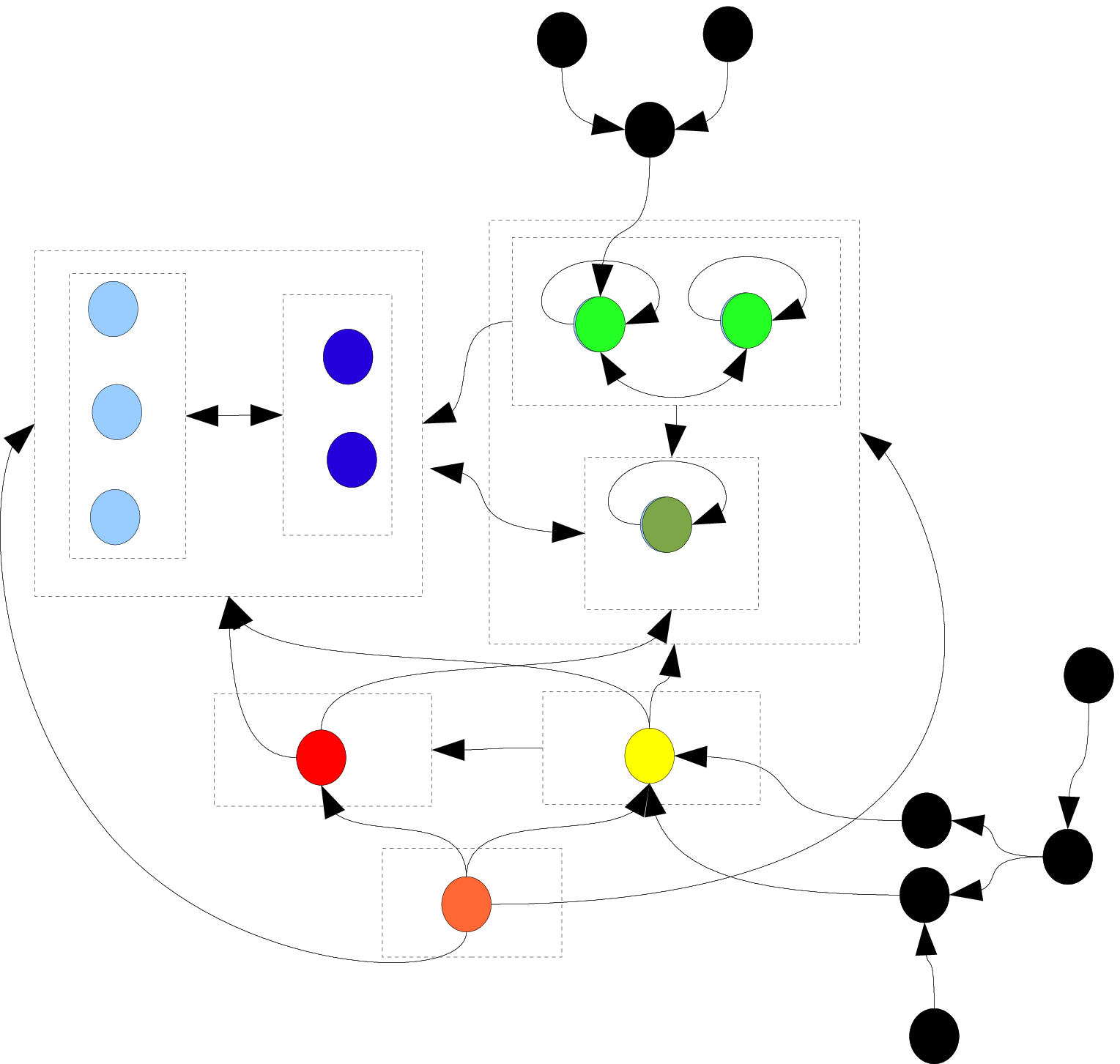}$$
For example, the SDSE associated to the following formal series has such a graph:
\begin{eqnarray*}
F_1&=&f_\beta(h_1)f_1(h_4+h_5)f_1(h_6+h_7+h_8)\\
F_2=F_3&=&(1+h_2+h_3)f_{\frac{\beta}{1+\beta}}((1+\beta)h_1)f_1(h_4+h_5)f_1(h_6+h_7+h_8)\\
F_4=F_5&=&f_{\frac{\beta}{1+\beta}}((1+\beta)h_1)f_1(h_6+h_7+h_8)\\
F_6=F_7=F_8&=&f_{\frac{\beta}{1+\beta}}((1+\beta)h_1)f_1(h_4+h_5)\\
F_9&=&f_{\frac{\beta}{1+\beta}}((1+\beta)h_1)f_1(h_4+h_5)f_1(h_6+h_7+h_8)\\
F_{10}&=&\frac{1}{\nu}f_{\frac{\beta}{\nu a^{(10)}_1}}\left(\nu a^{(10)}_1 h_1\right)f_{\frac{-1}{\nu a^{(10)}_2}}\left(\nu a^{(10)}_2 (h_2+h_3)\right)
f_{\frac{1}{\nu a^{(10)}_4}}\left(\nu a^{(10)}_4(h_4+h_5)\right)\\
&&f_{\frac{1}{\nu a^{(10)}_6}}\left(\nu a^{(10)}_6(h_6+h_7+h_8)\right)f_0\left(\nu a^{(10)}_9h_9\right)+1-\frac{1}{\nu},\\
F_{11}&=&\frac{1}{\nu'}f_{\frac{\beta}{a^{(10)}_1-1-\beta}}\left(\left(a^{(10)}_1-1-\beta \right) h_1\right)f_{\frac{-1}{a^{(10)}_2}}\left(a^{(10)}_2 (h_2+h_3)\right)\\
&&f_{\frac{1}{a^{(10)}_4-1}}\left(\left(a^{(10)}_4-1\right)(h_4+h_5)\right)f_{\frac{1}{a^{(10)}_6-1}}\left(\left(a^{(10)}_6-1\right)(h_6+h_7+h_8)\right)\\
&&f_0\left(a^{(10)}_9h_9\right)+a^{(11)}_{10} h_{10}+1-\frac{1}{\nu'},\\
F_{12}=F_{13}&=&1+a^{(12)}_{10} h_{10},\\
F_{14}&=&1+a^{(14)}_{13} h_{13},\\
F_{15}&=&1+a^{(15)}_{12} h_{12}+a^{(15)}_{13}h_{13},\\
F_{16}&=&1+a^{(16)}_{15} h_{15},\\
F_{17}&=&1+a^{(17)}_{2} h_{2},\\
F_{18}&=&1+a^{(18)}_{17} h_{17},\\
F_{19}&=&1+a^{(19)}_{17} h_{17},
\end{eqnarray*}
where $\beta \neq -1$, $\nu,\nu'\neq 0$, and the coefficients $a^{(i)}_j$ are non-zero.

\section{Characterisation and properties of Hopf SDSE}

\subsection{Subalgebras of $\h_\D$ generated by spans of trees}

Let us fix a non-empty set $\D$.

\begin{lemma}\label{15}
Let $V$ be a subspace of $Vect(\T_\D)$ and let us consider the subalgebra $A$ of $\h_\D$ generated by $V$.
Recall that for all $d\in \D$, $f_{\tdun{$d$}}$ is the following linear map:
$$ f_{\tdun{$d$}}: \left\{ \begin{array}{rcl}
\h_\D&\longrightarrow & K\\
t_1\cdots t_n&\longrightarrow &\delta_{t_1\cdots t_n,\tdun{$d$}}.
\end{array}\right.$$
Then $A$ is a Hopf subalgebra if, and only if, the two following assertions are both satisfied:
\begin{enumerate}
\item For all $d\in \D$, $(f_{\tdun{$d$}} \otimes Id) \circ \Delta(V)\subseteq V+K$.
\item For all $d\in \D$, $(Id \otimes f_{\tdun{$d$}}) \circ \Delta(V) \subseteq A$.
\end{enumerate}
\end{lemma}

\begin{proof}
$\Longrightarrow$. If $A$ is Hopf, then $\Delta(V) \subseteq A\otimes A$.
As $V \subseteq Vect(\T_\D)$, $\Delta(V) \subseteq \h \otimes (Vect(\T_\D)+K)$. So:
$$\Delta(V)\subseteq (A\otimes A)\cap ( \h \otimes (Vect(\T_\D)+K))=A\otimes (V \oplus K).$$
This implies both assertions.\\

$\Longleftarrow$. We use here Sweedler's notations: $\Delta(a)=a' \otimes a''$ and $(\Delta \otimes Id)\circ \Delta(a)=a'\otimes a'' \otimes a'''$
for all $a \in A$.

{\it First step.} Let us consider the following subspace of $Prim(\h_\D^*)$:
$$B=\{f\in Prim(\h_\D^*)\:/\: (f\otimes Id) \circ \Delta(V) \subseteq V+K\}.$$
By hypothesis 1, $f_{\tdun{$d$}} \in B$ for all $d\in \D$. We recall here that $\star$ is the pre-Lie product of $Prim(\h_\D^*)$.
Let $f$ and $g\in B$. For all $v\in V$:
$$(f\star g \otimes Id) \circ \Delta(v)=f\circ \pi(v') g\circ \pi (v'') v'''.$$
As $f \in B$, $f\circ \pi(v') v'' \in V+K$. As $g \in B$, $f\circ \pi(v') g\circ \pi(v'') v''' \in V+K$. So $f\star g\in B$, and $B$ is a sub-pre-Lie algebra 
of $Prim(\h_\D^*)$. As $Prim(\h_\D^*)$ is generated as a pre-Lie algebra by the $f_{\tdun{$d$}}$'s, $B=Prim(\h_\D^*)$.\\

{\it Second step.} Let us consider the following subspace of $\h_\D^*$:
$$B'=\{f \in \h_\D^*\:/\: (f\otimes Id) \circ \Delta(A) \subseteq A\}.$$
Let $f \in Prim(\h_\D^*)$. By the first step, for all $v_1,\cdots,v_n \in V$:
$$(f\otimes Id)\circ \Delta(v_1\cdots v_n)=f(v'_1\cdots v'_n) v''_1\cdots v''_n=\sum_{i=1}^n v_1\cdots f(v'_i)v''_i \cdots v_n\in A,$$
so $Prim(\h_\D^*) \subseteq B'$. Let $f,g \in B'$. For all $a \in A$:
$$(fg \otimes Id)\circ \Delta(a)=f(a')g(a'')a'''.$$
As $f \in B'$, $f(a')a''\in A$. As $g\in B'$, $f(a')g(a'')a''' \in A$. So $B'$ is a subalgebra of $\h_\D^*$.
As it contains $Prim(\h_\D^*)$, it is equal to $\h_\D^*$. So:
$$\Delta(A)\subseteq \h_\D \otimes A+ \bigcap_{f\in \h_\D^*} Ker(f) \otimes \h_\D=\h_\D \otimes A.$$

{\it Third step.} Let us consider the following subspace of $Prim(\h_\D^*)$:
$$C=\{f\in Prim(\h_\D^*)\:/\: (Id \otimes f) \circ \Delta(V) \subseteq A\}.$$
By the second hypothesis, $f_{\tdun{$d$}} \in B$ for all $d\in \D$. Let us take $f$ and $g \in C$. For all $v \in V$:
$$(Id \otimes (f\star g))\circ \Delta(v)=v' f\circ \pi(v'') g\circ \pi(v''').$$
As $g \in C$, $v' g\circ \pi(v'') \in A$. Let us denote:
$$v'\circ \pi(v'')=\sum v_1\cdots v_n,$$
where $v_1,\ldots,v_n$ are elements of $V$. Then:
$$v' f\circ \pi(v'')g\circ \pi(v''')=\sum v_1'\cdots v_n' f\circ \pi(v''_1\cdots v''_n) g\circ \pi(v''').$$
By the second step, as $V \subseteq Vect(\T_\D)$:
$$\Delta(V) \subseteq (\h_\D \otimes A)\cap \left(\h_\D \otimes (Vect(\T_\D)+K)\right)=\h_\D\otimes (V+K).$$
So:
$$\sum v'_1\cdots v'_n \otimes \pi(v''_1\cdots v''_n)=\sum \sum_{i=1}^n v_1\cdots v'_i \cdots v_n \otimes \pi(v''_i).$$
Finally:
$$(Id \otimes (f\star g))\circ \Delta(v)=\sum \sum_{i=1}^n v_1\cdots v'_i \cdots v_n \otimes f\circ \pi(v''_i).$$
As $f \in B'$, this belongs to $A$. So $f\star g \in B'$. As at the end of the first step, we conclude that $B'=Prim(\h_\D^*)$.\\

{\it Last step.} As in the second step, we conclude that for all $f \in \h_\D^*$,
$(Id \otimes f) \circ \Delta(A)\subseteq A$. So $\Delta(A)\subseteq A\otimes \h_\D$, and
$\Delta(A) \subseteq (\h_\D \otimes A)\cap (A\otimes \h_\D)=A\otimes A$. So $A$ is a Hopf subalgebra.
\end{proof}

\subsection{Definition of the structure coefficients}

\begin{prop} \label{16}
Let $(S)$ be an SDSE. It is Hopf if, and only if, for all $i,j \in I$, for all $n\geq 1$, there exists a scalar $\lambda_n^{(i,j)}$ such that
for all $t'\in \T_i(n)$:
$$\sum_{t\in \T_i(n+1)} n_j(t,t') a_t=\lambda_n^{(i,j)} a_{t'},$$
where $n_j(t,t')$  is the number of leaves $l$ of $t$ decorated by $j$ such that the cut of $l$ gives $t'$.
\end{prop}

\begin{proof} $\Longrightarrow$. Let us assume that $(S)$ is Hopf.
Then $\hs$ is a Hopf subalgebra of $\h_I$. Let us use lemma \ref{15}, with $V=Vect(X_i(n),\:i\in I,\: n\geq 1)$.
So $(f_{\tdun{$j$}} \otimes Id) \circ \Delta(X_i(n+1))$ belongs to $\hs$, and is a linear span of trees of degree $n$ with a root decorated by $i$, 
so is a multiple of $X_i(n)$. We then denote:
$$(f_{\tdun{$j$}} \otimes Id) \circ \Delta(X_i(n+1))=\lambda_n^{(i,j)} X_i(n)=\sum_{t'\in \T(n)} \lambda_n^{(i,j)} a_{t'}t'.$$
By definition of the coproduct $\Delta$:
$$(f_{\tdun{$j$}} \otimes Id) \circ \Delta(X_i(n+1))=\sum_{t\in \T(n+1),\:t' \in \T(n)} n_j(t,t')a_t t'.$$
The result is proved by identifying the coefficients in the basis $\T(n)$ of these two expressions of $(f_{\tdun{$j$}} \otimes Id) \circ \Delta(X_i(n+1))$.\\

$\Longleftarrow$. Let us prove that both conditions of lemma \ref{15} are satisfied, with the same $V$ as before.
By hypothesis, for all $i,j \in I$, for all $n \geq 2$, $(f_{\tdun{$j$}} \otimes Id) \circ \Delta(X_i(n)) =\lambda_{n-1}^{(i,j)}X_i(n-1) \in V$.
Moreover, $(f_{\tdun{$j$}} \otimes Id) \circ \Delta(X_i(1))=\delta_{i,j} \in K$, so the first condition is satisfied. For the second one:
$$(Id \otimes f_{\tdun{$j$}}) \circ \Delta(X_i)=(Id \otimes f_{\tdun{$j$}}) \circ \Delta(B_i^+(F_i(X_j,\: j\in I)))=F_i(X_j,\:j\in I)\in \hs.$$
So $\hs$ is a Hopf subalgebra of $\h_I$. \end{proof}

\subsection{Properties of the coefficients $\lambda_n^{(i,j)}$}

The coefficients $\lambda_n^{(i,j)}$'s are entirely determined by the $a^{(i)}_j$'s and $a^{(i)}_{j,k}$'s, and determine the other coefficients
of the $F_i$'s, as shown by the following result:

\begin{lemma}\label{17}
Let us assume that $(S)$ is Hopf, with $I=\{1,\ldots,N\}$. Let us fix $i\in I$.
\begin{enumerate}
\item For all sequence $i=i_1\longrightarrow \cdots \longrightarrow i_n$ of vertices of $\gs$:
$$\lambda_n^{(i,j)}=a^{(i_n)}_j+\sum_{p=1}^{n-1}(1+\delta_{j,i_{p+1}}) \frac{a^{(i_p)}_{j,i_{p+1}}}{a^{(i_p)}_{i_{p+1}}}.$$
In particular, $\lambda_1^{(i,j)}=a^{(i)}_j$.
\item For all $p_1,\cdots,p_N \in \mathbb{N}$:
$$a^{(i)}_{(p_1,\cdots,p_{j+1},\cdots,p_N)}=\frac{1}{p_j+1} \left( \lambda^{(i,j)}_{p_1+\cdots+p_N+1}-\sum_{l\in I} p_l a^{(l)}_j \right)
a^{(i)}_{(p_1,\cdots,p_N)}.$$
\end{enumerate} \end{lemma}

\begin{proof} 1. Let us consider a sequence $i_1,\cdots,i_n$ of elements of $I$, such that $i_1=i$ and for all $1\leq p\leq n-1$, 
$a^{(i_p)}_{i_{p+1}}\neq 0$. By definition of $\lambda_n^{(i,j)}$:
\begin{eqnarray*}
\lambda_n^{(i,j)} a_{\tdpartun{$i_1$}{$i_2$}{$i_{n-1}$}{$i_n$}}&=&a_{\tdpartdeux{$i_1$}{$i_2$}{$i_{n-1}$}{$i_n$}{$j$}}+(1+\delta_{j,i_n})
a_{\tdparttrois{$i_1$}{$i_2$}{$i_{n-1}$}{$i_n$}{$j$}}+\sum_{p=1}^{n-2}a_{\tdpartquatre{$i_1$}{$i_p$}{$i_{p+1}$}{$i_n$}{$j$}},\\
\lambda_n^{(i,j)}a^{(i_1)}_{i_2}\cdots a^{(i_{n-1})}_{i_n}&=&a^{(i_1)}_{i_2}\cdots a^{(i_{n-1})}_{i_n}a^{(i_n)}_j
+(1+\delta_{j,i_n})a^{(i_1)}_{i_2}\cdots a^{(i_{n-1})}_{i_n,j}\\
&&+\sum_{p=1}^{n-2}(1+\delta_{j,i_{p+1}})a^{(i_1)}_{i_2}\cdots a^{(i_p)}_{j,i_{p+1}}a^{(i_{p+1})}_{i_{p+2}}\cdots a^{(i_{n-1})}_{i_n},\\
\lambda_n^{(i,j)}&=&a^{(i_n)}_j+\sum_{p=1}^{n-1}(1+\delta_{j,i_{p+1}})
\frac{a^{(i_p)}_{j,i_{p+1}}}{a^{(i_p)}_{i_{p+1}}}.
\end{eqnarray*}
This proves the first point of the lemma. \\

2. Let us now fix $p_1,\cdots,p_N \in \mathbb{N}$. By definition, for $t'=B^+_i(\tdun{$1$}^{p_1}\cdots \tdun{$N$}^{p_N})$:
\begin{eqnarray*}
\lambda^{(i,j)}_{p_1+\cdots+p_N+1}a_{B^+_i(\tdun{$1$}^{p_1}\cdots \tdun{$N$}^{p_N})}
&=&(p_j+1)a_{B^+_i(\tdun{$1$}^{p_1}\cdots \tdun{$j$}^{p_j+1}\cdots \tdun{$N$}^{p_N})}\\
&&+\sum_{l=1}^N a_{B^+_i(\tdun{$1$}^{p_1}\cdots \tdun{$l$}^{p_l-1}\cdots \tdun{$N$}^{p_N} \tddeux{$l$}{$j$})},\\
\lambda^{(i,j)}_{p_1+\cdots+p_N+1}a^{(i)}_{(p_1,\cdots,p_N)}&=&(p_j+1)a^{(i)}_{(p_1,\cdots,p_j+1,\cdots,p_N)}
+\sum_{l=1}^N p_l a^{(i)}_{(p_1,\cdots,p_N)} a^{(l)}_j.
\end{eqnarray*}
This proves the second point of the lemma.  \end{proof}\\

{\bf Remarks.} \begin{enumerate}
\item As a consequence of the second point, if $(S)$ is Hopf and if $a_{(p_1,\cdots,p_N)}^{(i)}=0$, then $a^{(i)}_{(l_1,\cdots,l_N)}=0$ 
if $l_1\geq p_1, \cdots,  l_N\geq p_N$. In particular, as there is no constant $F_i$, for all $i$, there exists a $j$ such that $a^{(i)}_j \neq 0$.
\item So the sequences considered in the first point of lemma \ref{17} always exist.
\item Moreover, for all vertices $i,j$ of $\gs$, $i\rightarrow j$ if and only if $a^{(i)}_j\neq 0$.
\item Finally, for all $i\in I$, for all $p\geq 1$, $X_i(p)\neq 0$.
\end{enumerate}

\begin{prop} \label{18}
Let $(S)$ be a Hopf SDSE.
\begin{enumerate}
\item Let $i,j$ be vertices of $\gs$, such that $j$ is not a descendant of $i$. Then for all $n \geq 1$:
$$\lambda_n^{(i,j)}=0.$$
\item Let $(S)$ be a Hopf SDSE with set of vertices $I$ and let $(S')$ be a Hopf SDSE with set of vertices $J$. Then $(S')$ is a dilatation of $(S)$ 
if, and only if, $J$ admits a partition indexed by the elements of $I$ and for all $i,j\in I$, for all $x\in J_i$, $y\in J_j$, for all $n \geq 1$:
$$\lambda_n^{(i,j)}=\lambda_n^{(x,y)}.$$
\item Let $i\in I$ such that:
$$F_i=1+\sum_{j\in I} a^{(i)}_j h_j.$$
Then for all direct descendant $i'$ of $i$, for all $j$, for all $n \geq 1$:
$$\lambda_{n+1}^{(i,j)}=\lambda_n^{(i',j)}.$$
As a consequence, if $i',i'''$ are two direct descendants of $i$, $F_{i'}=F_{i''}$.
\end{enumerate} \end{prop}

\begin{proof} 1. Let us consider a sequence $i=i_1,\cdots,i_n$ of elements of $I$ such that $a^{(i_k)}_{i_{k+1}}\neq 0$ for all $1\leq k \leq n-1$.
Then $j$ is not a direct descendant of $i_1,\cdots,i_n$, so $a^{(i_n)}_j=0$ and $a^{(i_k)}_{j,i_{k+1}}=0$ for all $k$.
By lemma \ref{17}, $\lambda_n^{(i,j)}=0$.\\

2. $\Longrightarrow$. From lemma \ref{17}-1, choosing an element $x_i$ in $J_i$ for all $i\in I$.

$\Longleftarrow$. Let us consider the dilatation $(S'')$ of $(S)$ corresponding to the partition of $J$. Then the coefficients $\lambda_n^{(i,j)}$
of $(S')$ and $(S'')$ are equal, so by lemma \ref{17}-2, $(S')=(S'')$. \\

3. Let us consider a sequence $i,i'=i_1,\cdots,i_n$ of elements of $I$ such that $a^{(i_k)}_{i_{k+1}}\neq 0$ for all $1\leq k \leq n-1$.
By hypothesis on $i$, $a^{(i)}_{j,i'}=0$. By lemma \ref{17}-1:
$$\lambda_{n+1}^{(i,j)}=a^{(i_n)}_j+0+\sum_{k=1}^{n-1}(1+\delta_{j,i_{k+1}}) \frac{a^{(i_k)}_{j,i_{k+1}}}{a^{(i_k)}_{i_{k+1}}}=\lambda_n^{(i',j)}.$$
So, if $i'$ and $i''$ are two direct descendants of $i$, for all $k\in I$, for all $n\geq 1$, $\lambda_n^{(i',k)}=\lambda_n^{(i'',k)}$.
By lemma \ref{17}-2, $F_{i'}=F_{i''}$. \end{proof}

\begin{prop} \label{19}
Let $(S)$ be an SDSE, with $I=\{1,\ldots,N\}$. It is Hopf if, and only if, the two following conditions are satisfied:
\begin{enumerate}
\item There exist scalars $\lambda_n^{(i,j)}$ satisfying, for all $1\leq i,j \leq N$,
for all $(p_1,\cdots,p_N)\in \mathbb{N}^N$:
$$a^{(i)}_{(p_1,\cdots,p_{j+1},\cdots,p_N)}=\frac{1}{p_j+1} \left( \lambda^{(i,j)}_{p_1+\cdots+p_N+1}-\sum_{l\in I} p_l a^{(l)}_j \right)
a^{(i)}_{(p_1,\cdots,p_N)}.$$
\item For all $p\geq 1$, for all $i,j,d_1,\cdots,d_p \in I$, such that $a^{(i)}_{(p_1,\cdots,p_N)}\neq 0$
where $p_i$ is the number of $d_p$'s equal to $i$, for all $n_1,\cdots,n_p \geq 1$:
$$\lambda^{(i,j)}_{n_1+\cdots+n_p+1}-a^{(i)}_j=\lambda_{p+1}^{(i,j)}-a^{(i)}_j+\sum_{l\in I}\left(\lambda_{n_l}^{(d_l,j)}-a_j^{(d_l)}\right).$$
\end{enumerate}
\end{prop}

\begin{proof} {\it Preliminary step}. Let us assume the first point and let $t'\in \T_\D^{(i)}$. We use the following notations:
$$t'=B_i^+\left(\prod_{s\in \T_\D} s^{r_s}\right).$$
We also denote, for all $j\in I$:
$$p_j=\sum_{s\in \T_\D^{(j)}}r_s.$$
Then, by (\ref{E1}):
$$a_{t'}=\frac{\displaystyle \prod_{j=1}^N p_j!}{\displaystyle \prod_{s\in \T_\D}r_s!}
a_{(p_1,\cdots,p_N)}^{(i)} \prod_{s\in \T_\D} a_s^{r_s}.$$
Hence:
\begin{eqnarray*}
\sum_{t\in \T_\D^{(i)}} n_j(t,t')a_t&=&n_j\left(B_i^+\left(\tdun{$j$} \prod_{s\in \T_\D} s^{r_s}\right),t'\right)
a_{B_i^+\left(\tdun{$j$} \prod s^{r_s}\right)}\\
&&+\sum_{\substack{s_1,s_2\in \T_\D\\r_{s_2}\geq 1}}(r_{s_1}+1) n_j(s_1,s_2)a_{B^+_i\left( \frac{s_1}{s_2} \prod s^{r_s} \right)}\\
&=&(r_{\tdun{$j$}+1}) \frac{\displaystyle(p_j+1) \prod_{j=1}^N p_j!}{\displaystyle (r_{\tdun{$j$}+1})\prod_{s\in \T_\D}r_s!}
a_{(p_1,\cdots,p_{j+1},\cdots, p_N)}^{(i)} a_{\tdun{$j$}}\prod_{s\in \T_\D} a_s^{r_s}\\
&&+\sum_{s_1,s_2\in \T_\D}(r_{s_1}+1)n_j(s_1,s_2)\frac{r_{s_2}}{r_{s_1}+1}a_{t'} \frac{a_{s_1}}{a_{s_2}}\\
&=&(p_j+1)\frac{a_{(p_1,\cdots,p_{j+1},\cdots, p_N)}^{(i)}}{a_{(p_1,\cdots,p_N)}^{(i)}}a_{t'}+\sum_{s_1,s_2\in \T_\D}
n_j(s_1,s_2)r_{s_2}a_{t'} \frac{a_{s_1}}{a_{s_2}}\\
&=&\left(\lambda_{p_1+\cdots+p_N+1}^{(i,j)}-\sum_{l=1}^N p_j a_j^{(l)} 
+\sum_{\substack{s_1,s_2\in \T_\D\\ r_{s_2}>0}}n_j(s_1,s_2) r_{s_2}\frac{a_{s_1}}{a_{s_2}}\right)a_{t'}.
\end{eqnarray*}

$\Longrightarrow$. Let us assume that $(S)$ is Hopf. We already prove the existence of the scalars $\lambda_n^{(i,j)}$. 
We obtain from the preceding computation:
$$\lambda_{weight(t')}^{(i,j)} a_{t'}
=\left(\lambda_{p_1+\cdots+p_N+1}^{(i,j)}-\sum_{l=1}^N p_j a_j^{(l)} 
+\sum_{s_2\in \T_\D}r_{s_2}\lambda_{weight(s_2)}^{(d(s_2),j)}\right)a_{t'},$$
where $d(s_2)$ is the decoration of the root of $s_2$.
Let us choose $p$, $i,j,d_1,\cdots,d_p$, $n_1,\cdots,n_p$ as in the hypotheses of the proposition.
Let us choose for all $1\leq j\leq p$ a tree $s_j$ with root decorated by $d_j$, of weight $n_j$, such that $a_{s_j} \neq 0$:
this always exists (for example take a convenient  ladder). Let us take $t'=B^+_i(s_1\cdots s_p)$. Then $a_{t'} \neq 0$ because
$a^{(i)}_{(p_1,\cdots,p_N)}\neq 0$, so:
$$\lambda_{n_1+\cdots+n_p+1}^{(i,j)}=\lambda_{p+1}^{(i,j)}+\sum_{l=1}^p \left(\lambda^{(d_l,j)}_{n_l}-a^{(d_l)}_j\right).$$

$\Longleftarrow$. Let us show the condition of proposition \ref{16} by induction on the weight $n$ of $t'$.
For $n=1$, then $t'=\tdun{$i$}$. 
Then, by hypothesis on the $a^{(i)}_{(p_1,\cdots,p_N)}$, $a^{(i)}_j=\lambda_1^{(i,j)}$. So:
$$\sum_{t\in \T_i(n+1)} n_j(t,t') a_t=\tddeux{$i$}{$j$}=a^{(i)}_j=\lambda_1^{(i,j)} a_{\tdun{$i$}}.$$
Let us assume the result for all tree of weight $<n$. The preceding computation then gives:
$$\sum_{t\in \T_\D^{(i)}} n_j(t,t')a_t=\left(\lambda_{p_1+\cdots+p_N+1}^{(i,j)}-\sum_{l=1}^N p_j a_j^{(l)} 
+\sum_{\substack{s_1,s_2\in \T_\D\\ r_{s_2}>0}}n_j(s_1,s_2) r_{s_2}\frac{a_{s_1}}{a_{s_2}}\right)a_{t'}.$$
The induction hypothesis and the condition on the coefficients $\lambda_n^{(i,j)}$ then give that this is equal to $\lambda_{weight(t')+1}^{(i,j)} a_{t'}$.
So $\hs$ is a Hopf subalgebra of $\h_I$. \end{proof}

\subsection{Prelie structure on $\hs^*$}

Let us consider a Hopf SDSE $(S)$. Then $\hs^*$ is the enveloping algebra of the Lie algebra $\lies=Prim\left(\hs^*\right)$, which inherits from $Prim(\h_\D^*)$
a pre-Lie product given in the following way: for all $f,g \in \gs$, for all $x\in \hs$, $f\star g$ is the unique element of $\lies$ 
such that for all $x\in vect(X_i(n)\:/\:i\in I,n\geq 1)$,
$$(f\star g)(x)=(f\otimes g)\circ (\pi \otimes \pi) \circ \Delta(x).$$
Let $(f_i(p))_{i\in I, p\geq 1}$ be the basis of $\lies$, dual of the basis $(X_i(p))_{i\in I, p\geq 1}$.
By homogeneity of $\Delta$, and as $\Delta(X_i(n))$ is a linear span of elements $-\otimes X_i(p)$, $0\leq p \leq n$, we obtain
the existence of coefficients $a^{(i,j)}_{k,l}$ such that, for all $i,j\in I$, $k,l\geq 1$:
$$f_j(l) \star f_i(k)=a^{(i,j)}_{k,l} f_i(k+l).$$
By duality, $a^{(i,j)}_{k,l}$ is the coefficient of $X_j(l) \otimes X_i(k)$ in $\Delta(X_i(k+l))$, so is uniquely determined in the following way:
for all $t' \in \T_\D^{(j)}(l)$, $t'' \in \T_\D^{(i)}(k)$,
$$\sum_{t\in \T_\D^{(i)}(k+l)} n(t',t'';t)a_t=a_{k,l}^{(i,j)} a_{t'}a_{t''}.$$

\begin{lemma}
For all $t' \in \T_\D^{(j)}(l)$, $t'' \in \T_\D^{(i)}(k)$:
$$\sum_{t\in \T_\D^{(i)}(k+l)} n(t',t'';t)a_t=\lambda^{(i,j)}_k a_{t'}a_{t''}.$$
\end{lemma}

\begin{proof} By induction on $k$. If $k=1$, then $t''=\tdun{$i$}$, so:
$$\sum_{t\in \T_\D^{(i)}(k+l)} n(t',t'';t)a_t=a_{B^+_i(t'')}=a^{(i)}_j a_{t'}=\lambda_1^{(i,j)}a_{t'}a_{t''},$$
as $a_{t''}=1$. Let us assume the result at all rank $\leq k-1$. We put $\displaystyle t''=B_i^+(\prod_{s\in \T_\D}s^{r_s})$. 
We put $\displaystyle p_j=\sum_{s\in \T_\D^{(j)}} r_s$ for all $j\in I$. Then:
\begin{eqnarray*}
\sum_{t\in \T_\D^{(i)}(k+l)} n(t',t'';t)a_t&=&n\left(t',t'',B_i^+\left(\tdun{$j$} \prod_{s\in \T_\D} s^{r_s}\right)\right)a_{B_i^+\left(t'\prod s^{r_s}\right)}\\
&&+\sum_{\substack{s_1,s_2\in \T_\D\\r_{s_2}\geq 1}}(r_{s_1}+1) n(t',s_2;s_1)a_{B^+_i\left( \frac{s_1}{s_2} \prod s^{r_s} \right)}\\
&=&(r_{t'+1}) \frac{\displaystyle(p_j+1) \prod_{j=1}^N p_j!}{\displaystyle (r_{t'+1})\prod_{s\in \T_\D}r_s!}
a_{(p_1,\cdots,p_{j+1},\cdots, p_N)}^{(i)} a_{t'}\prod_{s\in \T_\D} a_s^{r_s}\\
&&+\sum_{s_1,s_2\in \T_\D}(r_{s_1}+1)n_j(s_1,s_2)\frac{r_{s_2}}{r_{s_1}+1}a_{t''} \frac{a_{s_1}}{a_{s_2}}\\
&=&(p_j+1)\frac{a_{(p_1,\cdots,p_{j+1},\cdots, p_N)}^{(i)}}{a_{(p_1,\cdots,p_N)}^{(i)}}a_{t'}a_{t''}+\sum_{s_1,s_2\in \T_\D}
n_j(s_1,s_2)r_{s_2} \frac{a_{s_1}}{a_{s_2}}a_{t''}\\
&=&\left(\lambda_{p_1+\cdots+p_N+1}^{(i,j)}-\sum_{l=1}^N p_j a_j^{(l)} 
+\sum_{\substack{s_1,s_2\in \T_\D\\ r_{s_2}>0}}n_j(s_1,s_2) r_{s_2}\frac{a_{s_1}}{a_{s_2}}\right)a_{t'}a_{t''}\\
&=&\left(\lambda_{p_1+\cdots+p_N+1}^{(i,j)}-\sum_{l=1}^N p_j a_j^{(l)} 
+\sum_{s_2\in T_\D}r_{s_2}\lambda_{|s_2|}^{(r(s_2),j)}\right)a_{t'}a_{t''},
\end{eqnarray*}
using the induction hypothesis on $s_2$, denoting by $r(s_2)$ the decoration of the root of $s_2$.
By proposition \ref{19}-2, if $a_{t'} \neq 0$, then $a^{(i)}_{(p_1,\cdots,p_n)}\neq 0$, so:
\begin{eqnarray*}
\lambda_{1+\sum r_s|s|}^{(i,j)}&=&\lambda_{1+\sum r_s}^{(i,j)}+\sum_s r_s \left(\lambda_{|s|}^{(r(s),j)}-a^{(r(s))}_j\right)\\
\lambda_{|t''|}^{(i,j)}&=&\lambda_{p_1+\cdots+p_N+1}^{(i,j)}+\sum_s r_s \lambda_{|s|}^{(r(s),j)}-\sum_l p_l a^{(l)}_j.
\end{eqnarray*}
So the induction hypothesis is proved at rank $n$. \end{proof}\\

Combining this lemma with the preceding observations:

\begin{prop} \label{21}
Let $(S)$ be a Hopf SDSE. The pre-Lie algebra $\lies=Prim\left(\hs^*\right)$ has a basis $(f_i(k))_{i\in I, k\geq 1}$, and the pre-Lie product of two elements 
of this basis is given by:
$$f_j(l) \star f_i(k)=\lambda_k^{(i,j)} f_i(k+l).$$
\end{prop}

\section{Level of a vertex}

The second item of proposition \ref{19}-2 is immediately satisfied if there exist scalars $b_j$ and $a^{(i)}_j$ such that 
$\lambda_n^{(i,j)}=b_j(n-1)+a^{(i)}_j$ for all $n \geq 1$ and all $i,j \in I$. This motivates the definition of the level of a vertex.

\subsection{Definition of the level}

\begin{defi}\textnormal{
Let $(S)$ be a Hopf SDSE, and let $i$ be a vertex of $\gs$.  It will be said to be of {\it level $\leq M$} if for all vertex $j$, 
there exist scalar $b^{(i)}_j$, $\tilde{a}^{(i)}_j$, such that for all $n>M$:
$$\lambda_n^{(i,j)}=b^{(i)}_j(n-1)+\tilde{a}^{(i)}_j.$$
The vertex $i$ will be said to be of {\it level $M$} if it is of level $\leq M$ and not of level $\leq M-1$.
}\end{defi}

{\bf Remark.} In order to prove that $i$ is of  level $\leq M$, it is enough to consider the $j$'s which are descendants of $i$. Indeed, if $j$ is not a descendant
of $i$, by proposition \ref{18}-1, $\lambda_n^{(i,j)}=0$ for all $n \geq 1$.

\begin{prop} \label{23}
Let $(S)$ be a Hopf SDSE, $i$ a vertex of $\gs$ and $j$ a direct descendant of $\gs$.
\begin{enumerate}
\item $i$ has level $0$ or $1$ if, and only if, $j$ as level $0$.
\item Let $M \geq 2$. Then $i$ has level $M$ if, and only if, $j$ has level $M-1$.
\end{enumerate}
Moreover, if this holds, then for all $k\in I$, $b^{(i)}_k=b^{(j)}_k$.
\end{prop}

\begin{proof} Let $i\in \gs$ and $j$ be a direct descendant of $i$. As $(S)$ is Hopf, let us use the second point of proposition \ref{19}, 
with $k=1$ and $d_1=j$. Then for all $l$, for all $n \geq 1$, as $a^{(i)}_j \neq 0$:
$$\lambda_{n+1}^{(i,l)}=\lambda_2^{(i,l)}+\lambda_n^{(j,l)}-a^{(j)}_l.$$
So for all $M\geq 1$, $i$ is of level $\leq M$ if, and only if, $j$ is of level $\leq M-1$. Moreover, if this holds, then $b_k^{(i)}=b_k^{(j)}$ for all $k$.

The first point is a reformulation of the preceding result for $M=1$. Let us assume that $M\geq 2$. If $i$ is of level $M$, then $j$ is of level $\leq M-1$. 
If $j$ is of level $\leq M-2$, then $i$ is of level $\leq M-1$: contradiction. So $j$ is of level $M-1$. The converse is proved in the same way. \end{proof}

\begin{cor}
Let $(S)$ be a connected Hopf SDSE. Then if one of the vertices of $\gs$ is of finite level, then all vertices of $\gs$ are of finite level.
Moreover, the coefficients $b^{(i)}_j$ depend only of $j$. They will now be denoted by $b_j$.
\end{cor}

Proposition \ref{18}-1 immediately implies the following result:

\begin{lemma} \label{25}
Let $(S)$ be a connected Hopf SDSE and let $j$ be a vertex of $\gs$ of finite level. 
If there exists a vertex $i$ in $\gs$ which is not a descendant of $j$, then $b_j=0$.
\end{lemma}

\subsection{Vertices of level 0}

Let $(S)$ be a Hopf SDSE with $I=\{1,\ldots,N\}$, and let us assume that $i$ is a vertex of level $0$. In this case, the coefficients 
$a^{(i)}_{(p_1,\cdots,p_N)}$ satisfy an induction of the following form:
\begin{eqnarray*}
&&\left\{ \begin{array}{rcl}
a^{(i)}_{(0,\cdots,0)}&=&1,\\
a^{(i)}_{(p_1,\cdots,p_j+1,\cdots,p_N)}&=&\displaystyle \frac{1}{p_j+1}\left(\lambda_j+\sum_{l=1}^N \mu_j^{(l)} p_l \right) a^{(i)}_{(p_1,\cdots,p_N)}.
\end{array}\right. \end{eqnarray*}
In order to ease the notation, we shall write $a_{(p_1,\cdots,p_N)}$ instead of $a^{(i)}_{(p_1,\cdots,p_N)}$ and $F$ instead of $F_i$ in this section. 

\begin{lemma}\label{26}
Under the preceding hypothesis:
\begin{enumerate}
\item Let us denote $J=\{j\in I\:/\: \lambda_j=0\}$. There exists a partition $I=I_1\cup \cdots \cup I_M\cup J$, and scalars $\beta_1,\cdots,\beta_M$,
such that for all $i,j \in I\setminus J=I_1\cup \cdots \cup I_M$:
$$\mu_i^{(j)}=\left\{ \begin{array}{l}
0\mbox{ if $i,j$ do not belong to the same $I_l$,}\\
\lambda_i \beta_l \mbox{ if $i,j \in I_l$}.
\end{array}\right.$$
\item Moreover $\displaystyle F(h_1,\cdots,h_N)=\prod_{p=1}^M f_{\beta_p}\left( \sum_{l\in I_p} \lambda_l h_l \right)$.
\end{enumerate} \end{lemma}

\begin{proof} Let us fix $i\neq j$. Then:
\begin{eqnarray*}
&&a_{(p_1,\cdots,p_i+1,\cdots,p_j+1,\cdots,p_N)}\\
&=&\frac{1}{p_i+1}\left( \lambda_i+\mu^{(j)}_i+\sum_{l=1}^N \mu_i^{(l)}p_l\right) a_{(p_1,\cdots,p_j+1,\cdots,p_N)}\\
&=&\frac{1}{(p_i+1)(p_j+1)}\left( \lambda_i+\mu^{(j)}_i+\sum_{l=1}^N \mu_i^{(l)}p_l\right)
\left( \lambda_j+\sum_{l=1}^N \mu_j^{(l)}p_l\right)a_{(p_1,\cdots,p_N)},\\
&=&\frac{1}{(p_i+1)(p_j+1)}\left( \lambda_j+\mu^{(i)}_j+\sum_{l=1}^N \mu_j^{(l)}p_l\right)
\left( \lambda_i+\sum_{l=1}^N \mu_i^{(l)}p_l\right)a_{(p_1,\cdots,p_N)}.
\end{eqnarray*}
For $(p_1,\cdots,p_N)=(0,\cdots,0)$, as $a_{(0,\cdots,0)}=1$:
\begin{equation} \label{E2}
\mu_i^{(j)}\lambda_j=\mu_j^{(i)} \lambda_i.
 \end{equation}
For $(p_1,\cdots,p_N)=\varepsilon_k$, we obtain:
$$\left(\lambda_i+\mu^{(j)}_i+\mu^{(k)}_i\right)\left(\lambda_j+\mu^{(k)}_j\right)\lambda_k
=\left(\lambda_j+\mu^{(i)}_j+\mu^{(k)}_j\right)\left(\lambda_i+\mu^{(k)}_i\right)\lambda_k.$$
So, if $\lambda_k \neq 0$:
\begin{equation}\label{E3}
\mu^{(j)}_i\mu^{(k)}_j=\mu^{(i)}_j\mu^{(k)}_i.
\end{equation}
If $\lambda_k=0$, it is not difficult to prove inductively that $a_{(p_1,\cdots,p_N)}=0$ if $p_k>0$,
so $F$ is an element of $K[[h_1,\cdots,h_{k-1},h_{k+1},\cdots,h_N]]$. Hence, up to a restriction to $I\setminus J$, we can suppose that 
all the $\lambda_k$'s are non-zero. We then put $\nu_i^{(j)}=\frac{\mu_i^{(j)}}{\lambda_i}$ for all $i,j$. Then (\ref{E2}) and (\ref{E3}) become: for all $i,j,k$,
\begin{eqnarray}
\label{E4} \nu_i^{(j)}&=&\nu_j^{(i)},\\
\label{E5} \nu_i^{(j)}\left(\nu_i^{(k)}-\nu_j^{(k)}\right)&=&0.
\end{eqnarray}

Let $1\leq i,j \leq N$. We shall say that $i\R j$ if $i=j$ or if $\nu_i^{(j)}\neq 0$. Let us show that $\R$ is an equivalence. By (\ref{E4}), it is clearly symmetric.
Let us assume that $i \R j$ and $j \R k$. If $i=j$ or $j=k$ or $i=k$, then $i \R k$. If $i,j,k$ are distinct, then $\nu_i^{(j)} \neq 0$ and $\nu_j^{(k)}\neq 0$.
By (\ref{E5}), $\nu_i^{(k)}=\nu_j^{(k)} \neq 0$, so $i \R k$. We denote by $I_1,\cdots,I_M$ the equivalence classes of $\R$.

Let us assume that $i\R j$, $i\neq j$. Then $\nu_i^{(j)} \neq 0$, so for all $k$, $\nu_j^{(k)}=\nu_i^{(k)}$.
In particular, $\nu_j^{(i)}=\nu_i^{(i)}=\nu_i^{(j)}=\nu_j^{(j)}$. So, finally, there exists a family of scalars $(\beta_i)_{1\leq i\leq M}$, such that:
\begin{itemize}
\item If $i,j \in I_l$, then $\nu_i^{(j)}=\beta_l$, and $\mu_i^{(j)}=\lambda_i \beta_l$.
\item If $i$ and $j$ are not in the same $I_l$, then $\nu_i^{(j)}=\mu_i^{(j)}=0$.
\end{itemize}
An easy induction then proves:
$$a_{(p_1,\cdots,p_N)}=\frac{\lambda_1^{p_1}\cdots \lambda_N^{p_N}}{p_1!\cdots p_N!}
\prod_{p=1}^M (1+\beta_p)\cdots \left(1+\beta_p\left(\sum_{l\in I_p} p_l-1\right)\right).$$
This implies the assertion on $F$.  \end{proof}

\subsection{Vertices of level 1}

Let us now assume that $i$ is of level $1$. Then, up to a restriction to $i$ and its direct descendants, 
the coefficients $a^{(i)}_{(p_1,\cdots,p_N)}=a_{(p_1,\cdots,p_N)}$ satisfy an induction of the form:
$$ \left\{ \begin{array}{rcl}
a^{(i)}_{(0,\cdots,0)}&=&1,\\
a^{(i)}_{\varepsilon_j}&=&a^{(i)}_j,\\
a^{(i)}_{(p_1,\cdots,p_j+1,\cdots,p_N)}&=&\displaystyle \frac{1}{p_j+1}\left(\lambda_j+\sum_{l=1}^N \mu_j^{(l)} p_l \right) a^{(i)}_{(p_1,\cdots,p_N)}
\mbox{ if }(p_1,\cdots,p_N) \neq (0,\cdots,0).
\end{array}\right.$$
In order to ease the notation, we shall write $a_{(p_1,\cdots,p_N)}$ instead of $a^{(i)}_{(p_1,\cdots,p_N)}$ and $F$ instead of $F_i$ in this section. \\

\begin{lemma}\label{27}
Under the preceding hypothesis, one of the following assertions holds:
\begin{enumerate}
\item There exists a partition $I=I_1\cup \cdots \cup I_M \cup J$, scalars $\beta_1,\cdots,\beta_M$,
a non-zero scalar $\nu$ such that:
$$ F(h_1,\cdots,h_N)=\frac{1}{\nu} \prod_{p=1}^M f_{\beta_p}\left( \sum_{l\in I_p} \nu a_l h_l\right)+\sum_{l\in J} a_lh_l+1-\frac{1}{\nu}.$$
\item There exists a partition $\{1,\cdots,N\}=I_1\cup \cdots \cup I_M \cup J$, scalars $\nu_p$ for $1\leq p \in M$, such that:
$$F(h_1,\cdots,h_N)=1-\sum_{p=1}^M\frac{1}{\nu_p}\ln\left(1-\nu_p\sum_{l\in I_p} a_lh_l\right)+\sum_{l\in J}a_lh_l.$$
\end{enumerate} \end{lemma}

\begin{proof} Let us compute $a_{j,k}$ in two different ways:
$$\left(\lambda_j+\mu_j^{(k)}\right)a_k=\left(\lambda_k+\mu_k^{(j)}\right)a_j.$$
In other words:
\begin{equation}  \label{E6}
\left| \begin{array}{cc}
\lambda_j+\mu_j^{(k)}&a_j\\
\lambda_k+\mu_k^{(j)}&a_k
\end{array}
\right|=0.
\end{equation}
Let us take $J=\{j\:/\:\forall k,\: \lambda_j+\mu_j^{(k)}=0\}$. Let us consider an element $j\in J$. Then an easy induction proves that
for all $(p_1,\cdots,p_N)$ such that $p_1+\cdots+p_N \geq 2$ and $p_j \geq 1$, $a_{(p_1,\cdots,p_N)}=0$. As a consequence:
$$F(h_1,\cdots,h_N)=F(h_1,\cdots,h_{j-1},0,h_{j+1},\cdots,h_N)+a_jh_j.$$
So:
$$F=\tilde{F}(h_i,\:i\notin J)+\sum_{j\in J} a_j h_j.$$

We now assume that, up to a restriction, $J=\emptyset$. Let us choose an $i$ and let us put 
$b_{(p_1,\cdots,p_N)}=(p_i+1)a_{(p_1,\cdots,p_i+1,\cdots,p_N)}$.
Then, for all $j\in I$, for all $(p_1,\cdots,p_N)$:
$$b_{(p_1,\cdots,p_j+1,\cdots,p_N)}=\displaystyle \frac{1}{p_j+1}\left(\lambda_j+\mu_j^{(i)}+\sum_{l=1}^N \mu_j^{(l)} p_l \right) b_{(p_1,\cdots,p_N)}.$$
We deduce from lemma \ref{26} that there exist a partition $I=I_1\cup \cdots \cup I_M$ and scalars $\beta_1,\ldots,\beta_M$, such that:
$$\mu_j^{(l)}=\left\{ \begin{array}{l}
0\mbox{ if $j,l$ are not in the same $I_k$},\\
\left(\lambda_j+\mu^{(i)}_j\right)\beta_k \mbox{ if $j,l \in I_k$}.
\end{array}\right.$$
So $\mu_j^{(i)}$ does not depend on $i$ such that $\mu_j^{(i)}\neq 0$. So there exist scalars $\mu_j$ such that:
$$\mu_j^{(l)}=\left\{ \begin{array}{l}
0\mbox{ if $j,l$ are not in the same $I_k$},\\
\left(\lambda_j+\mu_j\right)\beta_k \mbox{ if $j,l \in I_k$}.
\end{array}\right.$$
\begin{enumerate}
\item Let us assume that $M \geq 2$. Let us choose $j \in I_1$. Then for all $k\in I_2\cup \cdots \cup I_M$, (\ref{E6}) gives:
$$\left| \begin{array}{cc}
\lambda_j&a_j\\
\lambda_k&a_k
\end{array} \right|=0.$$
We denote $I_2\cup \cdots \cup I_k=\{i_1,\cdots,i_M\}$. We proved that the vectors $(\lambda_j,\lambda_{i_1},\cdots,\lambda_{i_M})$ 
and $(a_j,a_{i_1},\cdots,a_{i_M})$ are colinear. Choosing then a $j\in I_2$, we obtain that there exists a scalar $\nu$, such that 
$(\lambda_i)_{i\in I}=\nu (a_i)_{i\in I}$. Two cases are possible.
\begin{enumerate}
\item If $\nu \neq 0$,  putting $a'_{(p_1,\cdots,p_N)}=\nu a_{(p_1,\cdots,p_N)}$ if $(p_1,\cdots,p_N)\neq(0,\cdots,0)$ and $a'_{(0,\cdots,0)}$, then 
the family $\left(a'_{(p_1,\cdots,p_N)}\right)$ satisfies the hypothesis of lemma \ref{26}. As a consequence, $F(h_1,\cdots,h_N)$ satisfies the first case.
\item If $\nu=0$, then we put, for all $j$, $\mu_j=\nu'_j a_j$. By (\ref{E6}), for $j$ and $k$ in the same $I_l$,
$\nu'_j=\nu'_k$ if $j$ and $k$ are in the same $I_l$: this common value is now denoted $\nu_l$. It is then not difficult to prove that:
$$F(h_1,\cdots,h_N)=1-\sum_{p=1}^M \frac{1}{\nu_p}\ln\left(1-\nu_p\sum_{l\in I_p} a_lh_l\right).$$
This is a second case.
\end{enumerate}
\item Let us assume that $M=1$.  Then $(\lambda_j+\mu_j)\beta_1=\mu^{(i)}_j$ for all $i,j \in I$.
\begin{enumerate}
\item Let us suppose that $\beta_1\neq 1$. Then, for all $j,k \in I$  $\mu_j=\frac{\beta_1}{1-\beta_1}\lambda_j$. So, for all $j$,
$\lambda_j+\mu_j=\frac{1}{1-\beta_1}\lambda_j$. So (\ref{E6}) implies that $(\lambda_j)_{j\in I}$ and $(a_j)_{j\in I}$ are colinear. As in 1.(a), this is a first case.
\item Let us assume that $\beta_1=1$. So $\lambda_j=0$ for all $j$. As in 1.(b), this is a second case.
\end{enumerate} \end{enumerate} \end{proof}

\subsection{Vertices of level $\geq 2$}

\begin{lemma} \label{28}
Let $(S)$ be a Hopf SDSE and let $i$ be a vertex of $\gs$. We suppose that there exists a vertex $j$, such that:
\begin{itemize}
\item $j$ is a descendant of $i$.
\item All oriented path from $i$ to $j$ are of length $\geq 3$.
\end{itemize} 
Then $\displaystyle F_i=1+\sum_{i \longrightarrow l} a^{(i)}_l h_l$.
\end{lemma}

\begin{proof} We assume here that $I=\{1,\ldots,N\}$. Let $L$ be the minimal length of the oriented paths from $i$ to $j$. By hypothesis, $L\geq 3$. 
Then the homogeneous component of degree $L+1$ of $X_i$ contains trees with a leave decorated by $j$, and all these trees are ladders 
(that is to say trees with no ramification). By proposition \ref{16}, if $t'\in \T_\D^{(i)}(L)$:
$$\lambda_L^{(i,j)}a_{t'}=\sum_{t\in \T_\D^{(i)}(L+1)} n_j(t,t')a_t.$$
For a good-chosen ladder $t'$, the second member is non-zero, so $\lambda_L^{(i,j)}$ is non-zero. If $t'$ is not a ladder, the second member is $0$, 
so $a_{t'}=0$. As a conclusion, $X_i(L)$ is a linear span of ladders. Considering its coproduct, for all $p\leq L$, $X_i(p)$ is a linear span of ladders. 
In particular, $X_i(3)$ is a linear span of ladders. But:
$$X_i(3)=\sum_{l,m} a^{(i)}_la^{(l)}_m \tdtroisdeux{$i$}{$l$}{$m$}+\sum_{l\leq m} a_{l,m}^{(i)}\tdtroisun{$i$}{$m$}{$l$},$$
so $a_{l,m}^{(i)}=0$ for all $l,m$. Hence, $F_i$ contains only terms of degree $\leq 1$. \end{proof}\\

{\bf Remark.} This lemma can be applied with $i=j$, if $i$ is not a self-dependent vertex.

\begin{prop} \label{29}
Let $(S)$ be a Hopf SDSE and let $i$ be a vertex of $\gs$ of level $\geq 2$. Then $i$ is an extension vertex.
\end{prop}

\begin{proof} We denote by $M$ the level of $i$. By proposition \ref{23}, all the descendants of $i$ are of level $\leq M-1$, so $i$ is not a descendant of itself.

Let $M$ be the level of $i$ and let us assume that $M\geq 3$. Let $j$ be a direct descendant of $i$, $k$ be a direct descendant of $j$, 
$l$ be a direct descendant of $k$. Then $j$ has level $M-1$, $k$ has level $M-2$, $l$ has level $M-3$. So in the graph of the restriction to $\{i,j,k,l\}$ is:
$$\xymatrix{i\ar[r]&j\ar[r]&k\ar[r]&l}\mbox{ or }\xymatrix{i\ar[r]&j\ar[r]&k\ar[r]&l\ar@(ur,dr)}$$
The result is then deduced from lemma \ref{28}. \\

Let us now assume that $i$ is of level $2$ and is not an extension vertex. Let $j$ be a direct descendant of $i$ and $k$ be a direct descendant of $j$. 
By proposition \ref{23}, $j$ is of level $1$ and $k$ is of level $0$, so $k$ is not a direct descendant of $i$. The graph of the restriction of $(S)$ to $\{i,j,k\}$ is:
$$\xymatrix{i\ar[r]&j\ar[r]&k}\mbox{ or }\xymatrix{i\ar[r]&j\ar[r]&k\ar@(ur,dr)}$$

{\it First step.} Let us first prove that there exists a direct descendant  $j$ of $i$ such that $a^{(i)}_{j,j} \neq 0$.
Let us assume that this is not true. As $i$ is not an extension vertex, there exist $j,j' \in I$ such that $a^{(i)}_{j,j'} \neq 0$, $j \neq j'$.
Let $k$ be a direct descendant of $j$. Considering the different levels, the graph associated to the restriction to $\{i,j,j',k\}$ is:
$$\xymatrix{&i\ar[rd] \ar[ld]&\\j\ar[rd]&&j'\ar[ld]\\&k&}\mbox{ or } \xymatrix{&i\ar[rd] \ar[ld]&\\j\ar[rd]&&j'\ar[ld]\\&k\ar@(ur,dr)&}\mbox{ or }
\xymatrix{&i\ar[rd] \ar[ld]&\\j\ar[rd]&&j'\\&k&}\mbox{ or } \xymatrix{&i\ar[rd] \ar[ld]&\\j\ar[rd]&&j'\\&k\ar@(ur,dr)&}$$
Up to a change of variables, we put:
$$F_i(0,\cdots,0,h_j,0,\cdots,0,h_{j'},0,\cdots,0)=1+h_j+h_{j'}+b h_jh_{j'}+\mathcal{O}(h^3).$$
Then by proposition \ref{16}, $\lambda_2^{(i,j)} a_{\tddeux{$i$}{$j$}}=2a_{\tdtroisun{$i$}{$j$}{$j$}}+a_{\tdtroisdeux{$i$}{$j$}{$j$}}=0$, 
so $\lambda_2^{(i,j)}=0$. On the other hand, $\lambda_2^{(i,j)} a_{\tddeux{$i$}{$j'$}}=a_{\tdtroisun{$i$}{$j'$}{$j$}}+a_{\tdtroisdeux{$i$}{$j'$}{$j$}}=b$,
so $0=b$: this contradicts $a^{(i)}_{j,j'} \neq 0$.\\

{\it Second step}. Let us consider a vertex $j$ such that $a^{(i)}_{j,j}\neq 0$. Up to a change of variables, we can assume that $a^{(i)}_j=1$ and that 
for all direct descendant $k$ of $j$, $a^{(j)}_k=1$. By lemma \ref{25}, $b_i=b_j=0$. So, as $i$ is of level $2$, there exist scalars $a,b$, such that:
$$\lambda_n^{(i,j)}=\left\{ \begin{array}{l}
1\mbox{ if }n=1,\\
a\mbox{ if }n=2,\\
b\mbox{ if }n\geq 3.
\end{array}\right.$$
Then proposition \ref{19}-1 implies:
$$F_i(0,\cdots,0,h_j,0,\cdots,0)=1+h_j+\frac{a}{2!}h_j^2+\frac{ab}{6}h_j^3+\mathcal{O}(h_j^4).$$
By hypothesis, $a\neq 0$. Moreover, by proposition \ref{16}, $b=\lambda_3^{(i,j)} a_{\tdtroisdeux{$i$}{$j$}{$k$}}=a_{\tdquatretrois{$i$}{$j$}{$k$}{$j$}}=a$. So:
$$F_i(0,\cdots,0,h_j,0,\cdots,0)=1+h_j+\frac{a}{2!}h_j^2+\frac{a^2}{6}h_j^3+\mathcal{O}(h_j^4).$$
As $j$ has level $1$, we put:
$$\lambda_n^{(j,k)}=\left\{\begin{array}{c}
a^{(j)}_k=1 \mbox{ if }n=1,\\
c(n-1)+d\mbox{ if }n\geq 2,
\end{array}\right.$$
where $c(=b_k)$ and $d$ are scalars.  From proposition \ref{19}-1:
$$F_j(0,\cdots,0,h_k,0,\cdots,0)=1+h_k+\frac{c+d}{2!}h_k^2+\frac{(c+d)(2c+d)}{6}h_k^3+\mathcal{O}(h_k^4).$$
Moreover, $\lambda_3^{(i,k)}a_{\tdtroisun{$i$}{$j$}{$j$}}=a_{\tdquatretrois{$i$}{$j$}{$k$}{$j$}}$, so $\lambda_3^{(i,k)}\frac{a}{2}=a$ and $\lambda_3^{(i,k)}=2$. 
Then $\lambda_3^{(i,k)}a_{\tdtroisdeux{$i$}{$j$}{$k$}}=2a_{\tdquatrequatre{$i$}{$j$}{$k$}{$k$}}$, so $c+d=2$.
Similarly, using $\tdquatreun{$i$}{$j$}{$j$}{$j$}$, we obtain $\lambda_4^{(i,k)}=3$. Using $\tdquatrequatre{$i$}{$j$}{$k$}{$k$}$, we obtain:
$$3\frac{c+d}{2}=3\frac{(c+d)(2c+d)}{6}.$$
As $c+d=2$, $2c+d=3$, so $c=d=1$ and $\lambda_n^{(j,k)}=n$ for all $n\geq 2$. As $\lambda_1^{(j,k)}=1$, $\lambda_n^{(j,k)}=n$ for all $n \geq 1$. 

Let now $l\in I$ which is not a direct descendant of $j$ and let $k$ be a direct descendant of $j$. For all $n \geq 1$:
$$\lambda_n^{(j,l)}=\lambda_n^{(j,l)} a_{B^+_j(\tdun{$k$}^{n-1})}=a_{B^+_j(\tdun{$k$}^{n-1} \tddeux{$k$}{$l$})}=(n-1)a^{(k)}_l.$$
We proved that for any vertex $l$ of $\gs$, for all $n \geq 1$:
$$\lambda_n^{(j,l)}= \left\{ \begin{array}{l}
n\mbox { if $l$ is a direct descendant of $j$},\\
a^{(k)}_l(n-1) \mbox{ if $l$ is not a direct descendant of $j$},
\end{array}\right.$$
where $k$ is any direct descendant of $j$. This proves that $j$ has level $0$, so $i$ has level $1$: contradiction. So $i$ is an extension vertex. \end{proof}

\section{Examples of Hopf SDSE}

\subsection{cycles and multicycles}

{\bf Notation.} We denote by $l(i_1,\cdots,i_n)$ the ladder with decorations, from the root to the leave, $i_1,\cdots,i_n$.
In other words: 
$$l(i_1,\cdots,i_p)=B^+_{i_1} \circ \cdots\circ B^+_{i_n}(1)=\tdpartun{$i_1$}{$i_2$}{$i_{n-1}$}{$i_n$}.$$

\begin{theo} \label{30}
Let $N \geq 2$. The SDSE associated to the following formal series is Hopf:
$$\left\{ \begin{array}{rcl}
F_1&=&1+h_2,\\
&\vdots&\\
F_{N-1}&=&1+h_N,\\
F_N&=&1+h_1.
\end{array}\right.$$
\end{theo}

\begin{proof} We identify $\{1,\cdots,N\}$ and $\mathbb{Z}/N\mathbb{Z}$, via the bijection $i\longrightarrow \overline{i}$.
Then, for all $n\geq 1$ and for all $1\leq i \leq N$, $X_{\overline{i}}(n)=l(\overline{i},\cdots \overline{i+n-1})$. As a consequence:
$$\Delta(X_{\overline{i}})=X_{\overline{i}} \otimes 1+1\otimes X_{\overline{i}}+\sum_{p=1}^{+\infty} X_{\overline{i+p}} \otimes X_{\overline{i}}(p).$$
So $\hs$ is Hopf. \end{proof}\\

Note that the graph $\gs$ associated to such a system is an oriented cycle of length $N$, with only non-self-dependent vertices.

\begin{defi}\textnormal{
Let $(S)$ be a Hopf SDSE. It will be said to be {\it multicyclic} if, up to change of variable, it is a dilatation of a system described in theorem \ref{30}.
}\end{defi}

The graph of a multicyclic SDSE will be called a multicycle. In other term, a  $N$-multicycle ($N \geq 2$) is such that the set $I$ of its vertices admits 
a partition $I=I_{\overline{1}}\cup \cdots \cup I_{\overline{N}}$ indexed by the elements of $\mathbb{Z}/N\mathbb{Z}$, such that the direct descendants
of a vertex $i$ in $I_{\overline{j}}$ are the elements of $I_{\overline{j+1}}$ for all $j\in \mathbb{Z}/N\mathbb{Z}$. Moreover, up to a change of variables, 
for all $i \in \gs$:
$$F_i=1+\sum_{i \longrightarrow l} h_l.$$

Here is an example of a $5$-multicycle:
$$\includegraphics[height=4cm]{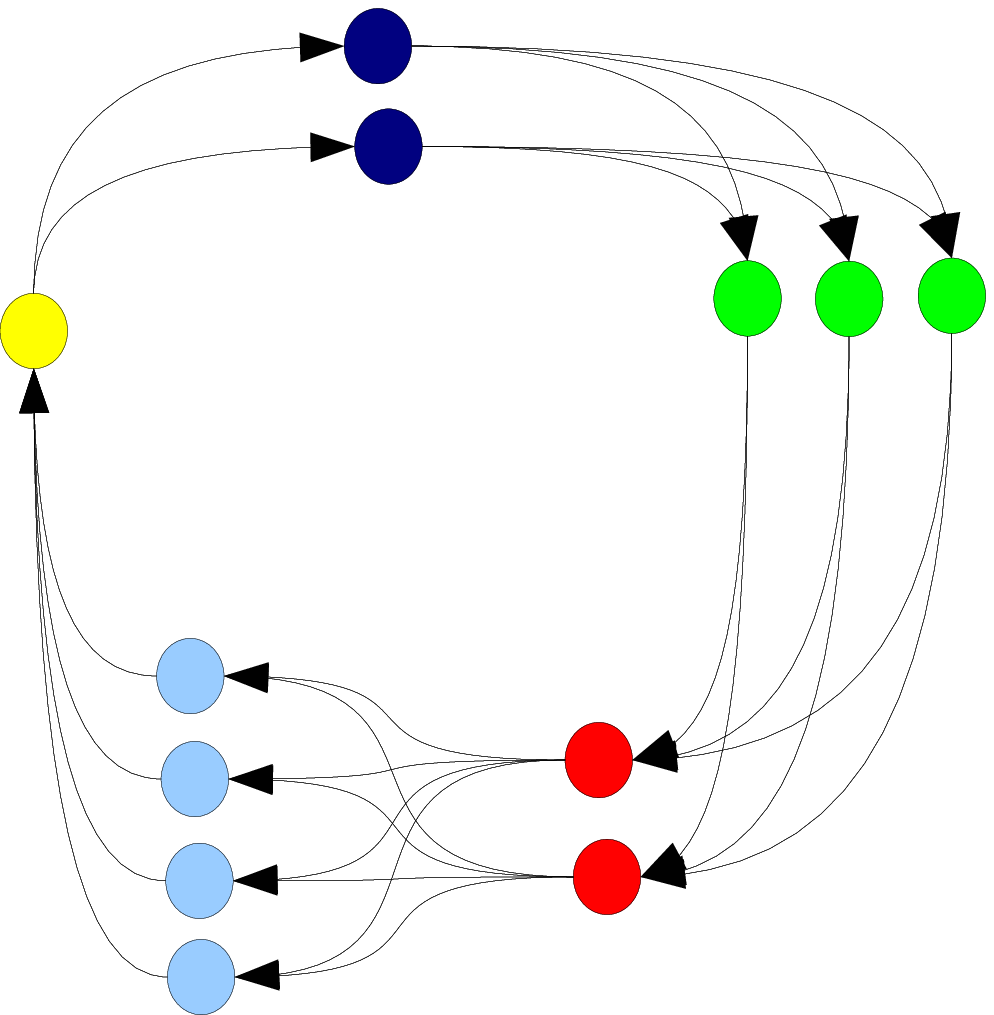}$$

Note that if $N=2$, $\gs$ is a complete bipartite graph, that is to say that the set of vertices of $\gs$ admits a partition into two parts,
and for all vertices $i$ and $j$, there is an edge from $i$ to $j$ if, and only if, $i$ and $j$ are not in the same part of the partition.

\subsection{Fundamental SDSE}

\begin{theo} \label{32}
Let $I$ be a set with a partition $I=I_0 \cup J_0 \cup K_0 \cup I_1 \cup J_1$, such that:
\begin{itemize}
\item $I_0$, $J_0$, $K_0$, $I_1$, $J_1$ can be empty.
\item $I_0 \cup J_0$ is not empty.
\end{itemize}
The SDSE defined in the following way is Hopf:
\begin{enumerate}
\item For all $i \in I_0$, there exists $\beta_i \in K$, such that:
$$F_i=f_{\beta_i}(h_i) \prod_{j\in I_0-\{i\}} f_{\frac{\beta_j}{1+\beta_j}}((1+\beta_j) h_j) \prod_{j\in J_0} f_1(h_j).$$
\item For all $i\in J_0$:
$$F_i=\prod_{j\in I_0} f_{\frac{\beta_j}{1+\beta_j}}((1+\beta_j) h_j) \prod_{j\in J_0-\{i\}} f_1(h_j).$$
\item For all $i\in K_0$:
$$F_i=\prod_{j\in I_0} f_{\frac{\beta_j}{1+\beta_j}}((1+\beta_j) h_j) \prod_{j\in J_0} f_1(h_j).$$
\item For all $i\in I_1$, there exist $\nu_i\in K$, a family of scalars $(a_j^{(i)})_{j\in I_0\cup J_0 \cup K_0}$, such that 
$(\nu_i \neq 1)$ or $(\exists j\in I_0, \: a^{(i)}_j \neq 1+\beta_j)$ or $(\exists j\in J_0, \: a^{(i)}_j \neq 1)$ or $(\exists j\in K_0, \: a^{(i)}_j \neq 0)$. 
Then, if $\nu_i \neq 0$:
$$F_i=\frac{1}{\nu_i} \prod_{j\in I_0} f_{\frac{\beta_j}{\nu_i a^{(i)}_j}}\left(\nu_i a^{(i)}_jh_j\right)
\prod_{j\in J_0} f_{\frac{1}{\nu_i a^{(i)}_j}}\left(\nu_i a^{(i)}_jh_j\right) \prod_{j\in K_0} f_0\left(\nu_i a^{(i)}_jh_j\right)+1-\frac{1}{\nu_i}.$$
If $\nu_i=0$:
$$F_i=-\sum_{j\in I_0} \frac{a^{(i)}_j}{\beta_j} \ln(1-h_j)-\sum_{j\in J_0} a^{(i)}_j \ln(1-h_j)+\sum_{j\in K_0} a^{(i)}_j h_j+1.$$
\item For all $i \in J_1$, there exists $\nu_i\in K-\{0\}$, a family of scalars $(a_j^{(i)})_{j\in I_0\cup J_0 \cup K_0\cup I_1}$, 
with the following conditions:
\begin{itemize}
\item $I_1^{(i)}=\{j\in I_1\:/\:a^{(i)}_j \neq 0\}$ is not empty.
\item For all $j\in I_1^{(i)}$, $\nu_j=1$.
\item For all $j,k \in I_1^{(i)}$, $F_j=F_k$. In particular, we put $b^{(i)}_t=a^{(j)}_t$ for any $j\in I_1^{(i)}$, for all $t\in I_0 \cup J_0 \cup K_0$.
\end{itemize}
Then:
\begin{eqnarray*}
F_i&=&\frac{1}{\nu_i} \prod_{j\in I_0} f_{\frac{\beta_j}{b^{(i)}_j-1-\beta_j}}\left(\left(b^{(i)}_j-1-\beta_j\right)h_j\right)
\prod_{j\in J_0} f_{\frac{1}{b^{(i)}_j-1}}\left(\left(b^{(i)}_j-1\right)h_j\right)\prod_{j\in K_0} f_0\left(b^{(i)}_jh_j\right)\\
&&+\sum_{j\in I_1^{(i)}} a^{(i)}_j h_1+1-\frac{1}{\nu_i}.
\end{eqnarray*}\end{enumerate}\end{theo}

\begin{proof} In order to simplify the notation, we assume that $I=\{1,\ldots,N\}$. We shall use proposition \ref{19} with, for all $i,j \in I$:
$$\lambda_n^{(i,j)}=\left\{ \begin{array}{l}
a^{(i)}_j \mbox{ if } n=1,\\
\tilde{a}^{(i)}_j+b_j(n-1) \mbox{ if }n \geq 2,
\end{array}\right.$$
the coefficients being given in the following arrays:
\begin{enumerate}
\item $a^{(j)}_i$:
$$\begin{array}{|c|c|c|c|c|c|}
\hline i\setminus j&\in I_0&\in J_0&\in K_0&\in I_1&\in J_1\\
\hline \in I_0&(1+\beta_i)-\delta_{i,j} \beta_i&1+\beta_i&1+\beta_i&a^{(j)}_i&\frac{b_i^{(j)} -1-\beta_i}{\nu_j}\\
\hline \in J_0&1&1-\delta_{i,j}&1&a^{(j)}_i&\frac{b_i^{(j)}-1}{\nu_j}\\
\hline \in K_0&0&0&0&a^{(j)}_i&\frac{b_i^{(j)}}{\nu_j}\\
\hline \in I_1&0&0&0&0&a^{(j)}_i\\
\hline \in J_1&0&0&0&0&0\\
\hline \end{array}$$
\item $\tilde{a}^{(j)}_i$:
$$\begin{array}{|c|c|c|c|c|c|}
\hline i\setminus j&\in I_0&\in J_0&\in K_0&\in I_1&\in J_1\\
\hline \in  I_0&(1+\beta_i)-\delta_{i,j} \beta_i&1+\beta_i&1+\beta_i&\nu_j a^{(j)}_i&b_i^{(j)} -1-\beta_i\\
\hline \in J_0&1&1-\delta_{i,j}&1&\nu_j a^{(j)}_i&b_i^{(j)}-1\\
\hline \in K_0&0&0&0&\nu_j a^{(j)}_i&b_i^{(j)}\\
\hline \in I_1&0&0&0&0&0\\
\hline \in J_1&0&0&0&0&0\\
\hline \end{array}$$
\item $b_j$:
$$\begin{array}{|c|c|c|c|c|c|}
\hline j&\in I_0&\in J_0&\in K_0&\in I_1&\in J_1\\
\hline b_j&1+\beta_j&1&0&0&0\\
\hline \end{array}$$
\end{enumerate}
The second item of proposition \ref{19} is immediate. Let us prove for example the first item for $i \in J_1$ and $j \in I_0$. 
Let us fix $(p_1,\ldots,p_N) \in \mathbb{N}^N-\{(0,\ldots,0)\}$. 
\begin{eqnarray*}
&&\lambda_{p_1+\ldots+p_N+1}^{(i,j)}-\sum_l a^{(l)}_j p_l\\
&=&b^{(i)}_j-1-\beta_j-(1+\beta_j) \sum_{l=1}^N p_l-\sum_{l\in I_0\cup J_0 \cup K_0} (1+\beta_j)p_l +\beta_j p_j
-\sum_{l\in I_1\cup J_1} a^{(l)}_j p_l\\
&=&b^{(i)}_j-1-\beta_j+\beta_j p_j+\sum_{l\in I_1\cup J_1}\left(1+\beta_j-a^{(l)}_j\right)p_l.
\end{eqnarray*}
If there exists $l\in (I_1\cup J_1)-I_1^{(i)}$, such that $p_l \neq 0$, then $a^{(i)}_{(p_1,\ldots,p_j+1,\ldots,p_N)}=a^{(i)}_{(p_1,\ldots,p_N)}=0$
and then the result is immediate. We now suppose that $p_l=0$ for all $l\in (I_1\cup J_1)-I_1^{(i)}$. Then:
\begin{eqnarray*}
\lambda_{p_1+\ldots+p_N+1}^{(i,j)}-\sum_l a^{(l)}_j p_l
&=&b^{(i)}_j-1-\beta_j+\beta_j p_j+\sum_{l\in I_1^{(i)}}\left(1+\beta_j-a^{(l)}_j\right)p_l\\
&=&b^{(i)}_j-1-\beta_j+\beta_j p_j+\left(1+\beta_j-b^{(i)}_j\right)\sum_{l\in I_1^{(i)}}p_l.
\end{eqnarray*}
\begin{enumerate}
\item If $\displaystyle \sum_{l\in I_1^{(i)}}p_l=0$, then:
$$a^{(i)}_{(p_1,\ldots,p_j+1,\ldots,p_N)}=\left(b^{(i)}_j-1-\beta_j p_j\right)\frac{a^{(i)}_{(p_1,\ldots,p_N)}}{p_j+1}.$$
The first item of proposition \ref{19} is immediate.
\item If $\displaystyle \sum_{l\in I_1^{(i)}}p_l=1$, then $a^{(i)}_{(p_1,\ldots,p_j+1,\ldots,p_N)}=0$
and $\lambda_{p_1+\ldots+p_N+1}^{(i,j)}-\sum_l a^{(l)}_j p_l=0$. So the first item of proposition \ref{19} holds.
\item If $\displaystyle \sum_{l\in I_1^{(i)}}p_l\geq 2$, then $a^{(i)}_{(p_1,\ldots,p_j+1,\ldots,p_N)}=a^{(i)}_{(p_1,\ldots,p_N)}=0$,
so the result is immediate.
\end{enumerate}
The other cases are proved in the same way, so this SDSE is Hopf. \end{proof} \\

{\bf Remarks.} \begin{enumerate}
\item For all $\lambda \neq 0$:
$$f_{\frac{\beta}{\lambda}}(\lambda h)=\sum_{k=0}^{\infty} \frac{\lambda (\lambda+\beta)\cdots(\lambda+(k-1)\beta)}{k!}h^k.$$
The second side of this formula is equal to $1$ if $\lambda=0$. So, formulas defining the SDSE of theorem \ref{32} are always defined.
\item The vertices of $I_0\cup J_0 \cup K_0$ are of level $0$. A vertex $i$ of $I_1$ is of level $0$ if $\nu_i=1$; otherwise, it is of level $1$.
The vertices of $J_1$ are of level $1$. 
\end{enumerate}

\begin{defi} \label{33} \textnormal{\begin{enumerate}
\item A Hopf SDSE will be said to be {\it fundamental} if, up to a change of variables, it is the dilatation of a system of theorem \ref{32}.
\item A fundamental Hopf SDSE $(S)$ will be said to be {\it abelian} if for any vertex $i\in I$, $b_i=0$.
\end{enumerate}}\end{defi}

{\bf Remark.} In other words, $(S)$ is abelian if $J_0=\emptyset$ and if for any $i \in I_0$, $\beta_i=-1$. Then, for all $i\in K_0$, $F_i=1$.
As there is no constant $F_i$, we obtain $K_0=\emptyset$.\\

A particular case is obtained when $I=J_0$. Then we obtain the following systems:

\begin{theo} \label{34}
Let $I$ be a finite subset which is not a singleton. The SDSE associated to the following formal series is Hopf:
$$F_i=\prod_{j\neq i} (1-h_j)^{-1},\mbox{ for all }i\in I.$$
\end{theo}

The graph associated to such an SDSE is a complete graph with only non-self-dependent vertices, that is to say that there is an edge from $i$ to $j$ 
in $\gs$if, and only if, $i \neq j$. In particular, if $N=2$, $\gs$ is $1\longleftrightarrow 2$, as for the SDSE of theorem \ref{30} with $N=2$.

\begin{defi}\textnormal{
Let $(S)$ be a Hopf SDSE. It will be said to be {\it quasi-complete} if, up to change of variable, it is a dilatation of one of the systems 
described in theorem \ref{34}.
}\end{defi}

The graphs associated to quasi-complete SDSE shall be called {\it quasi-complete}. A quasi-complete graph $G$ has only non-self-dependent vertices;
there exists a partition $I=I_1\cup \cdots \cup I_M$ of the set $I$ of vertices of $\gs$ such that, for all $x,y\in I$, there is an edge from $x$ to $y$
if, and only if, $x$ and $i$ are not in the same $I_i$. In particular, quasi complete graphs with $M=2$ are complete bipartite graphs. 
Moreover, if $(S)$ is quasi-complete, up to a change of variables, for all $x \in I_i$:
$$F_x=\prod_{j\neq i}\left(1-\sum_{y\in I_j} h_y \right)^{-1}.$$
Here is an example of a $2$-quasi-complete graph and a $3$-quasi-complete graph:

$$\includegraphics[height=2cm]{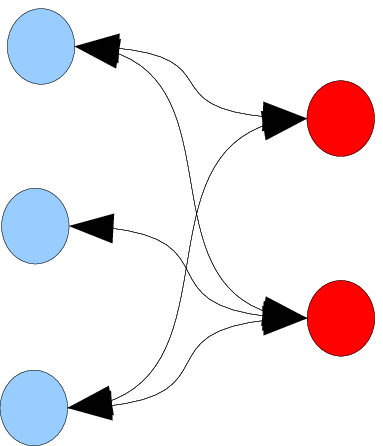} \hspace{1cm} \includegraphics[height=2.5cm]{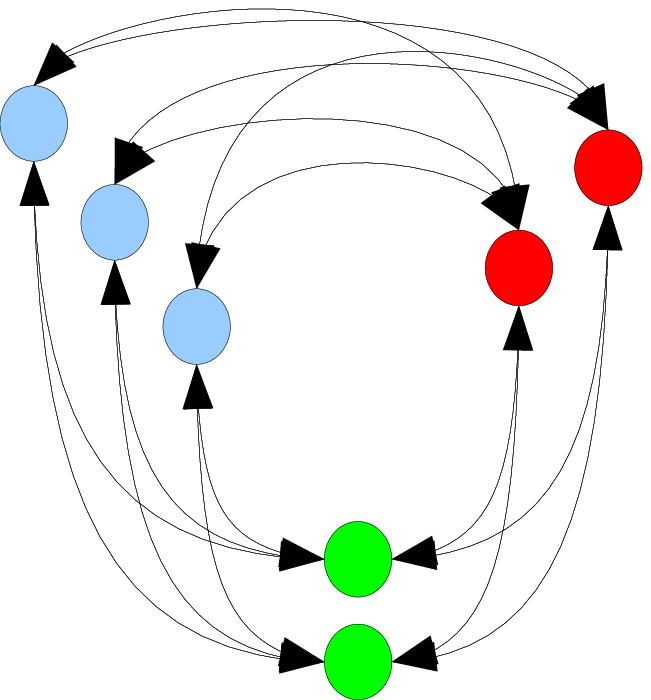}$$

Another particular case is the following: assume that $I=I_0$ and that $\beta_x=-1$ for all $x \in I_0$. Then, for all $x \in I$, $F_x=1+h_x$. 
Note that $\gs$ is not connected if $|I|\geq 2$, and this is the only case where $\gs$ is not connected. The dilatation of such an SDSE will be called 
{\it a non-connected fundamental SDSE}. For such an SDSE, the set of indices $I$ admits a partition $I=I_1\cup \cdots \cup I_M$ ($M\geq 2$)
and up to a change of variables, for all $1\leq i\leq M$, for all $x \in I_i$:
$$F_x=1+\sum_{y\in I_i} h_y.$$

{\bf Remark.} Note that a dilatation replacing $x \in K_0 \cup I_1 \cup J_1$ by a set $J_x$ in a system of theorem \ref{32} also gives a system 
of theorem \ref{32}. The same remark applies when the dilatation replaces $x \in I_0$, with $\beta_x=0$, by a set $J_x$. So we shall always assume 
that the dilatation giving a fundamental SDSE from an SDSE of theorem \ref{32} satisfies $J_x=\{x\}$ for any $x\in K_0 \cup I_1 \cup J_1$ 
and for any $x \in I_0$ such that $\beta_x=0$.

\section{Two families of Hopf SDSE}

We here first give characterisations of multicyclic and quasi-complete SDSE. We then consider Hopf SDSE such that any vertex is a descendant of 
a self-dependent vertex. We prove that such an SDSE is fundamental. The results of this section will be used to prove the main theorem \ref{14}.

\subsection{A lemma on non-self-dependent vertices}

\begin{lemma} \label{36} 
Let $(S)$ be a Hopf SDSE and let $i \in I$ such that $a_i^{(i)}=0$. Let $j$, $k$ and $l \in I$ such that $a_j^{(i)}\neq 0$, $a_k^{(j)}\neq 0$ 
and $a_l^{(i)}\neq 0$. Then $a^{(i)}_k \neq 0$ or $a^{(l)}_k \neq 0$.
\end{lemma}

\begin{proof} Let us assume that $a^{(i)}_k=0$. As $a^{(i)}_j \neq 0$, $j\neq k$. As $a^{(i)}_k=0$, $a_{\tdtroisun{$i$}{$k$}{$j$}}=a^{(i)}_{j,k}=0$. 
Then, from proposition \ref{16}, $a^{(i)}_j \lambda_2^{(i,k)}=\lambda_2^{(i,k)}a_{\tddeux{$i$}{$j$}}=a_{\tdtroisdeux{$i$}{$j$}{$k$}}+a_{\tdtroisun{$i$}{$k$}{$j$}}
= a^{(i)}_j a^{(j)}_k+0$; hence, $\lambda_2^{(i,k)}=a^{(j)}_k$. Moreover, As $a^{(i)}_l \neq 0$, $l\neq k$. 
Then, by proposition \ref{16}, $a^{(i)}_l \lambda_2^{(i,k)}=\lambda_2^{(i,k)}a_{\tddeux{$i$}{$l$}}=a_{\tdtroisdeux{$i$}{$l$}{$k$}}
+a_{\tdtroisun{$i$}{$k$}{$l$}}=a^{(i)}_l a^{(l)}_k+0$, so $\lambda_2^{(i,k)}=a^{(l)}_k$. Hence, $a^{(l)}_k=a^{(j)}_k \neq 0$. \end{proof} \\

{\bf Remark.} In other words, if $(S)$ is Hopf, then, in $\gs$:
$$\xymatrix{i\ar[r] \ar[d]&j\ar[d]\\l&k}\hspace{.5cm}\Longrightarrow \hspace{.5cm}
\xymatrix{i\ar[r] \ar[d]&j\ar[d]\\l\ar[r]&k}\hspace{.5cm}  \mbox{or} \hspace{.5cm} \xymatrix{i\ar[r] \ar[d] \ar[dr]&j\ar[d]\\l&k}.$$
A special case is given by $i=k$:
$$\xymatrix{i\ar@{<->}[r] \ar[d]&j\\l&}\hspace{.5cm}  \Longrightarrow \hspace{.5cm}\xymatrix{i\ar@{<->}[r] \ar@{<->}[d]&j\\l&}.$$

\subsection{Symmetric Hopf SDSE}

\begin{prop} \label{37}
Let $(S)$ be a Hopf SDSE, such that $\gs$ is a $N$-multicycle with $N \geq 3$. Then $(S)$ is a multicyclic SDSE.
\end{prop}

\begin{proof} Let $I=I_{\overline{1}}\cup \cdots \cup I_{\overline{N}}$ be the partition of the set of vertices of the multicycle $\gs$.
As $N \geq 3$, for all $i\in I$, by lemma \ref{28} with $i=j$:
$$F_i=1+\sum_{i \longrightarrow j} a^{(i)}_j h_j.$$
Let $j,j' \in I_{\overline{m}}$. Then any $i \in I_{\overline{m-1}}$ is a direct ascendant of $j$ and $j'$. By proposition \ref{18}-3, $F_j=F_{j'}$.
In particular, for $k \in I_{\overline{m+1}}$, $a^{(j)}_k=a^{(j')}_k$. We apply the change of variables sending $h_k$ to
$\frac{1}{a^{(j)}_k}h_k$ if $k\in I_{\overline{m+1}}$, where $j$ is any element of $I_{\overline{m}}$. Then, for any $j \in I_{\overline{m}}$:
$$F_j=1+\sum_{k\in I_{\overline{m+1}}}h_k.$$
So $(S)$ is multicyclic. \end{proof}

\begin{prop} \label{38}
Let $(S)$ be a Hopf SDSE, such that $\gs$ is $M$-quasi-complete graph ($M\geq 2$). Then $(S)$ is a $2$-multicyclic or a quasi-complete SDSE. 
\end{prop}

\begin{proof}  First, let us choose two vertices $x\rightarrow y$ in $\gs$. Then $y\rightarrow x$ in $\gs$, and by proposition \ref{16},
$\lambda_2^{(y,y)}a_{\tddeux{$y$}{$x$}}=a_{\tdtroisdeux{$y$}{$x$}{$y$}}+a_{\tdtroisun{$y$}{$x$}{$y$}}$, so $\lambda_2^{(y,y)}a^{(y)}_x=a^{(y)}_x a^{(x)}_y+0$,
and $a_y^{(x)}=\lambda_2^{(y,y)}$ depends only on $y$. So, up to a change of variables, we can suppose that all the $a_y^{(x)}$'s are equal to $0$ or $1$.
We first study three preliminary cases.\\

{\it First preliminary case}. Let us assume that $\gs=1\longleftrightarrow 2$. We put:
$$F_1(h_2)=\sum_{i=0}^\infty a_i h_2^i,\hspace{1cm} F_2(h_1)=\sum_{i=0}^\infty b_i h_1^i,$$
with $a_1=b_1=1$. Then $\lambda_3^{(1,1)}=\lambda_3^{(1,1)}a_{\tdtroisdeux{$1$}{$2$}{$1$}}=2a_{\tdquatrequatre{$1$}{$2$}{$1$}{$1$}}=2b_2$. 
On the other hand, $\lambda_3^{(1,1)}a_{\tdtroisun{$1$}{$2$}{$2$}}=2a_{\tdquatretrois{$1$}{$2$}{$1$}{$2$}}$, so $2a_2b_2=2a_2$: $a_2=0$ or $b_2=1$.
Similarly, $b_2=0$ or $a_2=1$. So $a_2=b_2=0$ or $1$. In the first case, $F_1(h_2)=1+h_2$ and $F_2(h_1)=1+h_1$.
In the second case, let us apply lemma \ref{17}-1 with $(i_1,\cdots,i_n)=(1,2,1,2,\cdots)$. If $n=2k$ is even, we obtain $\lambda_n^{(1,2)}=2+2(k-1)=2k=n$. 
If $n=2k+1$ is odd, $\lambda_n^{(1,2)}=1+2k=n$. So $\lambda_n^{(1,2)}=n$ for all $n \geq 1$. By proposition \ref{19}-1, for all $n \geq 1$,
$a_{n+1}=a_n$. So for all $n \geq 0$, $a_n=1$ and $F_1(h_2)=(1-h_2)^{-1}$. Similarly, $F_2(h_1)=(1-h_1)^{-1}$.\\

{\it Second preliminary case.} Let us suppose that $\gs$ is the following graph (which is $3$-quasi-complete):
$$\xymatrix{1\ar@{<->}[rr]\ar@{<->}[rd]&&2\ar@{<->}[ld]\\ &3&}$$
We put:
$$\left\{ \begin{array}{rcl}
F_1(h_2,h_3)&=&1+h_2+h_3+a_2h_2^2+a_3h_3^2+a'h_2h_3+\mathcal{O}(h^3),\\
F_2(h_1,h_3)&=&1+h_1+h_3+b_1h_1^2+b_3h_3^2+b'h_1h_3+\mathcal{O}(h^3),\\
F_3(h_1,h_2)&=&1+h_1+h_2+c_1h_1^2+c_2h_2^2+c'h_1h_2+\mathcal{O}(h^3).
\end{array}\right.$$
By restriction, using the first preliminary case, restricting to $\{1,2\}$, $\{1,3\}$ and $\{2,3\}$,$a_2=b_1$, $a_3=c_1$ and $b_3=c_2$ and 
all these elements are in $\{0,1\}$. Moreover, by proposition \ref{16}, $\lambda_2^{(1,2)}a_{\tddeux{$1$}{$2$}}=2a_{\tdtroisun{$1$}{$2$}{$2$}}$, 
so $\lambda_2^{(1,2)}=2a_2$. On the other hand, $\lambda_2^{(1,2)}a_{\tddeux{$1$}{$3$}}=a_{\tdtroisdeux{$1$}{$3$}{$2$}}+a_{\tdtroisun{$1$}{$3$}{$2$}}$,
so $\lambda_2^{(1,2)}=1+a'$. Hence, $1+a'=2a_2$. By symmetry, we obtain $1+a'=2a_3$, so $a_2=a_3$.
Similarly, $b_1=b_3$ and $c_1=c_2$, so $a_2=a_3=b_1=b_3=c_1=c_2=0$ or $1$.

If they are all equal to $0$, then $a'=-1$. Then $\lambda_3^{(3,1)}a_{\tdtroisdeux{$3$}{$1$}{$2$}}=a_{\tdquatrecinq{$3$}{$1$}{$2$}{$1$}}$,
so $\lambda_3^{(3,1)}=1$. Moreover, $\lambda_3^{(3,1)}a_{\tdtroisdeux{$3$}{$2$}{$1$}}=a_{\tdquatretrois{$3$}{$2$}{$1$}{$1$}}$,
so $\lambda_3^{(3,1)}=-1$: this is a contradiction, so $a_2=a_3=b_1=b_3=c_1=c_2=1$, and $a'=1$. Similarly, $b'=1$ and $c'=1$.
As in the first preliminary case, using lemma \ref{17}-1, we prove that $\lambda_n^{(i,j)}=n$ if $i\neq j$ for all $n \geq 1$, and then that
$F_1(h_2,h_3)=(1-h_2)^{-1}(1-h_3)^{-1}$. Similarly, $F_2(h_1,h_3)=(1-h_1)^{-1}(1-h_3)^{-1}$ and $F_3(h_1,h_2)=(1-h_1)^{-1}(1-h_2)^{-1}$.\\

{\it Third preliminary case.} We now consider the $2$-quasi-complete graph with three vertices $1\longleftrightarrow 2 \longleftrightarrow 3$.
Then $I_1=\{1,3\}$ and $I_2=\{2\}$. We put:
$$F_2(h_1,h_3)=1+h_1+h_3+a_{(2,0)}h_1^2+a_{(0,2)}h_3^2+a_{(1,1)}h_1h_3+\mathcal{O}(h^3).$$
Restricting to $\{1,2\}$, by the first preliminary case, we obtain $F_1(h_2)=1+h_2$ or $F_1(h_2)=(1-h_2)^{-1}$.
\begin{enumerate}
\item Let us assume that $F_1(h_2)=1+h_2$. Then by the first case, $F_2(h_1,0)=1+h_1$, so $a_{(2,0)}=0$. Moreover, 
$\lambda_2^{(2,1)} a_{\tddeux{$2$}{$1$}}=0$, so $\lambda_2^{(2,1)}a_{\tddeux{$2$}{$3$}}=a_{\tdtroisun{$2$}{$3$}{$1$}}$: $a_{(1,1)}=0$.
Then $\lambda_2^{(2,3)} a_{\tddeux{$2$}{$1$}}=a_{\tdtroisun{$2$}{$3$}{$1$}}$, so $\lambda_2^{(2,3)}=a_{(1,1)}=0$, and 
$\lambda_2^{(2,3)}a_{\tddeux{$2$}{$3$}}=2a_{\tdtroisun{$2$}{$3$}{$3$}}$: $a_{(0,2)}=0$. As a consequence, $F_2(h_1,h_3)=1+h_2+h_3$. 
Restricting to $2\longleftrightarrow 3$, by the first point, $F_3(h_2)=1+h_2$. 
\item Let us assume that $F_1(h_2)=(1-h_2)^{-1}$. Then $F_2(h_1,0)=(1-h_2)^{-1}$ by the first point, so $a_{(0,2)}=1$. By the first preliminary case, 
this implies that $F_2(0,h_3)=(1-h_3)^{-1}$ and $F_3(h_2)=(1-h_2)^{-1}$. Similarly with the first case, we prove that
$\lambda_n^{(2,i)}=n$ if $i=1$ or $3$ for all $n \geq 1$. By proposition \ref{19}-1:
$$a_{(m+1,n)}=\frac{m+n+1}{m+1}a_{(m,n)},\hspace{1cm} a_{(m,n+1)}=\frac{m+n+1}{n+1}a_{(m,n)}.$$
An easy induction proves that $a_{(m,n)}=\binom{m+n}{m}$ for all $m,n$, so $F_2(h_1,h_3)=(1-h_1-h_3)^{-1}$.
\end{enumerate}

We separate the proof of the general case into two subcases.\\

{\it General case, first subcase}.  $M=2$. We put $I_1=\{x_1,\cdots,x_r\}$ and $I_2=\{y_1,\cdots,y_s\}$. For $x_i \in I_1$, we put:
$$F_{x_p}=\sum_{(q_1,\cdots,q_s)} a^{(x_p)}_{(q_1,\cdots,q_s)} h_{y_1}^{q_1}\cdots h_{y_s}^{q_s}.$$
 Restricting to the vertices $x_p$ and $y_q$, by the first preliminary case,  two cases are possible.
\begin{enumerate}
\item $a^{(x_p)}_{y_q,y_q}=0$. Then, by the third preliminary case, restricting to $x_p$, $y_q$ and $y_{q'}$, for all $y_q$, $y_{q'}$,
$a^{(x_p)}_{y_q,y_{q'}}=0$. So:
$$F_{x_p}=1+\sum_q h_{y_q}.$$
\item $\lambda_n^{(x_p,y_q)}=n$ for all $n \geq 1$. Using proposition \ref{19}-1, we obtain:
$$a^{(x_p)}_{(q_1,\cdots,q_m+1,\cdots,q_s)}=\frac{1+q_1+\cdots+q_s}{q_m+1}a^{(x_p)}_{(q_1,\cdots,q_s)}.$$
An easy induction proves:
$$a^{(x_p)}_{(q_1,\cdots,q_s)}=\frac{(q_1+\cdots+q_s)!}{q_1!\cdots q_s!}.$$
So:
$$F_{x_p}=\left(1-\sum_{q} h_{y_q}\right)^{-1}.$$
\end{enumerate}
A similar result holds for the $y_q$'s. So, we prove that for any vertex $i$ of $\gs$, one of the following holds:
\begin{enumerate}
\item $\displaystyle F_i=1+\sum_{i \longrightarrow j} h_j$.
\item $\displaystyle F_i=\left(1-\sum_{i \longrightarrow j} h_j\right)^{-1}$.
\end{enumerate}
Moreover, by the first preliminary case, if $i$ and $j$ are related, they satisfy both (a) or both (b). As the graph is connected, every vertex satisfies (a) 
or every vertex satisfies (b). \\

{\it General case, second subcase}.  $M\geq 3$. Let us fix $i \in G$ and let us denote $y_1,\cdots,y_q$ its direct descendants.
Restricting to the vertices $i$ and $y_j$, two cases are possible.
\begin{enumerate}
\item $a^{(i)}_{y_j,y_j}=0$. As $M \geq 3$, with a good choice of $y_{j'}$, we can restrict to the second preliminary case, 
and we obtain $a^{(i)}_{y_j,y_j}=1$: contradiction. So this case is impossible.
\item $\lambda_n^{(x,y_j)}=n$ for all $n \geq 1$.
Using proposition \ref{19}-1, we obtain, similarly with the case $M=2$, if $i \in I_p$:
$$F_i=\prod_{q\neq p} \left(1-\sum_{l\in h_q} h_l \right)^{-1}.$$
\end{enumerate} 
So $(S)$ is quasi-complete. \end{proof}

\begin{defi}\textnormal{\begin{enumerate}
\item Let $G$ be a graph. We shall say that $G$ is {\it symmetric} if it has only non-self-dependent vertices and if, for $i\neq j$,
there is an edge from $i$ to $j$ if, and only if, there is an edge from $j$ to $i$.
\item Let $(S)$ be an SDSE. We shall say that $(S)$ is {\it symmetric} if $\gs$ is symmetric.
\end{enumerate}}\end{defi}

\begin{theo}
Let $(S)$ be a connected symmetric Hopf SDSE. Then $(S)$ is $2$-multicyclic or quasi-complete.
\end{theo}

\begin{proof} By proposition \ref{38}, it is enough to prove that $\gs$ is a $M$-quasi-complete graph, with $M \geq 2$.
Let us consider a maximal quasi-complete subgraph $G'$ of $\gs$. This exists, as $\gs$ contains quasi-complete subgraphs
(for example, two related vertices). Let us assume that $G'\neq \gs$. As $\gs$ is connected, there exists a vertex $i \in \gs$,
related to a vertex of $G'$. Let us put $I'=I'_1\cup \cdots I'_M$ be the partition of the set of vertices of $G'$.

First, if $i$ is related to a vertex $j$ of $I'_p$, it is related to any vertex of $I'_p$. Indeed, let $j'$ be another vertex of $I'_p$ and let $k \in I'_q$, $q \neq p$. 
By lemma \ref{36}, $j'$ is related to $i$. As $\gs$ is symmetric, $i$ is related to $j'$.

Let us assume that $i$ is not related to at least two $I_p$'s. Let us take $k$, $l$ in $G'$, in two different $I_p$'s,
not related to $i$. By the first step, $j$, $k$ and $l$ are in different $I_p$'s, so are related. By lemma \ref{36}, $k$ or $l$ is related to $i$. 
As $\gs$ is symmetric, then $i$ is related to $k$ or $l$: contradiction. So $i$ is not related to at most one $I_p$'s. 

As a conclusion:
\begin{enumerate}
\item If $i$ is related to every $I_p$'s, by the first step $i$ is related to every vertices of $G'$, so $G'\cup\{i\}$ is an $M+1$-quasi-complete graph, 
with partition $I_1\cup \cdots \cup I_M \cup\{x\}$: this contradicts the maximality of $G'$.
\item If $i$ is related to every $I_p$'s but one, we can suppose up to a reindexation that $i$ is not related to $I_M$. 
Then, by the first step, $i$ is related to every vertices of $I_1\cup \cdots \cup I_{M-1}$. So $G'\cup\{x\}$ is an $M$-quasi-complete graph, 
with partition $I_1\cup \cdots \cup (I_M \cup\{x\})$: this contradicts the maximality of $G'$.
\end{enumerate}
In both cases, this is a contradiction, so $\gs=G'$ is quasi-complete.
\end{proof}

\subsection{Formal series of a self-dependent vertex}

Let $(S)$ be a Hopf SDSE, and let us assume that $i$ is a self-dependent vertex of $\gs$. 
Up to a change of variables, we can suppose that $a_j^{(i)}=0$ or $1$ for all $j$. In particular, we assume that $a_i^{(i)}=1$. 

\begin{lemma}\label{41}
Under these hypotheses, $i$ is of level $0$ and for all $j\in I$, $b_j=(1+\delta_{i,j})a_{i,j}^{(i)}$.
\end{lemma}

\begin{proof} We apply lemma \ref{17}-1, with $i_k=i$ for all $i$. We obtain, for all $n \geq 1$:
$$\lambda_n^{(i,j)}=a_j^{(i)}+(1+\delta_{i,j})(n-1)\frac{a_{i,j}^{(i)}}{a_i^{(i)}}.$$
So this proves the assertion. \end{proof}\\

{\bf Remark.} So all the descendants of $i$ are also of level $0$.
 
 \begin{lemma} \label{42}
 Under the former hypotheses, there exists a partition $I=I_1\cup \cdots \cup I_M \cup J$ ($J$ eventually empty), with $i\in I_1$, such that
 the coefficients $a_j^{(k)}$ are given in the following array:
 $$\begin{array}{c|c|c|c|c|c|c}
j\setminus k&I_1&I_2&I_3&\cdots&I_M&J\\
\hline I_1&1&\beta_1+1&\cdots&\cdots&\beta_1+1&*\\
\hline I_2&\vdots&1-\beta_2&1&\cdots&1&\vdots\\
\hline I_3&\vdots&1&1-\beta_3&\ddots&\vdots&\vdots\\
\hline \vdots&\vdots&\vdots&\ddots&\ddots&1&\vdots\\
\hline I_M&1&1&\cdots&1&1-\beta_M&\vdots\\
\hline J&0&\cdots&\cdots&\cdots&0&*
\end{array}$$
Moreover, for all $j\in I_1$:
$$F_j=\prod_{p=1}^M f_{\beta_p}\left(\sum_{l\in I_p}h_l \right).$$
Finally, the coefficients $\lambda_n^{(j,k)}$ are given by $\lambda_n^{(j,k)}=b_k(n-1)+a^{(j)}_k$ for all $n\geq 1$ with:
 $$\begin{array}{c|c|c|c|c|c}
k&I_1&I_2&\cdots&I_M&J\\
\hline b_k&\beta_1+1&1&\cdots&1&0
\end{array}$$
 \end{lemma}
 
 \begin{proof} We can apply lemma \ref{26} with $\lambda_j=a^{(i)}_j$  and $\mu_j^{(l)}=-a_j^{(l)}+\left(1+\delta_{i,j}\right)a^{(i)}_{i,j}$. 
Then $I=I_1\cup \cdots I_M \cup J$, such that $-a_j^{(k)}+\left(1+\delta_{i,j}\right)a^{(i)}_{i,j}$ is given for all $j,k$ by the array:
$$\begin{array}{c|c|c|c|c|c}
j\setminus k&I_1&I_2&\cdots&I_M&J\\
\hline I_1&\beta_1&0&\cdots&0&*\\
\hline I_2&0&\beta_2&\ddots&\vdots&\vdots\\
\hline \vdots&\vdots&\ddots&\ddots&0&\vdots\\
\hline I_M&0&\cdots&0&\beta_M&\vdots\\
\hline J&0&\cdots&\cdots&0&*
\end{array}$$
We assume that $i\in I_1$, without loss of generality. For the row $j\in J$, the result comes from the following observation: 
let $j,k\in I$ such that $a^{(i)}_j=0$ and $a^{(i)}_k\neq 0$, then, by proposition \ref{19}-1:
$$a^{(i)}_{j,k}=\left(a^{(i)}_j-a_j^{(k)}+a^{(i)}_{i,j}\right)a^{(i)}_k=0.$$
As $a^{(i)}_j=0$, then $a^{(i)}_{i,j}=0$, so $a_j^{(k)}=0$. 

Lemma \ref{26} also gives:
$$F_i=\prod_{p=1}^k f_{\beta_p}\left(\sum_{l\in I_p}h_l \right).$$
So $(1+\delta_{i,j})a^{(i)}_{i,j}=\beta_1+1$ if $j\in I_1$, $1$ if $j\in I_2\cup \dots \cup I_M$, and $0$ if $j\in J$. 
So $a_j^{(k)}$ is given by  for all $j,k$ by the indicated array. We obtain in lemma \ref{41} that:
$$b_k=\left\{ \begin{array}{l}
\beta_1+1\mbox{ if }k\in I_1,\\
1\mbox{ if }k\in I_2\cup \cdots \cup I_M,\\
0\mbox{ if }k\in J.
\end{array}\right.$$

As a conclusion, if $j\in I_1$, then for all $1\leq k\leq N$, $a^{(j)}_k=a^{(i)}_k$ and $\lambda_n^{(j,k)}=\lambda_n^{(i,k)}$ for all $n \geq 1$.
By proposition \ref{19}, $F_i=F_j$. \end{proof}

\subsection{Hopf SDSE generated by self-dependent vertices}

\begin{prop}
Let $(S')$ be a Hopf SDSE, and let $i$ be a self-dependent vertex of $G_{(S')}$. Let $(S)$ be the restriction of $(S')$ to $i$ and all its descendants.
Then $(S)$ is fundamental, with $K_0=I_1=J_1=\emptyset$.
\end{prop}
 
\begin{proof} We use the notations of lemma \ref{42}. Note that if $i,j$ are in the same $I_k$, then $\lambda_n^{(i,k)}=\lambda_n^{(j,k)}$ for all $n \geq 1$, 
for all $k\in I$. So, by proposition \ref{18}-2 the Hopf SDSE formed by $i$ and its descendant is the dilatation of a system with the following 
coefficients $\lambda_n^{(j,k)}$:
$$\begin{array}{c|c|c|c|c|c}
j\setminus k&1&2&3&\cdots&M\\
\hline 1&(\beta_1+1)(n-1)+1&n&\cdots&\cdots&n\\
\hline 2&(\beta_1+1)n&n-\beta_2&n&\cdots&n\\
\hline 3&\vdots&n&n-\beta_3&\ddots&\vdots\\
\hline \vdots&\vdots&\vdots&\ddots&\ddots&n\\
\hline M&(\beta_1+1)n&n&\cdots&n&n-\beta_M
\end{array}$$
with $i=1$. We already proved in lemma \ref{42} that:
$$F_1=\prod_{j=1}^Mf_{\beta_j}(h_j).$$
If $j\neq 1$, for all $(k_1,\cdots,k_M)$:
\begin{eqnarray*}
a^{(j)}_{(k_1+1,\cdots,k_M)}&=&\left((\beta_1+1)\sum_{l=1}^M k_l+\beta_1+1-(\beta_1+1)\sum_{l=1}^M k_l-k_1\right)\frac{a^{(j)}_{(k_1,\cdots,k_M)}}{k_1+1}\\
&=&(\beta_1+1+\beta_1k_1)\frac{a^{(j)}_{(k_1,\cdots,k_M)}}{k_1+1},\\ \\
a^{(j)}_{(k_1,\cdots,k_j+1,\cdots,k_M)}&=&\left(\sum_{l=1}^M k_l+1-\beta_j-\sum_{l=1}^M k_l+\beta_j k_j\right)
\frac{a^{(j)}_{(k_1,\cdots,k_M)}}{k_j+1}\\
&=&(1-\beta_j+\beta_j k_j)\frac{a^{(j)}_{(k_1,\cdots,k_M)}}{k_j+1}.
\end{eqnarray*}
If $l\neq 1$ and $l \neq j$:
$$a^{(j)}_{(k_1,\cdots,k_l+1,\cdots,k_M)}=\left(\sum_{l=1}^M k_l -\sum_{l=1}^M k_l+\beta_l k_l\right)\frac{a^{(j)}_{(k_1,\cdots,k_M)}}{k_l+1}
=(1+\beta_l k_l)\frac{a^{(j)}_{(k_1,\cdots,k_M)}}{k_l+1}.$$
So, if $j\neq 1$:
$$F_j=f_{\frac{\beta_1}{1+\beta_1}}((1+\beta_1)h_1)f_{\frac{\beta_j}{1-\beta_j}}((1-\beta_j)h_j) \prod_{k\neq 1,j} f_{\beta_k}(h_k).$$
Let us put $I'_0=\{j\geq 2\:/\:\beta_j\neq 1\}$ and $J'_0=\{j\geq 2\:/\: \beta_j=1\}$. Then, after the change of variables 
$h_j\longrightarrow \frac{1}{1-\beta_j}h_j$ for all $j \in I'_0$:
$$ \left\{ \begin{array}{rcl}
F_1&=&\displaystyle f_{\beta_1}(h_1) \prod_{j\in I'_0} f_{\beta_j}\left(\frac{1}{1-\beta_j}h_j\right) \prod_{j\in J'_0} f_1(h_j),\\
F_j&=&\displaystyle f_{\frac{\beta_1}{1+\beta_1}}((1+\beta_1)h_1)f_{\frac{\beta_j}{1-\beta_j}}(h_j)
 \prod_{j\in I'_0-\{j\}} f_{\beta_j}\left(\frac{1}{1-\beta_j}h_j\right) \prod_{j\in J'_0} f_1(h_j) \mbox{ if }j\in I'_0,\\
 F_j&=&\displaystyle f_{\frac{\beta_1}{1+\beta_1}}((1+\beta_1)h_1)
\prod_{j\in I'_0} f_{\beta_j}\left(\frac{1}{1-\beta_j}h_j\right) \prod_{j\in J'_0-\{j\}} f_1(h_j) \mbox{ if }j\in J'_0.
\end{array}\right.$$
Putting $\gamma_j=\frac{\beta_j}{1-\beta_j}$ for all $j\in I_0$, then, as $\beta_j=\frac{\gamma_j}{1+\gamma_j}$ and $1-\beta_j=\frac{1}{1+\gamma_j}$:
$$ \left\{ \begin{array}{rcl}
F_1&=&\displaystyle f_{\beta_1}(h_1) \prod_{j\in I'_0} f_{\frac{\gamma_j}{1+\gamma_j}}\left((1+\gamma_j)h_j\right) \prod_{j\in J'_0} f_1(h_j),\\
F_j&=&\displaystyle f_{\frac{\beta_1}{1+\beta_1}}((1+\beta_1)h_1)f_{\gamma_j}(h_j)
 \prod_{j\in I'_0-\{j\}} f_{\frac{\gamma_j}{1+\gamma_j}}\left((1+\gamma_j)h_j\right) \prod_{j\in J'_0} f_1(h_j) \mbox{ if }j\in I'_0,\\
 F_j&=&\displaystyle f_{\frac{\beta_1}{1+\beta_1}}((1+\beta_1)h_1)
\prod_{j\in I'_0} f_{\frac{\gamma_j}{1+\gamma_j}}\left((1+\gamma_j)h_j\right) \prod_{j\in J'_0-\{j\}} f_1(h_j) \mbox{ if }j\in J'_0.
\end{array}\right.$$
So this a fundamental system, with $I_0=\{1\} \cup I'_0$ and $J_0=J'_0$. \end{proof}
 
 \begin{cor} 
Let $(S)$ be a connected Hopf SDSE such that any vertex of $\gs$ is the descendant of a self-dependent vertex. Then $(S)$ is fundamental,
with $K_0=I_1=J_1=\emptyset$.
\end{cor}
 
\begin{proof} Let $x$ be a self-dependent vertex of $(S)$. Then the system formed by $x$ and its descendants  is fundamental. We then put 
$I_0^{(x)}$ and $J_0^{(x)}$  the partition of the set formed by $x$ and its descendants. We separate $I_0^{(x)}$ into two parts:
$$I_{0,1}=\left\{y\in I_0^{(x)}\:/\beta_y \neq -1\right\},\:I_{0,2}=\left\{y\in I_0^{(x)}\:/\beta_y =-1\right\}.$$
Then, after elimination of an eventual dilatation by restriction, the direct descendants of $x \in I_{0,2}^{(x)}$  are $x$, the elements of $I^{(x)}_{0,1}$ 
and $J_0^{(x)}$; the direct descendants of $x \in I_{0,1}^{(x)}$  are the elements of $I^{(x)}_{0,1}$ and $J_0^{(x)}$; the direct descendants of $x \in J_0^{(x)}$
are the elements of $I^{(x)}_{0,1}$ and the elements of $J_0^{(x)}$ except $x$. Let us consider the following cases:
\begin{enumerate}
\item If there exists a vertex $x$, such that $J_0^{(x)}\neq \emptyset$, then, as $\gs$ is connected, for any self-dependent vertex $y$, $J_0^{(y)}=J_0^{(x)}$.
As a consequence, for any self-dependent vertex $y$, $I_{0,1}^{(x)}=I_{0,1}^{(y)}$. We then deduce that $(S)$ is fundamental, with $J_0=J_0^{(x)}$ 
for any self-dependent vertex $x$.
\item If for any self-dependent vertex $x$, $J_0^{(x)}=\emptyset$, and if there is a self-dependent vertex $x$ such that $I^{(x)}_{0,2}\neq \emptyset$, 
then by connectivity of $\gs$, for any self-dependent vertex $y$, $I^{(y)}_{0,2}=I^{(x)}_{0,2}$ and $I_{0,1}^{(y)}=\{y\}$, or $I^{(y)}_{0,2} $ is empty 
if $y \in I^{(x)}_{0,2}$. Then $(S)$ is a fundamental, with $J_0=\emptyset$.
\item If for any self-dependent vertex $x$, $J_0^{(x)}=\emptyset=I_{0,2}^{(x)}$. Then by connectivity, $I=I_{0,1}^{(x)}$ for any self-dependent vertex. 
So $(S)$ is fundamental, with $J_0=\emptyset$.
\end{enumerate} 
In all cases, $(S)$ is fundamental. \end{proof}

\section{The structure theorem of Hopf SDSE}

\subsection{Connecting vertices}

\begin{defi}\textnormal{Let $(S)$ be an SDSE and let $i \in \gs$.
\begin{enumerate}
\item We denote by $\gs^{(i)}$ is the subgraph of $\gs$ formed by $i$ and all its descendants.
\item The vertex $i$ is a {\it connecting vertex} of $\gs$ if $\gs^{(i)}-\{i\}$ is not connected.
\end{enumerate}} \end{defi}

\begin{lemma}\label{46}
Let $(S)$ be a Hopf SDSE and let $i \in \gs$ be a connecting vertex. Then ($i$ is the descendant of a self-dependent vertex) or 
($i$ belongs to a symmetric subgraph of $\gs$) or ($i$ is not self-dependent and relates several components of a non-connected fundamental SDSE).
\end{lemma}

\begin{proof} {\it First step.} If $i$ is self-dependent, it is a descendant of itself and the conclusion holds. Let us assume that $i$ is not self-dependent. 
Let $G_1,\cdots,G_M$ be the connected components of $\gs^{(i)}-\{i\}$ ($M\geq 2$). Let $x_p\in G_p$ be a direct descendant of $i$ for all $p$.
Let $x'_p$ be a direct descendant of $x_p$. Then $x'_p\in G_p$. Choosing $q\neq j$ and applying lemma \ref{36}, there is an edge from $i$ to $x'_p$. 
Iterating this process, we deduce that any vertex of $\gs^{(i)}-\{i\}$ is a direct descendant of $i$. If $i$ is the direct descendant of a vertex $j\in \gs^{(i)}-\{i\}$, 
then $i$ is included in the symmetric subgraph $i \longleftrightarrow j$ of $\gs^{(i)}$, so the conclusion holds. \\

{\it Second step.} Let us now assume that $i$ is not the direct descendant of any $j \in \gs^{(i)}-\{i\}$. Let $n \geq 2$, $j\in G_p$, 
and let $i\rightarrow x_2\rightarrow \cdots \rightarrow x_n$ in $\gs^{(i)}$, where $x_2,\cdots,x_n \in G_q$, $p\neq q$. Then, as $i$ is not related to any $x_l$,
$\lambda_n^{(i,j)}a_{l(i,x_2,\cdots,x_n)}^{(i)}=a_{B^+_i(\tdun{$j$}l(x_2,\cdots,x_n))}$, so $\lambda_n^{(i,j)}=\frac{a^{(i)}_{j,x_2}}{a^{(i)}_{x_2}}$, 
and $\lambda_n^{(i,j)}$ does not depend on $n$: we put $\lambda_n^{(i,j)}=\lambda_j$ for all $j \in G-\{i\}$, $n \geq 2$.
In other words, $i$ has level $\leq 1$, and $b_j=0$ for all $j$.\\

{\it Third step.} In order to simplify the writing of the proof, up to a reindexation, we shall suppose that $i=0$ and the vertices of $\gs^{(0)}-\{0\}$ are the
elements of $\{1,\cdots,N\}$. By a change of variables, we can suppose that $a^{(0)}_j=1$ for all $1\leq j\leq N$. By the second step, we can use lemma 
\ref{27}, with $\mu_j^{(l)}=-a_j^{(l)}$ for all $1\leq j,l\leq N$ and $\lambda_j=a^{(0)}_{j,k}$ for all $j,k$ in two different connected components of $\gs^{(0)}-\{0\}$.
\begin{enumerate}
\item In the first case, we obtain the following values for $a_j^{(k)}$ and $\lambda_j$:
$$\begin{array}{c|c|c|c|c|c}
j\setminus k&I_1&I_2&\cdots&I_M&J\\
\hline I_1&-\nu\beta_1&0&\cdots&0&-\nu\\
\hline I_2&0&-\nu \beta_2&\ddots&\vdots&\vdots\\
\hline \vdots&\vdots&\ddots&\ddots&0&\vdots \\
\hline I_M&0&\cdots&0&-\nu\beta_M&-\nu\\
\hline J&0&\cdots&\cdots&0&0
\end{array}$$
$$\begin{array}{c|c|c|c|c}
j&I_1&\cdots&I_M&J\\
\hline \lambda_j&\nu&\cdots&\nu&0
\end{array}$$
As there are no vertices with no descendants, necessarily $\nu \neq 0$ and $\beta_p\neq 0$ for all $p$.
For the same reason, $I_1\cup \cdots \cup I_M=\emptyset$ is impossible.
If $J \neq \emptyset$, then any vertex of $J$ is related to every vertex of $I_1\cup \cdots \cup I_M$,
so $\gs^{(0)}-\{0\}$ is connected: impossible, as $0$ is a connected vertex. So $J=\emptyset$,
and $0$ connects several totally self-dependent subgraphs.
\item In the second case, we obtain the following values for $a_j^{(k)}$ and $\lambda_j$:
$$\begin{array}{c|c|c|c|c|c}
j\setminus k&I_1&I_2&\cdots&I_M&J\\
\hline I_1&-\nu_1&0&\cdots&0&0\\
\hline I_2&0&-\nu_2&\ddots&\vdots&\vdots\\
\hline \vdots&\vdots&\ddots&\ddots&0&\vdots \\
\hline I_M&0&\cdots&0&-\nu_M&0\\
\hline J&0&\cdots&\cdots&0&0
\end{array}$$
$$\begin{array}{c|c|c|c|c}
j&I_1&\cdots&I_M&J\\
\hline \lambda_j&0&\cdots&0&0
\end{array}$$
As there are no vertices with no descendants, $J=\emptyset$ and $\nu_l \neq 0$ for all $l$.
\end{enumerate} 
Moreover, as $b_j=1+\beta_j=0$ for all $j\geq 1$, $0$ connects several components of a non-connected fundamental SDSE. \end{proof}

\subsection{Structure of connected Hopf SDSE}

\begin{lemma} \label{47}
Let $(S)$ be a Hopf SDSE containing a multicycle with set of vertices $I=I_{\overline{1}}\cup \cdots \cup I_{\overline{M}}$,
Then any non-self-dependent vertex of $\gs$ has direct descendants in at most one $I_{\overline{k}}$.
\end{lemma}

\begin{proof} Let us assume that the vertex $0$ of $\gs$ have a direct descendant $x \in I_{\overline{k}}$ and  $y\in I_{\overline{l}}$
with $\overline{k}\neq \overline{l}$. Then lemma \ref{36} implies that any direct descendant of $x$ is a direct descendant of $0$,
so $0$ has also a direct descendant in $I_{\overline{k+1}}$. Similarly, $0$ has a direct descendant in $I_{\overline{l+1}}$. 
 Iterating this process, $0$ has direct descendants in all the $I_{\overline{i}}$'s.
Up to a restriction, the situation is the following:
$$\xymatrix{0\ar[d]\ar[rd]\ar[rrd] \ar[rrrrd]&&&&\\
1\ar[r]&2\ar[r]&3\ar[r]&\cdots\ar[r]&N\ar@/^2pc/[llll]\\
&&&&}$$
Moreover, for all $1\leq i \leq k$, $F_i(h_{i+1})=1+h_{i+1}$, with the convention $h_{N+1}=h_1$. 

We first assume $M\geq 3$.  In order to ease the notation, we do not write the index $^{(0)}$ in the sequel of the proof.
By proposition \ref{16}, $\lambda_2^{(0,2)}a_{\tddeux{$0$}{$1$}}=a_{\tdtroisdeux{$0$}{$1$}{$2$}}+a_{\tdtroisun{$0$}{$2$}{$1$}}$, so
$\lambda_2^{(0,2)}=1+\frac{a_{1,2}}{a_1}$. On the other hand, $\lambda_2^{(0,2)}a_{\tddeux{$0$}{$2$}}=2a_{\tdtroisun{$0$}{$2$}{$2$}}$, 
so $\lambda_2^{(0,2)}=2\frac{a_{2,2}}{a_2}$. Hence:
$$1+\frac{a_{1,2}}{a_1}=2\frac{a_{2,2}}{a_2}.$$
Moreover, $\lambda_3^{(0,2)}a_{\tdtroisdeux{$0$}{$2$}{$3$}}=a_{\tdquatretrois{$0$}{$2$}{$3$}{$2$}}$, so $\lambda_3^{(0,2)}=2\frac{a_{2,2}}{a_2}$.
On the other hand, $\lambda_3^{(0,2)}a_{\tdtroisdeux{$0$}{$1$}{$2$}}=a_{\tdquatretrois{$0$}{$1$}{$2$}{$2$}}$, so $\lambda_3^{(0,2)}=\frac{a_{1,2}}{a_1}$.
Hence:
$$\frac{a_{1,2}}{a_1}=2\frac{a_{2,2}}{a_2}=1+\frac{a_{1,2}}{a_1}.$$
This is a contradiction. \\

Let us now prove the result for $N=2$. We assume that there exists a Hopf SDSE with the graph:
$$\xymatrix{&0\ar[dr] \ar[ld]&\\
1\ar@{<->}[rr]&&2}$$
and such that $F_1=1+h_2$ and $F_2=1+h_1$. We write:
$$F_0=\sum_{i,j}a_{(i,j)} h_1^i h_2^j,$$
with $a_{(1,0)}$ and $a_{(0,1)}$ non-zero. Then $\lambda_2^{(0,1)} a_{\tddeux{$0$}{$1$}}=2a_{\tdtroisun{$0$}{$1$}{$1$}}$, 
so $\lambda_2^{(0,1)}=\frac{2a_{(2,0)}}{a_{(1,0)}}$. On the other hand, $\lambda_2^{(0,1)} a_{\tddeux{$0$}{$2$}}=
a_{\tdtroisun{$0$}{$2$}{$1$}}+a_{\tdtroisdeux{$0$}{$2$}{$1$}}$, so $\lambda_2^{(0,1)}=\frac{a_{(1,1)}}{a_{(0,1)}}+1$. We obtain:
$$\frac{2a_{(2,0)}}{a_{(1,0)}}=\frac{a_{(1,1)}}{a_{(0,1)}}+1.$$
Moreover, $\lambda_3^{(0,1)} a_{\tdtroisdeux{$0$}{$1$}{$2$}}=a_{\tdquatretrois{$0$}{$1$}{$2$}{$1$}}+a_{\tdquatrecinq{$0$}{$1$}{$2$}{$1$}}$,
so $\lambda_3^{(0,1)}=\frac{2a_{(2,0)}}{a_{(1,0)}}+1$. On the other hand, $\lambda_3^{(0,1)} a_{\tdtroisdeux{$0$}{$2$}{$1$}}
=2a_{\tdquatretrois{$0$}{$2$}{$1$}{$1$}}$, so $\lambda_3^{(0,1)}=\frac{a_{(1,1)}}{a_{(0,1)}}$. So:
$$\frac{a_{(1,1)}}{a_{(0,1)}}+1=\frac{2a_{(2,0)}}{a_{(1,0)}}=\frac{a_{(1,1)}}{a_{(0,1)}}-1.$$
This is a contradiction. \end{proof}

\begin{lemma} 
Let $(S)$ be a Hopf SDSE, such that any vertex of $\gs$ has a direct ascendant. Let $i$ be a vertex of $\gs$. 
Then $(i$ is a descendant of a self-dependent vertex) or ($i$ belongs to a multicycle of $\gs$) or ($i$ belongs to a symmetric subgraph of $\gs$).
\end{lemma}

\begin{proof} Let us first prove that $i$ is the descendant of a vertex of a cycle of $\gs$. As any vertex has a direct ascendant, it is possible to
define inductively a sequence $(x_l)_{l\geq 0}$ of vertices of $\gs$, such that $x_0=i$ and $x_{l+1}$ is a direct ascendant of $x_l$ for all $l$. 
As $\gs$ is finite, there exists $0\leq l <m$, such that $x_l=x_m$. Then $x_l\leftarrow x_{l+1}\leftarrow \cdots \leftarrow x_{m-1} \leftarrow x_m=x_l$
is a cycle of $\gs$, and $i$ is a descendant of any vertex of this cycle.

Let $G'=x_1\rightarrow \cdots \rightarrow x_s \rightarrow x_1$ be a cycle such that $i$ is a descendant of a vertex of $G'$, chosen with a minimal $s$.
As $s$ is minimal, there are no edges from $x_l$ to $x_m$ in $\gs$ if $m \neq l+1$, with the convention $x_{s+1}=x_1$. The situation is the following:
$$\xymatrix{x_1\ar[r]\ar[d]&\cdots\ar[r]&x_s\ar@/_1pc/[ll]&\\
y_1\ar[r]&\cdots\ar[r]&y_{t-1}\ar[r]&i}$$
Three cases are possible:
\begin{enumerate}
\item If $s=1$, then $i$ is the descendant of a self-dependent vertex.
\item If $s=2$, the situation is the following:
$$\xymatrix{x_1\ar@{<->}[r]\ar[d]&x_2&&\\
y_1\ar[r]&\cdots\ar[r]&y_{t-1}\ar[r]&i}$$
By minimality of $s$, there are no self-dependent vertex in $\{x_1,x_2,y_1,\cdots,y_{t-1},i\}$. Applying repeatedly lemma \ref{36}, there is an edge 
from $y_1$ to $x_1$, then from $y_2$ to $y_1$, $\cdots$, then from $i$ to $y_{t-1}$. So $i$ belongs to a symmetric subgraph of $\gs$.
\item If $s \geq 3$, then the subgraph formed by $x_1,\cdots,x_s$ is a multicycle. Let $G'$ be a maximal multicycle of length $s$ of $G$, such that $i$ 
is a descendant of a vertex of $G'$. We denote by $I'$ the set of vertices of $G'$. Let us assume that $i \notin G'$. There exists 
$x_1\rightarrow y_1\rightarrow \cdots \rightarrow y_{t-1} \rightarrow y_t=i$ in $G$, with $t\leq 1$, and $x_1\in I'$. Up to a reindexation, we can assume that
$x_1 \in I'_{\overline{1}}$. By lemma  \ref{36}, $y_1$ is the direct descendant of any vertex of $I_{\overline{1}}$ and the direct ascendant
of any vertex of $I_{\overline{3}}$. By lemma \ref{47}, $y_1$ is not the direct ascendant of any vertex of $I'_{\overline{k}}$ if $\overline{k}\neq \overline{3}$.
So $I'\cup \{x\}=I'_{\overline{1}}\cup  \left(I'_{\overline{2}}\cup \{i\}\right)\cup \cdots \cup I'_{\overline{s}}$ gives a multicycle of length $s$,
such that $i$ is a descendant of a vertex of $I'\cup \{i\}$: this contradicts the maximality of $G'$. So $i \in I'$.
\end{enumerate} \end{proof}

By the preceding study of Hopf symmetric SDSE:

\begin{cor} \label{49}
Let $(S)$ be connected Hopf SDSE, such that any vertex of $\gs$ has a direct ascendant. Then (any vertex of $\gs$ is the descendant of a 
self-dependent vertex, so $(S)$ is fundamental) or ($(S)$ is quasi-complete, so $(S)$ is fundamental) or ($(S)$ is multicyclic).
\end{cor}

\begin{cor} \label{50}
Let $(S)$ be a connected Hopf SDSE. Then there exists a sequence $(G_i)_{0\leq i \leq k}$ of subgraphs of $\gs$, such that:
\begin{itemize}
\item The system $(S_0)$ associated to the $F_i$'s, $i\in G_0$, is fundamental or is multicyclic.
\item $G_k=\gs$.
\item For all $0\leq i\leq k-1$, $G_{i+1}$ is obtained from $G_i$ by adding a non-self-dependent vertex without any ascendant in $G_i$.
\end{itemize}
If $G_0$ is fundamental, any vertex is of finite level. If $G_0$ is multicyclic, no vertex is of finite level.
\end{cor}

\begin{proof} {\it First step.} Let us first prove the following (weaker) result: if $(S)$ is a Hopf SDSE, there exists a sequence $(G_i)_{0\leq i \leq k}$ 
of subgraphs of $\gs$, such that:
\begin{itemize}
\item $G_0$ is the disjoint union of several fundamental systems or is multicyclic.
\item $G_k=\gs$.
\item For all $0\leq i\leq k-1$, $G_{i+1}$ is obtained from $G_i$ by adding a non-self-dependent vertex without any ascendant in $G_i$.
\end{itemize}
Let us proceed by induction on $N$. If $N=1$, then $\gs=G_0$ is formed by a single vertex which is necessarily self-dependent, so $(S)$ is fundamental.
Let us assume the induction hypothesis at rank $\leq N-1$. If any vertex of $\gs$ has an ascendant, then by corollary \ref{49},
we can take $\gs=G_0$.  If it is not the case, let us take $i$ being a vertex with no ascendant. The induction hypothesis can be applied to 
the components of $\gs-\{i\}$. We complete the sequence $(G_0,\cdots,G_k)$ given in this way by $G_{k+1}=\gs$.\\

As a consequence, the set of descendants of any self-dependent vertex, every symmetric subgraph, every multicycle of $\gs$ is included in $G_0$.\\

{\it Second step.} Let us assume that $\gs$ is connected. If $G_0$ is connected, then it is fundamental or multicyclic. If it is not, let us assume that it is 
not a non-connected abelian fundamental SDSE. So one of the components $H$ of $G_0$ is not a fundamental abelian SDSE with $I=I_0$.
Then for a good choice of $i$, the vertex  added to $G_{i-1}$ to obtain $G_i$ is a connecting vertex, connecting a subgraph containing $H$
and other subgraphs. By the first step, as it does not belong to $G_0$, this vertex is not the descendant of a self-dependent vertex and does not belong 
to a symmetric subgraph. By construction, it does not connect several components of a non-connected fundamental SDSE:
 this is a contradiction with lemma \ref{46}. So $G_0$ is of the announced form. \end{proof}

\subsection{Connected Hopf SDSE with a multicycle}

Let us precise the structure of connected Hopf SDSE containing a multicycle.

\begin{theo} \label{51}
Let $(S)$ be a connected Hopf SDSE containing a $N$-multicyclic SDSE.
Then $I$ admits a partition $I=I_{\overline{1}}\cup\cdots \cup I_{\overline{N}}$, with the following conditions:
\begin{enumerate}
\item If $x \in I_{\overline{k}}$, its direct descendants are all in $I_{\overline{k+1}}$.
\item If $x$ and $x'$ have a common direct ascendant, then they have the same direct descendants.
\end{enumerate}
Moreover, for all $x\in I$:
$$F_x=1+\sum_{x\longrightarrow y} a^{(x)}_y h_y.$$
If $x$ and $x'$ have a common direct ascendant, then $F_x=F_{x'}$. Such an SDSE will be called an extended multicyclic SDSE.
\end{theo}

\begin{proof} We use the notations of corollary \ref{50}. We proceed by induction on $k$. If $k=0$, $(S)$ is a multicycle and the result is immediate.
Let us assume the result at rank $k-1$ and let $(S')$ be the restriction of $(S)$ to all the vertices except the last one, denoted by $x$. 
By the induction hypothesis, the set of its vertices admits a partition $I'=I'_{\overline{1}}\cup\cdots \cup I'_{\overline{N}}$, with the required conditions.
Let us first prove that all the direct descendants of $x$ are in the same $I'_{\overline{m}}$. Let $y \in I_{\overline{k}}$ and $z\in I_{\overline{l}}$ be two 
direct descendants of $x$, with $\overline{k}\neq \overline{l}$. Let $y' \in I_{\overline{k+1}}$ be a direct descendant of $y$ and $z' \in I_{\overline{l+1}}$ 
be a direct descendant of $z$. Lemma \ref{36} implies that $x$ is a direct ascendant of $z'$ and $y'$, as $y$ can't be a direct ascendant of $z'$ and
$z$ can't be a direct ascendant of $y'$ because $\overline{k}\neq \overline{l}$. So we can replace $y$ by $y'$ and $z$ by $z'$. Iterating the process, 
we can assume that $y$ and $z$ are in the multicycle: this contradicts lemma \ref{47}. So the direct descendants of $x$ are all in $I_{\overline{m}}$ 
for a good $m$. We then take $I_{\overline{l}}=I_{\overline{l}}'$ if $\overline{l}\neq \overline{m-1}$ and $I_{\overline{m-1}}=I_{\overline{m-1}}'\cup \{x\}$
and this proves the first assertion on $\gs$.

We now prove the assertion on $F_x$. We separate the proof into two subcases. Let us first assume $M\geq 3$. There is an oriented path 
$x\rightarrow x_{\overline{m}} \rightarrow \cdots \rightarrow x_{\overline{m+M-1}}$, with $x_{\overline{i}}\in I'_{\overline{i}}$ for all $i$. Moreover, there is 
no shorter oriented path from $x$ to $x_{\overline{m+M-1}}$. As $M\geq 3$, from lemma \ref{28}:
$$F_x=1+\sum_{x\longrightarrow y} a^{(x)}_y h_y.$$
Let us secondly assume that $M=2$.  Let $1,\ldots,p$ be the direct descendants of $x$ and let $0$ be a direct descendant of $1$. Then as $1,\ldots,p$ 
are in the same part of the partition of $I'$, they are not direct descendants of $1$. Let us first restrict to $\{x,1,0\}$. By proposition \ref{16},
$\lambda_3^{(x,0)}a_{\tdtroisdeux{$x$}{$1$}{$0$}}=0$ as $a^{(1)}_{0,0}=0$ by the induction hypothesis, $\lambda_3^{(x,0)}=0$. Moreover,
$0=\lambda_3^{(x,0)}a_{\tdtroisun{$x$}{$1$}{$1$}}=a_{\tdquatretrois{$x$}{$1$}{$0$}{$1$}}$, so $a^{(x)}_{1,1}=0$. Similarly,
$a^{(x)}_{2,2}=\cdots=a^{(x)}_{p,p}=0$. Let us now take $1\leq i<j\leq p$. Then $\lambda_2^{(x,i)}a_{\tddeux{$x$}{$i$}}=0$, so $\lambda_2^{(x,i)}=0$ 
and $0=\lambda_2^{(x,i)}a_{\tddeux{$x$}{$j$}}=a_{\tdtroisdeux{$x$}{$j$}{$i$}}$, so $a^{(x)}_{i,j}=0$. As a conclusion, $F_x$ is of the required form.

Proposition \ref{18}-3 implies that $F_x=F_{x'}$ if $x$ and $x'$ have a common ascendant, and this implies the second assertion on $\gs$. \end{proof}\\

{\bf Remark.} In particular, the vertex added to $G_i$ in order to obtain $G_{i+1}$ is an extension vertex. By proposition \ref{11}, any such SDSE is Hopf.

\subsection{Connected Hopf SDSE with finite levels}

We now prove the following theorem:

\begin{theo} \label{52}
Let $(S)$ be a connected Hopf SDSE, such that any vertex of $(S)$ has a finite level. Then $(S)$ is obtained from a fundamental system
by a finite number (possibly $0$) of extensions. Such an SDSE will be called an extended fundamental SDSE.
\end{theo}

\begin{proof} Let $(S)$ be a connected Hopf SDSE, such that any vertex of $(S)$ is of finite level. We use notations of corollary \ref{50}. We shall proceed 
by induction on $k$. If $k=0$, then $S=S_0$ and the result is obvious. Let us now assume the result at rank $k-1$. By the induction hypothesis, 
the system $(S')$ associated to $G_{k-1}$ is a dilatation of a system of theorem \ref{32}. Moreover, $G$ is obtained from $G_{k-1}$ by adding a vertex 
with all its direct descendants in $G_{k-1}$. Let us denote by $0$ this vertex. We separate the proof into three cases.\\

{\it First case}. Let us assume that $0$ is of level $0$. Then all the direct descendants of $0$ are of level $0$, so are in $I_0 \cup J_0 \cup I_1$, 
and $\nu_x=1$ for all direct descendants of $x$ in $J_i$ with $i\in I_1$. Moreover, for all $x\in I$, $\lambda_n^{(0,x)}=b_x(n-1)+a^{(0)}_x$.

Let us take $x,y \in I$. Using proposition \ref{19}-1 into two different ways:
$$a^{(0)}_{x,y}=\left(b_y+a_y^{(0)}-a^{(x)}_y\right)a^{(0)}_x=\left(b_x+a_x^{(0)}-a^{(y)}_x\right)a^{(0)}_y.$$
So, for all $x,y \in I$:
\begin{equation} \label{E7}
\left(b_y-a^{(x)}_y\right)a_x^{(0)}=\left(b_x-a_x^{(y)}\right)a_y^{(0)}.
\end{equation}
If $x$ and $y$ are in the same $I_i$ with $i\in I_0 \cup J_0$, then $b_y-a^{(x)}_y=b_x-a^{(y)}_x\neq 0$, so $a_x^{(0)}=a_y^{(0)}$ and for all $n \geq 1$,
$\lambda_n^{(0,x)}=\lambda_n^{(0,y)}$. Hence, up to a restriction, we can assume that there is no dilatations on $(S')$.

Let $i\in I_1$. If $\nu_i\neq 1$, we already know that $a^{(0)}_i=0$. Let us assume $\nu_i=1$ and let us choose $j\in I_0\cup J_0 \cup K_0$, such that
$a^{(i)}_j \neq b_j$. Then $b_i=a^{(j)}_i=0$, so (\ref{E7}) gives $\left(b_j-a^{(i)}_j\right)a^{(0)}_i=0$. So $a^{(0)}_i=0$ for all $i\in I_1$. So the direct
descendants of $0$ are all in $I_0 \cup J_0 \cup K_0$. Using proposition \ref{19}-1 with $i \in I_0 \cup J_0 \cup K_0$:
\begin{eqnarray*}
a^{(0)}_{(p_1,\cdots,p_{i+1},\cdots,p_N)}&=&\left(a^{(0)}_i+b_i(p_1+\cdots+p_N)-\sum_{j\in I_0\cup J_0 \cup K_0-\{i\}}b_i p_j-a^{(i)}_i p_i\right)
\frac{a^{(0)}_{(p_1,\cdots,p_N)}}{p_i+1}\\
&=&\left(a^{(0)}_i+\left(b_i-a^{(i)}_i\right)p_i\right)\frac{a^{(0)}_{(p_1,\cdots,p_N)}}{p_i+1}.
\end{eqnarray*}
So:
$$F_0=\prod_{i\in I_0} f_{\frac{\beta_i}{a^{(0)}_i}}\left(a^{(0)}_i h_i\right) \prod_{i\in J_0} f_{\frac{1}{a^{(0)}_i}}\left(a^{(0)}_i h_i\right)
\prod_{i\in K_0}f_0\left(a^{(0)}_i h_i\right).$$
So $(S)$ is a system of theorem \ref{32}, with $0\in K_0 \cup I_1$.\\

{\it Second case}.  Let us assume that $0$ is of level $1$ and is not an extension vertex. Then all the direct descendants of $0$ are of level $0$, 
so are in $I_0 \cup J_0 \cup I_1$, and $\nu_x=1$ for all direct descendants of $x$ in $I_1$. Moreover, for all $i\in I$, $\lambda_1^{(0,i)}=a^{(0)}_i$ 
and $\lambda_n^{(0,i)}=b_i(n-1)+\tilde{a}^{(0)}_i$ if $n\geq 2$.

{\it First item}. Let us assume that $a^{(0)}_i=0$. Then by proposition \ref{19}-1:
\begin{eqnarray*}
a^{(0)}_{(p_1,\cdots,1,\cdots,p_N)}&=&\left(\tilde{a}_i^{(0)}+b_i(p_1+\cdots+p_N)-\sum_{j=1}^N a^{(j)}_i p_j\right)a^{(0)}_{(p_1,\cdots,0,\cdots,p_N)}\\
0&=&\left(\tilde{a}_i^{(0)}-\sum_{j\in I_1} a^{(j)}_i p_j\right)a^{(0)}_{(p_1,\cdots,0,\cdots,p_N)}.
\end{eqnarray*}
If there is a $j\in I_0 \cup J_0 \cup K_0$, such that $a^{(0)}_j\neq 0$, then for $(p_1,\cdots,p_N)=\varepsilon_j$, we obtain $\tilde{a}_i^{(0)}=0$. 
If it is not the case, as $0$ is not an extension vertex, there exists $j,k \in I_1$, $a^{(0)}_{j,k}\neq 0$ (so $a^{(0)}_j \neq 0$ and $a^{(0)}_k \neq 0$). 
Then, for $(p_1,\cdots,p_N)=\varepsilon_j$, $(p_1,\cdots,p_N)=\varepsilon_k$, and $(p_1,\cdots,p_N)=\varepsilon_j+\varepsilon_k$, we obtain:
$$\tilde{a}^{(0)}_i+a^{(j)}_i=\tilde{a}^{(0)}_i+a^{(k)}_i=\tilde{a}^{(0)}_i+a^{(j)}_i+a^{(k)}_i=0.$$
So $\tilde{a}^{(0)}_i=0$. So in all cases, $\tilde{a}^{(0)}_i=0$. Moreover, for $(p_1,\cdots,p_N)=\varepsilon_j$ for any $j \in I_1$,
we obtain $a^{(j)}_i a^{(0)}_j=0$. As a conclusion, we proved:
\begin{enumerate}
\item For all $i\in I$, $\left(a^{(0)}_i=0\right) \Longrightarrow \left(\tilde{a}^{(0)}_i=0\right)$.
\item Let us put $I_1^{(0)}=\left\{i\in I_1\:/\:a^{(0)}_i \neq 0\right\}$. Then for $i\in I$, such that $a^{(0)}_i=0$, for all $j \in I_1^{(0)}$, $a^{(j)}_i=0$.
\end{enumerate}

{\it Second item}. Let us take $i,j \in I$. Using proposition \ref{19}-1 into two different ways:
\begin{equation} \label{E8}
a^{(0)}_{i,j}=\left(b_j+\tilde{a}_j^{(0)}-a^{(i)}_j\right)a^{(0)}_i=\left(b_i+\tilde{a}_i^{(0)}-a^{(j)}_i\right)a^{(0)}_j.
\end{equation}
Let us take $i,j \in I_1$. Then $a^{(i)}_j=a^{(j)}_i=b_i=b_j=0$, so (\ref{E8}) gives:
$$\tilde{a}_j^{(0)}a_i^{(0)}=\tilde{a}_i^{(0)}a_j^{(0)}.$$
So $\left(\tilde{a}^{(0)}_i\right)_{i\in I_1}$ and $\left(a^{(0)}_i\right)_{i\in I_1}$ are colinear. By the first item, we deduce that there exists a scalar $\nu \in K$, 
such that for all $i\in I_1$, $\tilde{a}^{(0)}_i=\nu a^{(0)}_i$. Let us now take $i,j \in I_0\cup J_0 \cup K_0$, with $i\neq j$. Then $b_i=a^{(j)}_i$
and $b_j=a^{(i)}_j$, so (\ref{E8}) gives:
$$\tilde{a}_j^{(0)}a_i^{(0)}=\tilde{a}_i^{(0)}a_j^{(0)}.$$
So $\left(\tilde{a}^{(0)}_i\right)_{i\in I_0 \cup J_0 \cup K_0}$ and $\left(a^{(0)}_i\right)_{i\in  I_0 \cup J_0 \cup K_0}$ are colinear. 
By the first item, we deduce that there exists a scalar $\nu' \in K$, such that for all $i\in I_0 \cup J_0 \cup K_0$, $\tilde{a}^{(0)}_i=\nu' a^{(0)}_i$.
Let us now take $i\in I_0\cup J_0 \cup K_0$ and $j \in I_1$. Then $b_j=a^{(i)}_j=0$, so $\nu a^{(0)}_j a^{(0)}_i=\left(b_i+\nu'a_i^{(0)}-a^{(j)}_i\right)a^{(0)}_j$. 
In other words:
\begin{equation}\label{E9}
\forall i\in I_0\cup J_0 \cup K_0,\: \forall j\in I_1,\: (\nu-\nu') a^{(0)}_i a^{(0)}_j=(b_i-a^{(j)}_i)a_j^{(0)}.
\end{equation}

{\it Third item.} Let us assume that $I_1^{(0)}=\emptyset$. Then all the direct descendants of $0$ are in $I_0 \cup J_0 \cup K_0$. 
Moreover, if $i\in I_0 \cup J_0 \cup K_0$:
\begin{eqnarray*}
a^{(0)}_{(p_1,\cdots,p_{i+1},\cdots,p_N)}&=&\left(\nu a^{(0)}_i+b_i(p_1+\cdots+p_N)-\sum_{j\in I_0\cup J_0 \cup K_0-\{i\}}b_i p_j-a^{(i)}_i p_i\right)
\frac{a^{(0)}_{(p_1,\cdots,p_N)}}{p_i+1}\\
&=&\left(\nu a^{(0)}_i+\left(b_i-a^{(i)}_i\right)p_i\right)\frac{a^{(0)}_{(p_1,\cdots,p_N)}}{p_i+1}.
\end{eqnarray*}
It is then not difficult to show that $(S)$ is a system of theorem \ref{32}, with $0 \in I_1$.\\

{\it Fourth item}. Let us assume that $\nu=\nu'$. Let $j\in I_1$. If $\nu_j\neq 1$, then we already know that $a^{(0)}_j=0$. If $\nu_j=1$, then for a good choice
of $i$, $b_i-a^{(j)}_i \neq 0$ in (\ref{E9}), so $a^{(0)}_j=0$: then $I_1^{(0)}=\emptyset$, and the result is proved in the third item.\\

{\it Fifth item}. Let us assume that $I_1^{(0)} \neq \emptyset$. By the preceding item, $\nu \neq \nu'$. Let us take $j\in I_1^{(0)}$. By (\ref{E9}), 
for all $i\in I_0\cup J_0 \cup K_0$, $a^{(j)}_i=b_i-(\nu-\nu')a^{(0)}_i$ does not depend of $j$. As a consequence, $F_j=F_k$ for all $j,k\in I_1^{(0)}$.
We put $b_i^{(0)}=a^{(j)}_i$ for all $i\in I_0 \cup J_0 \cup K_0$, where $j$ is any element of $I_1^{(0)}$. Let us use proposition \ref{19}-1. 
For all $i\in I_0 \cup J_0 \cup K_0$, if $(p_1,\cdots,p_N)\neq (0,\cdots,0)$:
$$a^{(0)}_{(p_1,\cdots,p_i+1,\cdots,p_N)}=
\left( \nu' a_i^{(0)}+\left(b_i-a^{(i)}_i\right)p_i+(\nu-\nu')a_i^{(0)}\sum_{j\in I_1^{(0)}} p_j\right)\frac{a^{(0)}_{(p_1,\cdots,p_N)}}{p_i+1}.$$
For all $j\in I_1^{(0)}$, if $(p_1,\cdots,p_N)\neq (0,\cdots,0)$:
$$a^{(0)}_{(p_1,\cdots,p_i+1,\cdots,p_N)}=\nu a^{(0)}_i \frac{a^{(0)}_{(p_1,\cdots,p_N)}}{p_i+1}.$$
Let us fix $i\in I_0\cup J_0 \cup K_0$ and $j\in I_1^{(0)}$. Then:
\begin{eqnarray*}
a^{(0)}_{i,i}&=&\left(\nu'a^{(0)}_i+b_i-a^{(i)}_i\right)a_i^{(0)},\\
a^{(0)}_{i,i,j}&=&\nu a^{(0)}_i a^{(0)}_j \left(\nu'a^{(0)}_i+b_i-a^{(i)}_i\right),\\
a^{(0)}_{i,j}&=&\nu a^{(0)}_i a^{(0)}_j ,\\
a^{(0)}_{i,i,j}&=&\nu a^{(0)}_i a^{(0)}_j \left(\nu'a^{(0)}_i+b_i-a^{(i)}_i+(\nu-\nu') a^{(0)}_i\right).
\end{eqnarray*}
Identifying the two expressions of $a^{(0)}_{i,i,j}$, as $\nu \neq \nu'$ and $a^{(0)}_j \neq 0$, we obtain $\nu\left(a^{(0)}_i\right)^2=0$. 
If for all $i\in I_0 \cup J_0 \cup K_0$, $a^{(0)}_i=0$, then by the second item, for all $j\in I_1^{(0)}$, $a^{(j)}_i=0$, then $F_j=1$; this is impossible. 
So there is an $i\in I_0 \cup J_0 \cup K_0$, such that $a^{(0)}_i\neq 0$. As a consequence, $\nu=0$. So $\nu' \neq 0$, and we then easily obtain that:
\begin{eqnarray*}
F_0&=&\frac{1}{\nu'} \prod_{i\in I_0} f_{\frac{\beta_i}{b_i^{(0)}-1-\beta_i}}\left(\left(b_i^{(0)}-1-\beta_i\right)h_i\right)
\prod_{i\in J_0} f_{\frac{1}{b_i^{(0)}-1}}\left(\left(b_i^{(0)}-1\right)h_i\right)\prod_{i\in I_0} f_0\left(b_i^{(0)}h_i\right)\\
&&+\sum_{i\in I_1^{(0)}} a^{(0)}_i h_i+1-\frac{1}{\nu'}.
\end{eqnarray*}
So $(S)$ is a system of theorem \ref{32}, with $0 \in J_1$.\\

{\it Third case.} $0$ is a vertex of level $\geq 2$. By proposition \ref{29}, it is an extension vertex.
\end{proof}

\section{Lie algebra and group associated to $\hs$, associative case}

Let us consider a connected Hopf SDSE $(S)$. We now study the pre-Lie algebra $\lies$ of proposition \ref{21}. We separate this study into three cases:
\begin{itemize}
\item {\it Associative case}: the pre-Lie algebra $\lies$ is associative. This holds in particular if $(S)$ is an extended multicyclic SDSE.
\item {\it Abelian case}: $(S)$ is an extended fundamental, abelian SDSE (see definition \ref{33}).
\item {\it Non-abelian case}: $(S)$ is an extended fundamental, non-abelian SDSE.
\end{itemize}

We first treat the associative case.

\subsection{Characterization of the associative case}

\begin{prop} \label{53}
Let $(S)$ be a Hopf SDSE. Then the pre-Lie algebra $\lies$ is associative if, and only if, for all $i\in I$:
$$F_i=1+\sum_{i\longrightarrow j} a^{(i)}_j h_j.$$
\end{prop}

\begin{proof} $\Longrightarrow$. Let us assume that $\star$ is associative. Let $i,j,k \in I$, let us show that $a^{(i)}_{j,k}=0$. If $a^{(i)}_j=0$ or $a^{(i)}_k=0$,
then $a^{(i)}_{j,k}=0$. Let us suppose that $a^{(i)}_j\neq 0$ and $a^{(i)}_k \neq 0$. Then:
\begin{eqnarray*}
0&=&(f_k(1)\star f_j(1))\star f_i(1)-f_k(1)\star (f_j(1)\star f_i(1))\\
&=&\left(\lambda^{(j,k)}_1 \lambda^{(i,j)}_1-\lambda^{(i,j)}_1 \lambda^{(i,k)}_2\right)f_i(3)\\
&=&\lambda^{(i,j)}_1\left(\lambda^{(j,k)}_1- \lambda^{(i,k)}_2\right)f_i(3)\\
&=&a^{(i)}_j \left(a^{(j)}_k- \lambda^{(i,k)}_2\right)f_i(3).
\end{eqnarray*}
So $\lambda_2^{(i,k)}=a^{(j)}_k$. Moreover, by proposition \ref{16}: 
$$a^{(i)}_j a^{(j)}_k=\lambda_2^{(i,k)} a_{\tddeux{$i$}{$j$}}=a_{\tdtroisdeux{$i$}{$j$}{$k$}}+(1+\delta_{j,k})a_{\tdtroisun{$i$}{$k$}{$j$}}
=a^{(i)}_j a^{(j)}_k+(1+\delta_{j,k})a^{(i)}_{j,k}.$$
So $a^{(i)}_{j,k}=0$. As a consequence:
$$F_i=1+\sum_{i\longrightarrow j} a^{(i)}_j h_j.$$

$\Longleftarrow$. Then $X_i(n)$ is a linear span of ladders of weight $n$ for all $n\geq 1$, for all $i\in I$. As a consequence, 
if $x \in Vect(X_i(n)\:/\:i\in I,n\geq 1)$, for all $f,g \in \lies$:
$$(f\star g)(x)=(f\otimes g) \circ (\pi \otimes \pi)\circ \Delta(x)=(f\otimes g)\circ\Delta(x)=f(x')g(x'').$$
So if $f,g,h\in \gs$, for all $x \in Vect(X_i(n)\:/\:i\in I,n\geq 1)$:
$$((f\star g) \star h)(x)=f(x')g(x'')h(x''')=(f\star (g\star h))(x).$$
So $(f\star g)\star h=f \star (g\star h)$: $\lies$ is an associative algebra. \end{proof}

\begin{cor}
Let $(S)$ be a connected Hopf SDSE. Then $\lies$ is associative if, and only if one of the following assertions holds:
\begin{enumerate}
\item $(S)$ is an extended multicyclic SDSE.
\item $(S)$ is an extended fundamental SDSE, with:
\begin{itemize}
\item For all $i\in I_0$, $\beta_i=-1$.
\item $J_0$, $K_0$, $I_1$ and $J_1$ are empty.
\end{itemize}\end{enumerate}\end{cor}

If the second assertion holds, then $(S)$ is also an extended fundamental abelian SDSE, and another interpretation of $\lies$ can be given;
see theorem \ref{70}.

\subsection{An algebra associated to an oriented graph}

{\bf Notations.} Let $G$ an oriented graph, $i,j \in G$, and $n\geq 1$. We shall denote $i \fleche{n} j$ if there is an oriented path from 
$i$ to $j$ of length $n$ in $G$.

\begin{defi}\textnormal{
Let $G$ be an oriented graph, with set of vertices denoted by $I$. 
The associative, non-unitary algebra $A_G$ is generated by $P_i(1)$, $i\in I$, and the following relations:
\begin{itemize}
\item If $j$ is not a direct descendant of $i$ in $G$, $P_j(1)P_i(1)=0$.
\item If $i_1\rightarrow i_2\rightarrow \cdots \rightarrow i_n$ and $i_1\rightarrow i'_2\rightarrow \cdots \rightarrow i'_n$ in $G$, then:
$$P_{i_n}(1)\cdots P_{i_2}(1)P_{i_1}(1)=P_{i'_n}(1)\cdots P_{i'_2}(1)P_{i_1}(1).$$
\end{itemize}}\end{defi}

Let $G$ be an oriented graph, and let $i\in I$ and $n \geq 1$. For any oriented path $i\rightarrow i_2\rightarrow \cdots \rightarrow i_n$ in $G$,
we denote $P_i(n)=P_{i_n}(1)\cdots P_{i_2}(1)P_i(1)$. If there is no such an oriented path, we put $P_i(n)=0$.
By definition of $A_G$ (second family of relations), this does not depend of the choice of the path.

\begin{lemma}
Let $G$ be an oriented graph. Then the $P_i(n)$'s, $i\in I$, $n\geq 1$, linearly generate $A_G$. Moreover, if $P_i(m)$ and $P_j(n)$ are non-zero, then:
$$P_j(n)P_i(m)=\left\{ \begin{array}{l}
P_i(m+n) \mbox{ if }i\fleche{m} j,\\
0\mbox{ if not.}
\end{array}\right.$$
\end{lemma}

\begin{proof} By the first relation, $P_i(n)=P_{i_n}(1)\cdots P_{i_2}(1)P_i(1)=0$ if $(i,i_1,\ldots,i_n)$ is not an oriented path in $G$.
So the $P_i(n)$'s, $i\in I$, $n\geq 1$, linearly generate $A_G$. 

let us fix $P_i(m)=P_{i_m}(1)\cdots P_{i_2}(1)P_i(1)$ and $P_j(n)=P_{j_n}(1)\cdots P_{j_2}(1)P_j(1)$ both non-zero.
If $i\fleche{m} j$ we can choose $i_2,\ldots,i_m$ such that $i\rightarrow i_2\rightarrow \cdots \rightarrow i_m \rightarrow j$. Then:
$$P_j(n)P_i(m)=P_{j_n}(1)\cdots P_{j_2}(1)P_j(1)P_{i_m}(1)\cdots P_{i_2}(1)P_i(1)=P_i(m+n).$$
If this is not the case, then $j$ is not a direct descendant of $i_m$, so $P_j(1)P_{i_m}(1)=0$ and $P_j(n)P_i(m)=0$. \end{proof}

\begin{prop}\label{57}
Let $G$ be an oriented graph. 
\begin{enumerate}
\item The following conditions are equivalent:
\begin{enumerate} 
\item The family $(P_i(n))_{i\in I,n\geq 1}$ is a basis of $A_G$.
\item All the $P_i(n)$ are non-zero.
\item The graph $G$ satisfies the following conditions:
\begin{itemize}
\item Any vertex of $G$ has a direct descendant.
\item If two vertices of $G$ have a common direct ascendant, then they have the same direct descendants.
\end{itemize}
\item The SDSE associated to the following formal series is Hopf:
$$\forall i\in I, \:F_i=1+\sum_{i\rightarrow j} h_j.$$
\end{enumerate}
\item If this holds, then $A_G$ is generated by $P_i(1)$, $i\in I$, and the following relations:
\begin{itemize}
\item If $j$ is not a direct descendant of $i$ in $G$, $P_j(1)P_i(1)=0$.
\item If $i\rightarrow j$ and $i\rightarrow k$ in $G$, then $P_j(1)P_i(1)=P_k(1)P_i(1)$.
\end{itemize} 
The product of $A_G$ is given by:
$$P_j(n)P_i(m)=\left\{ \begin{array}{l}
P_i(m+n) \mbox{ if }i\fleche{m} j,\\
0\mbox{ if not.}
\end{array}\right.$$
Moreover, if $(S)$ is the system of condition $(d)$, $\lies$ is associative and isomorphic to $A_G$.
\end{enumerate} \end{prop}

\begin{proof} 1. $(a) \Longrightarrow (b)$ is obvious. \\

$(b) \Longrightarrow (c)$. Let us assume $(b)$. Then for all $i\in I$, $P_i(2)\neq 0$, so there exists a $j$ such that $i\rightarrow j$ in $G$: any vertex of $G$
has a direct descendant. Let us assume $i\rightarrow j$ and $i\rightarrow j'$ in $G$. Let $k$ be a direct descendant of $j$. Then
$P_i(2)=P_j(1)P_i(i)=P_{j'}(1)P_i(1)$ and $P_i(3)=P_k(1)P_j(1)P_i(1)=P_k(1)P_i(2)\neq 0$, so $P_k(1)P_i(2)=P_k(1)P_{j'}(1)P_i(1)\neq 0$. 
As a consequence, $P_k(1)P_{j'}(1)\neq 0$ and $k$ is a direct descendant of $j'$. By symmetry, the direct descendants of $j'$ are also 
direct descendants of $j$: two direct descendants of a same vertex have the same direct descendants.\\

$(c)\Longrightarrow (d)$. Then for all $i\in I$, for all $n \geq 1$:
$$X_i(n)=\sum l(i,i_2,\cdots,i_n),$$
where the sum runs on all oriented paths $i\rightarrow i_2\rightarrow \cdots \longrightarrow i_n$ in $\gs$. So:
$$\Delta(X_i(n))=\sum \sum_{k=0}^n l(i_{k+1},\ldots,i_n)\otimes l(i,i_2,\cdots,i_k).$$
If $i\rightarrow i_2 \cdots \rightarrow i_k\rightarrow i_{k+1}$ and $i\rightarrow i'_2 \cdots \rightarrow i'_k\rightarrow i'_{k+1} $,
the second condition on $G$ implies that $i_3$ and $i'_3$ are direct descendants of $i_2$ and $i'_2$,$\ldots$, $i_{k+1}$ and $i'_{k+1}$ are direct
descendants of $i_k$ and $i'_k$. So:
$$\Delta(X_i(n))=\sum_{k=0}^n \sum_{\substack{i\rightarrow \cdots \rightarrow i_k,\\i\fleche{k} i_{k+1},\\ i_{k+1}\rightarrow \cdots \rightarrow i_n}}
l(i_{k+1},\ldots,i_n)\otimes l(i,i_2,\cdots,i_k)=\sum_{k=0}^n \sum_{i\fleche{k} j} X_j(n-k) \otimes X_i(k).$$
So $(S)$ is Hopf.\\

$(d) \Longrightarrow (a)$. Then, for all $i\in I$, for all $n \geq 1$:
$$X_i(n)=\sum l(i,i_2,\cdots,i_n),$$
where the sum runs on all oriented paths $i\rightarrow i_2\rightarrow \cdots \longrightarrow i_n$ in $\gs$. By proposition \ref{53}, $\lies$ is associative.
Moreover, it is quite immediate to prove that in $\lies$:
\begin{itemize}
\item If $j$ is not a direct descendant of $i$ in $G$, $f_j(1)f_i(1)=0$.
\item If $i_1\rightarrow i_2\rightarrow \cdots \rightarrow i_n$ and $i_1\rightarrow i'_2\rightarrow \cdots \rightarrow i'_n$ in $G$, then:
$$f_{i_n}(1)\cdots f_{i_2}(1)f_{i_1}(1)=f_{i'_n}(1)\cdots f_{i'_2}(1)f_{i_1}(1)=f_{i_1}(n).$$
\end{itemize}
So there is a morphism of algebras from $A_G$ to $\lies$, sending $P_i(1)$ to $f_i(1)$. This morphism sends $P_i(n)$ to $f_i(n)$. 
As the $f_i(n)$'s are linearly independent, so are the $P_i(n)$'s.\\

$2$. Let $A'_G$ be the associative, non-unitary algebra generated by the relations of proposition \ref{57}-2. As these relation are immediatly satisfied 
in $A_G$, there is a unique morphism of algebras:
$$\Phi: \left\{ \begin{array}{rcl}
A'_G&\longrightarrow & A_G\\
P_i(1)&\longrightarrow &P_i(1).
\end{array}\right.$$
Let $i_1\rightarrow i_2\rightarrow \cdots \rightarrow i_n$ and $i_1\rightarrow i'_2\rightarrow \cdots \rightarrow i'_n$ in $G$. Let us prove that
$P_{i_k}(1)\cdots P_{i_2}(1)P_{i_1}(1)=P_{i'_k}(1)\cdots P_{i'_2}(1)P_{i_1}(1)$ in $A'_G$ by induction on $k$. For $k=2$, this is implied by the second
family of relations defining $A'_G$. Let us assume the result at rank $k$. Then, both in $A_G$ and $A'_G$:
$$P_{i_{k+1}}(1)P_{i_k}(1)\cdots P_{i_2}(1)P_{i_1}(1)=P_{i_{k+1}}(1)P_{i'_k}(1)\cdots P_{i'_2}(1)P_{i_1}(1).$$
This is equal to $P_i(k+1)$ in $A_G$, so is non-zero. As a consequence, $P_{i_{k+1}}(1)P_{i'_k}(1)\neq 0$ in $A_G$, so $i'_k\rightarrow i_{k+1}$ in $G$.
By definition of $A'_G$, $P_{i_{k+1}}(1)P_{i'_k}(1)=P_{i'_{k+1}}(1)P_{i'_k}(1)$ in $A'_G$, so:
$$P_{i_{k+1}}(1)P_{i_k}(1)\cdots P_{i_2}(1)P_{i_1}(1)=P_{i'_{k+1}}(1)P_{i'_k}(1)\cdots P_{i'_2}(1)P_{i_1}(1).$$
So the relations defining $A_G$ are also satisfied in $A'_G$, so there is a morphism of algebras:
$$\Psi : \left\{ \begin{array}{rcl}
A_G&\longrightarrow & A'_G\\
P_i(1)&\longrightarrow &P_i(1).
\end{array}\right.$$
It is clear that $\Phi$ and $\Psi$ are inverse isomorphisms of algebras. \end{proof}

\begin{cor}
Let $(S)$ a Hopf SDSE. If $\lies$ is associative, then the graph $\gs$ satisfies condition $(c)$ of proposition \ref{57} and $\lies$ is isomorphic to $A_{\gs}$.
\end{cor}

\begin{proof} {\it First step.} Let $i,j,k$ be vertices of $\gs$ and $n\geq 1$ such that $i\fleche{n} j$ and $i \fleche{n} k$. 
Let us prove that $F_j=F_k$ by induction on $n$. If $n=1$, by proposition \ref{18}-3, $F_j=F_k$. If $n \geq 2$, then there exists vertices of $\gs$ such that:
$$i\rightarrow j_1 \rightarrow \ldots \rightarrow j_{n-1} \rightarrow j,\hspace{1cm} i\rightarrow k_1 \rightarrow \ldots \rightarrow k_{n-1} \rightarrow k.$$
The case $n=1$ implies that $F_{j_1}=F_{k_1}$, so $j_1 \fleche{n-1} j$ and $j_1 \fleche{n-1} k$. By the induction hypothesis, $F_j=F_k$.
In other words, if $i\fleche{n} j$ and $i \fleche{n} k$, then $a^{(j)}_l=a^{(k)}_l$ for all $l\in I$. \\

{\it Second step.} Then, for all $i\in I$, for all $n \geq 1$:
$$X_i(n)=\sum a^{(i)}_{i_1} \cdots a^{(i_{n-1})}_{i_n} l(i,i_2,\cdots,i_n),$$
where the sum runs on all oriented paths $i\rightarrow i_2\rightarrow \cdots \longrightarrow i_n$ in $\gs$.
The first step implies that $a^{(i)}_{i_1}\ldots a^{(i_{n-1})}_{i_n}$ depends only of $i$ and $n$: we denote it by $a^{(i)}_n$. Then:
\begin{eqnarray*}
X_i(n)&=&\sum a^{(i)}_n l(i,i_2,\cdots,i_n),\\
\Delta(X_i(n))&=&\sum_{k+l=n} \sum_{i\fleche{l} j} \frac{a^{(i)}_n}{a^{(i)}_la^{(j)}_k} X_j(k) \otimes X_i(l).
\end{eqnarray*}
Dually, putting $p_i(n)=a^{(i)}_n f_i(n)$ for all $1\leq i\leq N$, $n \geq 1$, the pre-Lie product of $\lies$ is given by:
\begin{eqnarray*}
f_j(n)\star f_i(m)&=&\left\{\begin{array}{l}
\displaystyle \frac{a^{(i)}_{m+n}}{a^{(i)}_m a^{(j)}_n}f_i(m+n) \mbox{ if } i\fleche{m} j,\\[2mm]
0\mbox{ otherwise};
\end{array}\right.\\
p_j(n)\star p_i(m)&=&\left\{\begin{array}{l}
p_i(m+n) \mbox{ if } i\fleche{m} j,\\[2mm]
0\mbox{ otherwise}.
\end{array}\right. \end{eqnarray*}

{\it Last step.} It is then clear that the associative algebra $\lies$ is generated by the $p_i(1)$, $i\in I$, and that these elements satisfy the relations
defining $A_{\gs}$. So there is an epimorphism of algebras:
$$\Theta: \left\{ \begin{array}{rcl}
A_{\gs} &\longrightarrow &\lies\\
P_i(1)&\longrightarrow &p_i(1).
\end{array}\right.$$
This morphism sends $P_i(n)$ to $p_i(n)$ for all $n\geq 1$. As the $p_i(n)$'s are a basis of $A_{\gs}$, the $P_i(n)$'s are linearly independent in $A_{\gs}$,
so the graph $\gs$ satisfies condition $(c)$ of proposition \ref{57}. Moreover, $\Theta$ is an isomorphism. \end{proof}

\subsection{Group of characters}

The non-unitary, associative algebra $\lies$ is graded, with $p_i(k)$ homogeneous of degree $k$ for all $k\geq 1$. Moreover, $\lies(0)=(0)$. The completion
$\widehat{\lies}$ is then an associative non-unitary algebra. We add it a unit and obtain an associative unitary algebra $K\oplus \widehat{\lies}$. 
It is then not difficult to show that the following set is a subgroup of the units of $K\oplus \widehat{\lies}$:
$$G=\left\{1+\sum_{k\geq 1}x_k\:|\:\forall k\geq 1,\: x_k \in \lies(k)\right\}.$$

\begin{prop}
The group of characters $Ch\left(\hs \right)$ is isomorphic to $G$.
\end{prop}

\begin{proof} We put $V=Vect(X_i(k)|i\in I,k\geq 1)$. Let $g \in V^*$. Then $g$ can be uniquely extended in a map $\widehat{g}$ from $\hs$ to $K$
by $g((1)+Ker(\varepsilon)^2)=(0)$, where $\varepsilon$ is the counit of $\hs$. Moreover, $\widehat{g} \in \widehat{\lies}$. This construction implies a bijection:
$$\Omega: \left\{ \begin{array}{rcl}
Ch\left(\hs \right)&\longrightarrow & G\\
f&\longrightarrow&\displaystyle 1+ \widehat{f_{\mid V}}.
\end{array}\right.$$
Let $f_1,f_2 \in Ch\left(\hs \right)$. For all $x \in V$, we put $\Delta(x)=x\otimes 1+1\otimes x+x'\otimes x''$.
As $x$ is a linear span of ladders, $x'\otimes x'' \in V \otimes V$. So:
\begin{eqnarray*}
(f_1.f_2)(x)&=&(f_1\otimes f_2)\circ \Delta(x)\\
&=&f_1(x)+f_2(x)+f_1(x')f_1(x'')\\
&=&{f_1}_{\mid V}(x)+{f_2}_{\mid V}(x)+{f_1}_{\mid V}(x'){f_2}_{\mid V}(x'')\\
&=&\widehat{{f_1}_{\mid V}}(x)+\widehat{{f_2}_{\mid V}}(x)+\widehat{{f_1}_{\mid V}}(x')\widehat{{f_2}_{\mid V}}(x'')\\
&=&\widehat{{f_1}_{\mid V}}(x)+\widehat{{f_2}_{\mid V}}(x)+\left(\widehat{{f_1}_{\mid V}}\star\widehat{{f_2}_{\mid V}}\right)(x).
\end{eqnarray*}
So $\widehat{(f_1.f_2)_{\mid V}}=\widehat{{f_1}_{\mid V}}+\widehat{{f_2}_{\mid V}}+\widehat{{f_1}_{\mid V}}\star\widehat{{f_2}_{\mid V}}$.
This implies that $\Omega$ is a group isomorphism. \end{proof}

\section{Lie algebra and group associated to $\hs$, non-abelian case}

In non-abelian or abelian cases, then any vertex of $\gs$ is of finite level. By proposition \ref{21}, the constant structures of the pre-Lie product satisfy:
$$\lambda_n^{(i,j)}=\left\{\begin{array}{l}
a^{(i)}_j\mbox{ if } n=1,\\
b_j(n-1)+\tilde{a}^{(i)}_j \mbox{ if } n\geq level(i)+1,
\end{array}\right.$$
where the $a^{(i)}_j$'s, $\tilde{a}^{(j)}_i$'s and $b_j$'s are scalars.

\subsection{Modules over the Faà di Bruno Lie algebra}

Let $\gfdb$ be the Faà di Bruno Lie algebra. Recall that it has a basis $(e(k))_{k\geq 1}$, with bracket given by:
$$[e(k),e(l)]=(l-k)e(k+l).$$
The $\gfdb$-module $V_0$ has a basis $(f(k))_{k\geq 1}$, and the action of $\gfdb$ is given by:
$$e(k).f(l)=lf(k+l).$$

We can then construct a semi-direct product $V_0^M \triangleleft \gfdb$, described in the following proposition:

\begin{prop} 
Let $M \in \mathbb{N}^*$. The Lie algebra $V_0^M \triangleleft \gfdb$ has a basis:
$$\left(f^{(i)}(k)\right)_{1\leq i \leq M,\: k\geq 1}\cup (e(k))_{k\geq 1},$$
and its Lie bracket given by:
$$\left\{ \begin{array}{rcl}
[e(k),e(l)]&=&(l-k)e(k+l),\\[0mm]
[e(k),f^{(i)}(l)]&=&lf^{(i)}(k+l),\\[0mm]
[f^{(i)}(k),f^{(j)}(l)]&=&0.
\end{array}\right.$$
\end{prop}

We now take $\g=V_0^{\oplus M} \triangleleft \gfdb$. We define a family of $\g$-modules.  Let $c \in K$ and $\upsilon=(\upsilon_1,\ldots,\upsilon_M) \in K^M$.
The module $W_{c,\upsilon}$ has a basis $(g(k))_{k\geq 1}$, and the action of $\g$ is given by:
$$\left\{ \begin{array}{rcl}
e(k).g(l)&=&(l+c) g(k+l),\\
f^{(i)}(k).g(l)&=&\upsilon_i g(k+l).
\end{array}\right.$$
The semi-direct product is given in the following proposition:

\begin{prop}
Let $\g$ be the following Lie algebra:
$$\left(W_{c_1,\upsilon^{(1)}}\oplus \ldots \oplus W_{c_N,\upsilon^{(N)}} \right) \triangleleft \left(V_0^M \triangleleft \gfdb\right).$$
It has a basis:
$$\left(g^{(j)}(k)\right)_{1\leq j \leq N,\: k\geq 1} \cup \left(f^{(i)}(k)\right)_{1\leq i \leq M,\: k\geq 1} \cup (e(k))_{k\geq 1},$$
and its bracket is given by:
$$\left\{ \begin{array}{rcl}
[e(k),e(l)]&=&(l-k)e(k+l),\\[0mm]
[e(k),f^{(i)}(l)]&=&lf^{(i)}(k+l),\\[0mm]
[e(k),g^{(i)}(l)]&=&(l+c'_i)g^{(i)}(k+l),\\[0mm]
[f^{(i)}(k),f^{(j)}(l)]&=&0,\\[0mm]
[f^{(i)}(k),g^{(j)}(l)]&=&\upsilon^{(j)}_i g^{(j)}(k+l),\\[0mm]
[g^{(i)}(k),g^{(j)}(l)]&=&0.
\end{array}\right.$$
\end{prop}

Let us take $\g$ as in this proposition. We define three families of modules over $\g$:

\begin{enumerate}
\item Let $\nu=(\nu_1,\ldots,\nu_M) \in K^M$. The module $W'_{\nu,0}$ has a basis $(h(k))_{k\geq 1}$, and the action of $\g$ is given by:
$$\left\{ \begin{array}{rcl}
e(k).g(l)&=&(l-1) h(k+l),\\
f^{(i)}(k).h(1)&=&\nu_i h(k+1),\\
f^{(i)}(k).h(l)&=&0 \mbox{ if }l \geq 2,\\
g^{(i)}(k).h(l)&=&0.
\end{array}\right.$$
\item Let $\nu=(\nu_1,\ldots,\nu_M) \in K^M$. The module $W'_{\nu,1}$ has a basis $(h(k))_{k\geq 1}$, and the action of $\g$ is given by:
$$\left\{ \begin{array}{rcl}
e(k).h(1)&=&h(k+1),\\
e(k).h(l)&=&(l-1) h(k+l) \mbox{ if }l \geq 2,\\
f^{(i)}(k).h(1)&=&\nu_i h(k+1),\\
f^{(i)}(k).h(l)&=&0 \mbox{ if }l \geq 2,\\
g^{(i)}(k).h(l)&=&0.
\end{array}\right.$$
\item Let $c \in K$, $\nu=(\nu_1,\ldots,\nu_M) \in K^M$, $\mu=(\mu_1,\ldots,\mu_N) \in K^N$. The module $W''_{c,\nu,\mu}$ has a basis $(h(k))_{k\geq 1}$,
and the action of $\g$ is given by:
$$\left\{ \begin{array}{rcl}
e(k).h(l)&=&(l+c) h(k+l),\\
f^{(i)}(k).h(l)&=&\nu_i h(k+l),\\
g^{(i)}(k).h(1)&=&\mu_i h(k+1),\\
g^{(i)}(k).h(l)&=&0 \mbox{ if }l\geq 2.
\end{array}\right.$$
\end{enumerate}

\subsection{Description of the Lie algebra}

\begin{theo} \label{62}
Let us consider a connected, fundamental non-abelian SDSE. Then $\lies$ has the following form:
$$\lies\approx W \triangleleft\left(\left(W_{c_1,\upsilon^{(1)}}\oplus \ldots \oplus W_{c_N,\upsilon^{(N)}} \right) \triangleleft \left(V_0^M 
\triangleleft \gfdb\right)\right),$$
where $W$ is a direct sum of $W'_{\nu,0}$, $W'_{\nu,1}$ and $W''_{c,\nu,\mu}$.
\end{theo}

\begin{proof} {\it First step.} We first consider a Hopf SDSE $(S)$, dilatation of a system of theorem \ref{32}, such that $I=I_0 \cup J_0 \cup K_0$.
The set $J$ of the vertices of $\gs$ admits a partition $J=(J_x)_{x \in I_0} \cup (J_x)_{x \in J_0} \cup (J_x)_{x \in K_0} $. We put:
$$A=\{j \in J\:/\: b_j \neq 0\}, \: B=\{j \in J\:/\: b_j=0\}.$$
In other terms, $i \in A$ if, and only if, ($i\in J_x$, with $x \in I_0$ such that $b_x \neq -1$) or ($i\in J_x$, with $x \in J_0$). As we are in the non-abelian case,
$A \neq \emptyset$. Let us choose $i_x \in J_x$ for all $x \in I$, and $i_{x_0} \in A$. In order to enlighten the notations, we put
$i_0=i_{x_0}$. We define, for all $k\geq 1$:
$$\left\{ \begin{array}{rcl}
p_{i_0}(k)&=&\displaystyle \frac{1}{b_{x_0}} f_{i_0}(k),\\
p_i(k)&=&\displaystyle \frac{1}{b_{x_0}}(f_i(k)-f_{i_0}(k))\mbox{ if } i\in J_{x_0}-\{i_0\},\\
p_{i_x}(k)&=&\displaystyle \frac{1}{b_{x}}f_i(k)-\frac{1}{b_{x_0}}f_{i_0}(k)\mbox{ if } x\neq x_0 \mbox{ and }x\in A,\\
p_{i_x}(k)&=&\displaystyle f_i(k)\mbox{ if } x\in B,\\
p_i(k)&=&\displaystyle \frac{1}{b_{x}}(f_i(k)-f_{i_x}(k))\mbox{ if } i\in J_x-\{i_x\}, \: x\neq x_0 \mbox{ and }x\in A,\\
p_i(k)&=&\displaystyle f_i(k)-f_{i_x}(k)\mbox{ if } i\in J_x-\{i_x\}, \: x\in B.
\end{array}\right.$$
Then direct computations show that the Lie bracket of $\lies$ is given in the following way: for all $k,l \geq 1$,
\begin{itemize}
\item $[p_{i_0}(k),p_{i_0}(l)]=(l-k) p_{i_0}(k+l)$.
\item For all $i\in I$,
$[p_{i_0}(k),p_i(l)]= \left\{ \begin{array}{l}
(l+d_{x_0})p_i(k+l) \mbox{ if }i\in J_{x_0}-\{i_0\},\\
lp_i(k+l)\mbox{ if }i\notin J_{x_0}.
\end{array}\right.$ 
\item For all $i\in J_{x_0}-\{i_0\}$, for all $x \neq x_0$,
$[p_{i_x}(k),p_i(l)]= \left\{ \begin{array}{l}
-d_{x_0}p_i(k+l) \mbox{ if }x \in A,\\
0\mbox{ if }x\in B.
\end{array}\right.$ 
\item For all $x,x'\in I-\{x_0\}$, $[p_{i_x}(k),p_{i_{x'}}(l)]=0$.
\item For all $x,x'\in I-\{x_0\}$, $i\in J_{x'}-\{i_{x'}\}$, 
$[p_{i_x}(k),p_i(l)]= \left\{ \begin{array}{l}
0 \mbox{ if }x \neq x',\\
d_x p_i(k+l)\mbox{ if } x=x'.
\end{array}\right.$ 
\item For all $x,x'\in I-\{x_0\}$, $i\in J_x-\{i_x\}$, $j\in J_{x'}-\{i_{x'}\}$, $[p_i(k),p_j(l)]=0$.
\end{itemize} 
We used the following notations:
$$d_x=\left\{ \begin{array}{l}
\displaystyle \frac{-\beta_x}{1+\beta_x}\mbox{ if }x \in I_0,\:\beta_x \neq -1,\\
1 \mbox{ if }x \in I_0,\: \beta_x=-1,\\
-1\mbox{ if }x \in J_0,\\
0 \mbox{ if }x\in K_0.
\end{array}\right.$$
So the Lie algebra $\lies$ is isomorphic to:
$$\left(W_{d_{x_0},(-d_{x_0},\cdots,-d_{x_0},0,\cdots,0)}^{|J_{x_0}|-1}\oplus \bigoplus_{x \in I-\{x_0\}} W_{0,(0,\cdots,0,d_x,0,\cdots,0)}^{|I_x|-1}\right)
\triangleleft\left( V_0^{|I|-1} \triangleleft \g_{FdB}\right).$$
A basis adapted to this decomposition is:
$$(p_i(k))_{i\in J_{x_0}-\{i_0\},k\geq 1} \cup \left(\bigcup_{x \in I-\{x_0\}} (p_i(k))_{i\in J_{x}-\{i_x\},k\geq 1} \right)
\cup \left(\bigcup_{x \in I-\{x_0\}} (p_{i_x}(k))_{k\geq 1} \right) \cup (p_{i_0}(k))_{k\geq 1}.$$

{\it Second step.} We now assume that $I_1 \neq \emptyset$. Then the descendants of $j \in I_1$ form a system of the first step, so:
$$\lies=W_{I_1} \triangleleft \g_{(S_0)},$$
where $W_{I_1}=Vect(f_j(k)\:/\: j\in I_1,k\geq 1\}$ and $(S_0)$ is a restriction of $(S)$ as in the first step.
Let us fix $j \in I_1$ and let us consider the $\g_{(S_0)}$-module $W_j=Vect(f_j(k)\:/\:k\geq 1)$. With the notations of the preceding step:
\begin{itemize}
\item $[p_{i_0}(k),f_j(l)]=\left(l-1+\frac{a_{i_0}^{(j)}}{b_{x_0}}\right) f_j(k+l)$ if $l=1$.
\item $[p_{i_0}(k),f_j(l)]=\left(l-1+\nu_j \frac{a_{i_0}^{(j)}}{b_{x_0}}\right) f_j(k+l)$ if $l\geq 2$.
\item $[p_{i_x}(k),f_j(l)]=\left(\frac{a^{(j)}_{i_x}}{b_x}-\frac{a_{i_0}^{(j)}}{b_{x_0}}\right) f_j(k+l)$ if $l=1$, $x\in A$.
\item $[p_{i_x}(k),f_j(l)]=\nu_j\left(\frac{a^{(j)}_{i_x}}{b_x}-\frac{a_{i_0}^{(j)}}{b_{x_0}}\right) f_j(k+l)$ if $l\geq 2$, $x\in A$.
\item $[p_{i_x}(k),f_j(l)]=a^{(j)}_{i_x} f_j(k+l)$ if $l=1$, $x\in B$.
\item $[p_{i_x}(k),f_j(l)]=\nu_j a^{(j)}_{i_x} f_j(k+l)$ if $l\geq 2$, $x\in B$.
\item $[p_i(x),f_j(l)]=0$ if $i$ is not a $i_x$.
\end{itemize}
If $\nu_j\neq 0$, we put $p_j(k)=f_j(k)$ if $k \geq 2$ and $p_j(1)=\nu_j f_j(1)$. Then, for all $l$:
\begin{itemize}
\item $[p_{i_0}(k),p_j(l)]=\left(l-1+\nu_j \frac{a_{i_0}^{(j)}}{b_{x_0}}\right) p_j(k+l)$.
\item $[p_{i_x}(k),p_j(l)]=\nu_j\left(\frac{a^{(j)}_{i_x}}{b_x}-\frac{a_{i_0}^{(j)}}{b_{x_0}}\right) p_j(k+l)$ if $x\in A$.
\item $[p_{i_x}(k),p_j(l)]=\nu_j a^{(j)}_{i_x} p_j(k+l)$ if $x\in B$.
\item $[p_i(x),p_j(l)]=0$ if $i$ is not a $i_x$.
\end{itemize}
So $W_j$ is a module $W_{c,\upsilon}$. If $\nu_j=0$ and $a^{(j)}_{i_0} \neq 0$, we put $p_j(k)=f_j(k)$ if $k \geq 2$ and 
$p_j(1)=\frac{b_{x_0}}{a^{(j)}_{i_0}} f_j(1)$. Then:
\begin{itemize}
\item $[p_{i_0}(k),p_j(l)]=p_j(k+l)$ if $l=1$.
\item $[p_{i_0}(k),p_j(l)]=(l-1)p_j(k+l)$ if $l\geq 2$.
\item $[p_{i_x}(k),f_j(l)]=\left(\frac{a^{(j)}_{i_x}}{b_x}-\frac{a_{i_0}^{(j)}}{b_{x_0}}\right) f_j(k+l)$ if $l=1$, $x\in A$.
\item $[p_{i_x}(k),f_j(l)]=0$ if $l\geq 2$, $x\in A$.
\item $[p_{i_x}(k),f_j(l)]=a^{(j)}_{i_x} f_j(k+l)$ if $l=1$, $x\in B$.
\item $[p_{i_x}(k),f_j(l)]=0$ if $l\geq 2$, $x\in B$.
\item $[p_i(x),p_j(l)]=0$ if $i$ is not a $i_x$.
\end{itemize}
So $W_j$ is a module $W'_{\nu,1}$. If $\nu_j=0$ and $a^{(j)}_{i_0}=0$, we put $p_j(k)=f_j(k)$ for all $k \geq 1$. Then:
\begin{itemize}
\item $[p_{i_0}(k),p_j(l)]=(l-1)p_j(k+l)$.
\item $[p_{i_x}(k),f_j(l)]=\left(\frac{a^{(j)}_{i_x}}{b_x}-\frac{a_{i_0}^{(j)}}{b_{x_0}}\right) f_j(k+l)$ if $l=1$, $x\in A$.
\item $[p_{i_x}(k),f_j(l)]=0$ if $l\geq 2$, $x\in A$.
\item $[p_{i_x}(k),f_j(l)]=a^{(j)}_{i_x} f_j(k+l)$ if $l=1$, $x\in B$.
\item $[p_{i_x}(k),f_j(l)]=0$ if $l\geq 2$, $x\in B$.
\item $[p_i(x),p_j(l)]=0$ if $i$ is not a $i_x$.
\end{itemize}
So $W_j$ is a module $W'_{\nu,0}$.\\

{\it Last step.} We now consider vertices in $J_1$. If $j\in J_1$, then its descendants are vertices of the first step and $i$ elements of $I_1$ such 
that $\nu_i=1$. As before:
$$\lies=W_{J_1} \triangleleft \g_{(S_1)},$$
where $W_{J_1}=Vect(f_j(k)\:/\: j\in J_1,k\geq 1\}$ and $(S_1)$ is a restriction of $(S)$ as in the second step. Let us fix $j \in J_1$ and let us consider 
the $\g_{(S_1)}$-module $W_j=Vect(f_j(k)\:/\:k\geq 1)$. As $\nu_j \neq 0$, putting $p_j(k)=f_j(k)$ if $k \geq 2$ and $p_j(1)=\nu_j f_j(1)$:
\begin{itemize}
\item $[p_{i_0}(k),p_j(l)]=\left(l-1+\nu_j \frac{a_{i_0}^{(j)}}{b_{x_0}}\right) p_j(k+l)$.
\item $[p_{i_x}(k),p_j(l)]=\nu_j\left(\frac{a^{(j)}_{i_x}}{b_x}-\frac{a_{i_0}^{(j)}}{b_{x_0}}\right) p_j(k+l)$ if $x\in A$.
\item $[p_{i_x}(k),p_j(l)]=\nu_j a^{(j)}_{i_x} p_j(k+l)$ if $x\in B$.
\item $[p_i(k),p_j(l)]=\nu_j a^{(j)}_i p_j(k+l)$ if $l=1$, $i\in I_1$, with $\nu_i=1$.
\item $[p_i(k),p_j(l)]=0$ if $l\geq 2$, $i\in I_1$.
\item $[p_i(x),p_j(l)]=0$ if $i\notin I_1$ and is not a $i_x$.
\end{itemize}
So $W_j$ is a module $W''_{c,\nu,\mu}$. \end{proof}

\begin{theo} \label{63}
Let $(S)$ be a connected, extended, fundamental, non-abelian SDSE. Then the Lie algebra $\lies$ is of the form:
$$\g_m\triangleleft(\g_{m-1}\triangleleft(\cdots \g_2 \triangleleft (\g_1 \triangleleft \g_0)\cdots),$$
where $\g_0$ is the Lie algebra associated to the restriction of $(S)$ to the vertices which are not extension vertices (so $\g_0$ is described 
in theorem \ref{62}) and, for $j\geq 1$, $\g_j$ is an abelian  $(\g_{j-1}\triangleleft(\cdots \g_2 \triangleleft (\g_1 \triangleleft \g_0)\cdots)$-module having a
basis $(h^{(j)}(k))_{k\geq 1}$. 
\end{theo}

\begin{proof} The Lie algebra $\g_j$ is the Lie algebra $Vect(f_{x_j}(k)\:/\:k\geq 1)$, 
where $J_2=\{x_1,\ldots,x_m\}$, with the notations of theorem \ref{14}. \end{proof}

\subsection{Associated group}

Let us now consider the character group $Ch\left(\hs \right)$ of $\hs$. In the preceding cases, $\lies$ contains a sub-Lie algebra isomorphic to
the Faà di Bruno Lie algebra, so $Ch\left(\hs \right)$ contains a subgroup isomorphic to the Faà di Bruno subgroup:
$$G_{FdB}=\{x+a_1x^2+a_2x^3+\cdots \:|\: \forall i, \;a_i\in K\},$$
with the product defined by $A(x).B(x)=B\circ A(x)$. Moreover, each modules earlier defined on $\g_{FdB}$ corresponds to a module over $G_{FdB}$
by exponentiation:

\begin{defi}\textnormal{\begin{enumerate}
\item The module $\modulev_0$ is isomorphic to $yK[[y]]$ as a vector space, and the action of $G_{FdB}$ is given by:
$$A(x).P(y)=P\circ A(y).$$
\item Let $G=\left(\modulev_0^{\oplus M}\right)\rtimes G_{FdB}$. Let $c \in K$, and $\upsilon=(\upsilon_1,\cdots,\upsilon_M) \in K^M$. 
Then $\modulew_{c,\upsilon}$ is $zK[[z]]$ as a vector space, and the action of $G$ is given by:
$$(P_1(y),\cdots,P_M(y),A(x)).Q(z)=exp\left(\sum_{i=1}^M \upsilon_i P_i(z)\right)\left(\frac{A(z)}{z}\right)^c Q\circ A(z).$$
\item Let us consider the following semi-direct product:
$$G=\left(\modulew_{c_1,\varepsilon^{(1)}}\oplus \cdots \oplus \modulew_{c_N,\varepsilon^{(N)}}\right)\triangleleft
\left(\modulev_0^{\oplus M}\triangleleft G_{FdB}\right).$$ 
\begin{enumerate}
\item Let $\nu=(\nu_1,\cdots,\nu_M) \in K^M$. Then $\modulew'_{\nu,0}$ is $tK[[t]]$ as a vector space, and  for all
$X= (Q_1(z),\cdots,Q_N(z),P_1(y),\cdots,P_M(y),A(x))\in G$:
\begin{eqnarray*}
X.t&=&\left(1+\sum_{i=1}^M  \nu_i P_i(t)\right)t,\\
X.R(t)&=&\left(\frac{t}{A(t)}\right)R\circ A(t),
\end{eqnarray*}
for all $R(t)\in t^2K[[t]]$.
\item Let $\nu=(\nu_1,\cdots,\nu_M) \in K^M$. Then $\modulew'_{\nu,1}$ is $tK[[t]]$ as a vector space, and  for all 
$X= (Q_1(z),\cdots,Q_N(z),P_1(y),\cdots,P_M(y),A(x))\in G$:
\begin{eqnarray*}
X.t&=&\left(1+\sum_{i=1}^M  \nu_i P_i(t)\right)
\left(t+t\ln\left(\frac{A(t)}{t}\right)\right),\\
X.R(t)&=&\left(\frac{t}{A(t)}\right)R\circ A(t),
\end{eqnarray*}
for all $R(t)\in t^2K[[t]]$.
\item Let $c \in K$, $\nu=(\nu_1,\cdots,\nu_M) \in K^M$, $\mu=(\mu_1,\ldots,\mu_N) \in K^N$. Then $\modulew''_{c,\nu,\mu}$ is $tK[[t]]$ as a vector space,
and  for all $X= (Q_1(z),\cdots,Q_N(z),P_1(y),\cdots,P_M(y),A(x))\in G$:
\begin{eqnarray*}
X.t&=&\left(\frac{A(t)}{t}\right)^c exp\left(\sum_{i=1}^M  \mu_i P_i(t)\right)\left(1+\sum_{i=1}^M  \mu_i Q_i(t)\right)A(t),\\
X.R(t)&=&\left(\frac{t}{A(t)}\right)^{c}exp\left(\sum_{i=1}^M  \mu_i P_i(t)\right)R\circ A(t),
\end{eqnarray*}
for all $R(t)\in t^2K[[t]]$.
\end{enumerate}\end{enumerate}}\end{defi}

Direct computations prove that they are indeed modules.

\begin{theo} \label{65}
Let $(S)$ be a connected Hopf SDSE in the non-abelian, fundamental case. Then the group $Ch\left(\hs \right)$ is of the form:
$$G_m\rtimes(G_{m-1}\rtimes(\cdots G_2 \rtimes (G_1 \rtimes G_0)\cdots),$$
where $G_0$ is a semi-direct product of the form:
$$G_0=\modulew'\rtimes (\modulew \rtimes (\modulev \rtimes G_{FdB})),$$
where $\modulev$ is a direct sum of modules $\modulev_0$, $\modulew$ a direct sum of modules $\modulew_{c,\upsilon}$, and $\modulew'$ a direct 
sum of modules $\modulew'_{\nu,0}$,  $\modulew'_{\nu,1}$ and $\modulew''_{c,\nu,\mu}$. Moreover, for all $m\geq 1$, $G_m=(tK[[t]],+)$ as a group.
\end{theo}

\begin{proof} The group $Ch\left(\hs \right)$ is isomorphic to the group of characters of ${\cal U}(\g)^*$, where $\g$ is described in theorem \ref{63}. 
This implies that this group has a structure of semi-direct product as described in theorem \ref{65}. Let us consider the Hopf algebra $\h$ of coordinates 
of $G_0$. It is a graded Hopf algebra, and direct computations prove that its graded dual is the enveloping algebra of $\g_0$ of theorem \ref{63}. 
So $\h$ is isomorphic to $\h_{(S_0)}$. \end{proof}

\section{Lie algebra and group associated to $\hs$, abelian case}

We now treat the abelian case. Recall that in this case, $J_0=K_0=\emptyset$ and, for all $i\in I_0$, $\beta_i=-1$.

\subsection{Modules over an abelian Lie algebra}

Let $\g_{ab}$ be an abelian Lie algebra, with basis $\left(e^{(i)}(k)\right)_{1\leq i \leq M,k\geq 1}$. We define a family of modules over this Lie algebra:

\begin{defi}\textnormal{Let $\upsilon=(\upsilon_1,\cdots,\upsilon_M)\in K^M$.
Then $V_\upsilon$ has a basis $(f(k))_{k\geq 1}$, and the action of $\g_{ab}$ is given by:
$$e^{(i)}(k).f(l)=\upsilon_i f(k+l).$$
}\end{defi}

We can then describe the semi-direct product:

\begin{prop}\label{67}
Let us consider the following Lie algebra:
$$\g=\left(\bigoplus_{i=1}^N V_{\upsilon^{(i)}}\right)\triangleleft \g_{ab}.$$
It has a basis:
$$(e^{(i)}(k))_{1\leq i \leq M,k\geq 1} \cup (f^{(i)}(k))_{1\leq i \leq N,k\geq 1} ,$$
and the Lie bracket is given by:
$$\left\{ \begin{array}{rcl}
[e^{(i)}(k),e^{(j)}(l)]&=&0,\\[0mm]
[e^{(i)}(k),f^{(j)}(l)]&=&\upsilon^{(j)}_i f^{(j)}(k+l),\\[0mm]
[f^{(i)}(k),f^{(j)}(l)]&=&0.
\end{array}\right.$$
\end{prop}

We now define two families of modules over such a Lie algebra.

\begin{defi}\textnormal{Let $\g$ be a Lie algebra of proposition \ref{67}.
\begin{enumerate}
\item Let $\nu=(\nu_1,\ldots,\nu_M)\in K^M$. The module $W_\nu$ has a basis $(g(k))_{k\geq 1}$, and the action of $\g$ is given by:
$$\left\{ \begin{array}{rcl}
e^{(i)}(k).g(1)&=&\nu_i g(k+1),\\
e^{(i)}(k).g(l)&=&0\mbox{ if }l\geq 2,\\
f^{(i)}(k).g(l)&=&0.
\end{array}\right.$$
\item Let $\nu=(\nu_1,\ldots,\nu_M)\in K^M$ and $\mu=(\mu_1,\ldots,\mu_N)\in K^N$, such that for all $1\leq i\leq M$, for all $1\leq j \leq N$,
$\mu_j\left(\nu_i-\upsilon^{(j)}_i\right)=0$. The module $W'_{\nu,\mu}$ has a basis $(g(k))_{k\geq 1}$, and the action of $\g$ is given by:
$$\left\{ \begin{array}{rcl}
e^{(i)}(k).g(l)&=&\nu_i g(k+l),\\
f^{(j)}(k).g(1)&=&\mu_j g(k+1),\\
f^{(j)}(k).g(l)&=&0\mbox{ if }l\geq 2.
\end{array}\right.$$
\end{enumerate}}\end{defi}

{\bf Remark.} The condition $\mu_j\left(\nu_i-\upsilon^{(j)}_i\right)=0$ is necessary for $W'_{\nu,\mu}$ to be a $\g$-module. Indeed:
\begin{eqnarray*}
[e^{(i)}(k),f^{(j)}(l)].g(1)&=&\upsilon^{(j)}_i \mu_j g(k+l+1),\\
e^{(i)}(k).\left(f^{(j)}(l).g(1)\right)-f^{(j)}(l).\left(e^{(i)}(k).g(1)\right)&=&\mu_j \nu_i g(k+l+1).
\end{eqnarray*}

\subsection{Description of the Lie algebra}

We here consider a connected Hopf SDSE $(S)$ in the abelian case.

\begin{theo} \label{69}
Let us consider a Hopf SDSE of abelian fundamental type, with no extension vertices. Then $\lies$ has the following form:
$$\lies\approx W \triangleleft\left(\left(V_{\upsilon^{(1)}}\oplus \ldots \oplus V_{\upsilon^{(N)}} \right) \triangleleft \g_{ab}\right),$$
where $W$ is a direct sum of $W_\nu$ and $W'_{\nu,\mu}$.
\end{theo}

\begin{proof}
{\it First step.} We first consider a Hopf SDSE such that:
$$I=\bigcup_{x\in I_0} J_x.$$
For all $x\in I_0$, let us fix $i_x \in J_x$. We put $p_{i_x}(k)=f_{i_x}(k)$ and $p_i(k)=f_i(k)-f_{i_x}(k)$ if $i\in J_x-\{i_x\}$. Then direct computations show that:
\begin{itemize}
\item $[p_{i_x}(k),p_{i_{x'}(l)}]=0$.
\item $[p_{i_x}(k), p_j(l)]=\delta_{x,x'} p_j(k+l)$ if $j\in J_{x'}-\{i_{x'}\}$.
\item $[p_i(k),p_j(l)]=0$ if $i,j$ are not $i_x$'s.
\end{itemize}
So:
$$\lies\approx \left(\bigoplus_{x\in I_0} V_{(0,\ldots,0,1,0,\ldots,0)}^{\oplus|J_x|-1}\right)\triangleleft \g_{ab},$$
where $\g_{ab}=Vect(p_{i_x}(k)\:/\:x \in I_0,\:k\geq 1)$.\\

{\it Second step.} We now assume that $I_1 \neq \emptyset$. Then the descendants of $j \in I_1$ form a system as in the first step, so:
$$\lies=W_{I_1} \triangleleft \g_{(S_0)},$$
where $W_{I_1}=Vect(f_j(k)\:/\: j\in I_1,k\geq 1\}$ and $(S_0)$ is the restriction of $(S)$ to the regular vertices.
Let us fix $j \in I_1$ and let us consider the $\g_{(S_0)}$-module $W_j=Vect(f_j(k)\:/\:k\geq 1)$. With the notations of the preceding step:
\begin{itemize}
\item $[p_{i_x}(k),f_j(l)]=a^{(j)}_{i_x} f_j(k+l)$ if $l=1$.
\item $[p_{i_x}(k),f_j(l)]=\nu_j a^{(j)}_{i_x} f_j(k+l)$ if $l\geq 2$.
\item $[p_i(x),f_j(l)]=0$ if $i$ is not a $i_x$.
\end{itemize}
If $\nu_j\neq 0$, we put $p_j(k)=f_j(k)$ if $k \geq 2$ and $p_j(1)=\nu_j f_j(1)$. Then, for all $l$:
\begin{itemize}
\item $[p_{i_x}(k),f_j(l)]=\nu_j a^{(j)}_{i_x} f_j(k+l)$.
\item $[p_i(x),f_j(l)]=0$ if $i$ is not a $i_x$.
\end{itemize}
So $W_j$ is a module $V_\upsilon$. If $\nu_j=0$, we put $p_j(k)=f_j(k)$ for all $k\geq 1$. Then:
\begin{itemize}
\item $[p_{i_x}(k),f_j(l)]=a^{(j)}_{i_x} f_j(k+l)$ if $l=1$.
\item $[p_{i_x}(k),f_j(l)]=0$ if $l\geq 2$.
\item $[p_i(x),f_j(l)]=0$ if $i$ is not a $i_x$.
\end{itemize}
So $W_j$ is a module $W_\nu$.\\

{\it Last step.} We now consider vertices in $J_1$. If $j\in J_1$, then its descendants are vertices of the first step and vertices in $I_1$ such that $\nu_i=1$.
As before:
$$\lies=W_{J_1} \triangleleft \g_{(S_1)},$$
where $W_{J_1}=Vect(f_j(k)\:/\: j\in J_1,k\geq 1\}$ and $(S_1)$ is the restriction of $(S)$ to the regular vertices and the vertices of $I_1$.
Let us fix $j \in J_1$ and let us consider the $\g_{(S_1)}$-module $W_j=Vect(f_j(k)\:/\:k\geq 1)$.
As $\nu_j \neq 0$, putting $p_j(k)=f_j(k)$ if $k \geq 2$ and $p_j(1)=\nu_j f_j(1)$:
\begin{itemize}
\item $[p_{i_x}(k),p_j(l)]=\nu_j a^{(j)}_{i_x} p_j(k+l)$.
\item $[p_i(k),p_j(l)]=0$ if $i\in J_x-\{i_x\}$.
\item $[p_i(k),p_j(l)]=\nu_j a^{(j)}_i p_j(k+l)$ if $l=1$ and $i\in I_1$.
\item $[p_i(k),p_j(l)]=0$ if $l\geq 2$ and $i\in I_1$.
\end{itemize}
So $W_j$ is a module $W'_{\nu,\mu}$. \end{proof}

\begin{theo} \label{70}
Let $(S)$ be a connected Hopf SDSE in the non-abelian, fundamental case. Then the Lie algebra $\lies$ is of the form:
$$\g_m\triangleleft(\g_{m-1}\triangleleft(\cdots \g_2 \triangleleft (\g_1 \triangleleft \g_0)\cdots),$$
where $\g_0$ is the Lie algebra associated to the restriction of $(S)$ to the non-extension vertices (so is described in theorem \ref{69}), and,
for $j\geq 1$, $\g_j$ is an abelian  $(\g_{j-1}\triangleleft(\cdots \g_2 \triangleleft (\g_1 \triangleleft \g_0)\cdots)$-module having a basis $(h^{(j)}(k))_{k\geq 1}$. 
\end{theo}

\begin{proof} Similar with the proof of theorem \ref{62}. \end{proof}

\subsection{Associated group}

Let us now consider the character group $Ch\left(\hs \right)$ of $\hs$. In the preceding cases, $\lies$ contains an abelian sub-Lie algebra $\g_{ab}$,
so $Ch\left(\hs \right)$ contains a subgroup isomorphic to the group:
$$G_{ab}=\left\{\left(a^{(i)}_1x+a^{(i)}_2x^2+\cdots\right)_{1\leq i \leq M},\:|\: \forall 1\leq i \leq M,\forall k\geq 1, \;a^{(i)}_k\in K\right\},$$
with the product defined by $(A^{(i)}(x))_{i\in I}.(B^{(i)}(x))_{i\in I}=(A^{(i)}(x)+B^{(i)}(x)+A^{(i)}(x)B^{(i)}(x))_{i\in I}$. Note that $G_{ab}$ is isomorphic
to the following subgroup of the following group of the units of the ring $K[[x]]^M$:
$$G_1=\left\{\left(\substack{1+xf_1(x)\\ \vdots \\ 1+xf_M(x)}\right)\:\mid \: f_1(x),\ldots,f_M(x) \in K[[x]]\right\}.$$
The isomorphism is given by:
$$\left\{ \begin{array}{rcl}
G_{ab}&\longrightarrow&G_1\\
\left(a^{(i)}_1x+a^{(i)}_2x^2+\cdots\right)_{1\leq i \leq M}&\longrightarrow&
\left(\substack{1+a_1^{(1)}x+a_2^{(1)}x^2+\ldots\\ \vdots \\ 1+a_1^{(M)}x+a_2^{(M)}x^2+\ldots}\right).
\end{array}\right.$$

Moreover, each modules earlier defined on $\g_{ab}$ corresponds to a module over $G_{ab}$ by exponentiation, as explained in the following definition:

\begin{defi}\textnormal{\begin{enumerate}
\item Let $\upsilon=(\upsilon_1,\ldots,\upsilon_M)\in K^M$. The module $\modulev_\upsilon$ is isomorphic to $yK[[y]]$ as a vector space, 
and the action of $G_{ab}$ is given by:
$$(A^{(i)}(x))_{1\leq i\leq M}.P(y)=exp\left(\sum_{i=1}^M \upsilon_i A^{(i)}(y)\right)P(y).$$
\item Let us consider the following semi-direct product:
$$G=\left(\bigoplus_{i=1}^N \modulev_{\upsilon^{(i)}}\right)\triangleleft G_{ab}.$$
\begin{enumerate}
\item Let $\nu=(\nu_1,\ldots,\nu_M)\in K^M$. The module $\modulew_\nu$ is $zK[[z]]$ as a vector space, and the action of $G$ is given in the following way:
for all $X=(P_1(y),\ldots,P_N(y),A_1(x),\ldots,A_m(x))\in G$,
$$\left\{ \begin{array}{rcl}
X.z&=&\displaystyle \left(1+\sum_{i=1}^M \nu_i A_i(z)\right)z,\\
X.z^2R(z)&=&z^2R(z),
\end{array}\right.$$
for all $R(z)\in K[[z]]$.
\item Let $\nu=(\nu_1,\ldots,\nu_M)\in K^M$ and $\mu=(\mu_1,\ldots,\mu_N)\in K^N$, such that for all $1\leq i\leq M$, for all $1\leq j \leq N$,
$\mu_j\left(\nu_i-\upsilon^{(j)}_i\right)=0$. The module $\modulew'_{\nu,\mu}$ is $zK[[z]]$ as a vector space, and the action of $G$ is given 
in the following way: for all $X=(P_1(y),\ldots,P_N(y),A_1(x),\ldots,A_m(x))\in G$,
$$\left\{ \begin{array}{rcl}
X.z&=&\displaystyle exp\left(\sum_{i=1}^M \nu_i A_i(z) \right)\left(1+\sum_{i=1}^N \mu_i P_i(z)\right)z,\\
X.z^2R(z)&=&\displaystyle exp\left(\sum_{i=1}^M \nu_i A_i(z) \right)z^2R(z),
\end{array}\right.$$
for all $R(z) \in K[[z]]$.
\end{enumerate}\end{enumerate}}\end{defi}

Direct computations prove that they are indeed modules. The condition $\mu_j\left(\nu_i-\upsilon^{(j)}_i\right)=0$ is necessary for $\modulew'_{\nu,\mu}$ 
to be a module. Indeed:
\begin{eqnarray*}
A_i(x).(P_j(y).z)&=&\left(exp(\nu_i A_i(z))+\mu_j exp(\nu_i A_i(z)) P_j(z)\right)z,\\
(A_i(x)P_j(y)).z&=&\left(exp(\upsilon^{(j)}_i A_i(y))P_j(y)A_i(x)\right).z\\
&=&\left(1+exp(\upsilon^{(j)}_i A_i(z))P_j(z)\right)z+(exp(\nu_i A_i(z))-1)z\\
&=&\left(exp(\nu_i A_i(z))+\mu_j exp(\upsilon^{(j)}_i A_i(z)) P_j(z)\right)z.
\end{eqnarray*}

\begin{theo} \label{72}
Let $(S)$ be a connected Hopf SDSE in the abelian case. Then the group $Ch\left(\hs \right)$ is of the form:
$$G_N\rtimes(G_{N-1}\rtimes(\cdots G_2 \rtimes (G_1 \rtimes G_0)\cdots),$$
where $G_0$ is a semi-direct product of the form:
$$G_0=\modulew \rtimes (\modulev \rtimes G_{ab}),$$
where $\modulev$ is a direct sum of modules $\modulev_\upsilon$, and $\modulew$ a direct sum of modules $\modulew_\nu$ and $\modulew'_{\nu,\mu}$.
Moreover, for all $m\geq 1$, $G_m=(tK[[t]],+)$ as a group.
\end{theo}

\begin{proof} Similar as the proof of theorem \ref{65}. \end{proof}

\section{Appendix: dilatation of a pre-Lie algebra}

Let $(S)$ be a Hopf SDSE with set of indices $I$. We choose a set $J$ and consider the disjoint union $I'$ of several copies $J_i$ of $J$ indexed by $I$.
The Lie algebra $\lies$ has a basis $(f_i(k))_{i \in I,\:k\geq 1}$ and the Lie bracket is given by:
$$[f_i(k),f_j(l)]=\lambda_l^{(j,i)} f_j(k+l)-\lambda_k^{(i,j)} f_i(k+l).$$
Let $(S')$ be the dilatation of $(S)$ with set of indices $I'$. Then the Lie algebra $\g_{(S')}$ has a basis $(f_i(k))_{i\in J,\:k\geq 1}$ 
and the Lie bracket is given in the following way: for all $x\in J_i$, $y\in J_j$,
$$[f_i(k),f_j(l)]=\lambda_l^{(j,i)} f_y(k+l)-\lambda_k^{(i,j)} f_x(k+l).$$
We shall say that $\g_{(S')}$ is a dilatation of $\g_{(S)}$. We prove in this section that this construction is equivalent to give a pre-Lie product of $\lies$.

\subsection{Dilatation of a pre-Lie algebra}

\begin{defi}\textnormal{\cite{Chapoton2}
A {\it permutative, associative} algebra is a couple $(A,\cdot)$ where $A$ is a vector space and $\cdot$ 
is a bilinear associative (non-unitary) product on $A$ such that for all $a,b,c\in A$:
$$abc=bac.$$
}\end{defi}

\begin{prop}
Let $(A,\cdot)$ be a vector space with a bilinear product. For any pre-Lie algebra $(\g,\star)$, we define a product on $\g \otimes A$ by:
$$(x\otimes a)\star (y \otimes b)=(x\star y) \otimes (ab).$$
Then $\g \otimes A$ is pre-Lie for any pre-Lie algebra $\g$ if, and only if, $A$ is permutative, associative.
\end{prop}

\begin{proof} $\Longleftarrow$. Let $\g$ be a pre-Lie algebra, and let $x,y,z \in \g$, $a,b,c \in A$. Then:
\begin{eqnarray*}
&&((x \otimes a) \star (y\otimes b)) \star (z \otimes c)-(x \otimes a) \star ((y\otimes b) \star (z \otimes c))\\
&=&((x \star y)\star z-x \star(y\star z))\otimes abc\\
&=&((y \star x)\star z-y \star(x\star z))\otimes bac\\
&=&((y \otimes b) \star (x\otimes a)) \star (z \otimes c)-(y \otimes b) \star ((x\otimes a) \star (z \otimes c)).
\end{eqnarray*}
So $\g \otimes A$ is pre-Lie. 

$\Longrightarrow$. Let us assume that $\g\otimes A$ is pre-Lie for any pre-Lie algebra $\g$. Let us choose $\g$ as the pre-Lie algebra $Prim(\h_\D^*)$, 
with $\D$ containing three distinct elements $i,j,k$. Then, for any $a,b,c \in A$:
\begin{eqnarray*}
&&((f_{\tdun{$i$}}\otimes a)\star (f_{\tdun{$j$}}\otimes b)) \star (f_{\tdun{$k$}}\otimes c)-
(f_{\tdun{$i$}}\otimes a)\star ((f_{\tdun{$j$}}\otimes b) \star (f_{\tdun{$k$}}\otimes c))\\
&=&f_{\tdtroisdeux{$k$}{$j$}{$i$}}\otimes (ab)c-\left(f_{\tdtroisun{$k$}{$j$}{$i$}}+f_{\tdtroisdeux{$k$}{$j$}{$i$}}\right)\otimes a(bc)\\
&=&f_{\tdtroisdeux{$k$}{$j$}{$i$}}\otimes \left((ab)c-a(bc)\right)-f_{\tdtroisun{$k$}{$j$}{$i$}}\otimes a(bc)\\
&=&((f_{\tdun{$j$}}\otimes b)\star (f_{\tdun{$i$}}\otimes a)) \star (f_{\tdun{$k$}}\otimes c)-
(f_{\tdun{$j$}}\otimes b)\star ((f_{\tdun{$i$}}\otimes a) \star (f_{\tdun{$k$}}\otimes c))\\
&=&f_{\tdtroisdeux{$k$}{$i$}{$j$}}\otimes \left((ba)c-b(ac)\right)-f_{\tdtroisun{$k$}{$j$}{$i$}}\otimes b(ac).
\end{eqnarray*}
So:
$$f_{\tdtroisdeux{$k$}{$j$}{$i$}}\otimes \left((ab)c-a(bc)\right)-f_{\tdtroisun{$k$}{$j$}{$i$}}\otimes a(bc)
=f_{\tdtroisdeux{$k$}{$i$}{$j$}}\otimes \left((ba)c-b(ac)\right)-f_{\tdtroisun{$k$}{$j$}{$i$}}\otimes b(ac).$$
Applying $\tdtroisdeux{$k$}{$j$}{$i$}\otimes Id_A$ on the two sides of this equality, we obtain $(ab)c-a(bc)=0$.
So $A$ is associative. Applying $\tdtroisun{$k$}{$j$}{$i$}\otimes Id_A$ on the two sides of this equality,
we obtain $a(bc)=b(ac)$, so $A$ is permutative, associative. \end{proof}\\

{\bf Example.} Let $I$ a set, and let $A_I=Vect(e_i)_{i\in I}$. Then $A$ is given a permutative, associative product: for all $i,j \in I$,
$$e_i.e_j=e_j.$$
Let $(\g,\star)$ be a pre-Lie algebra. The pre-Lie product of $\g\otimes A$ is given by:
$$(x\otimes e_i) \star (y \otimes e_j)=x \star y \otimes e_j.$$

The following proposition is immediate:

\begin{prop}
Let $(S)$ be a Hopf SDSE with set of indices $I$ and $(S')$ be a dilatation of $(S)$, with set of indices $J$ being the disjoint union of finite sets $J_i$ 
indexed by $i\in I$. Let $J'$ be a set and for all $i\in I$, let $\phi_i:J_i \longrightarrow J'$ be a map.
The following morphism is a morphism of pre-Lie algebras:
$$\left\{ \begin{array}{rcl}
\g_{(S')}&\longrightarrow & \lies \otimes A_{J'}\\
f_x(k),\:x\in J_i&\longrightarrow & f_i(k) \otimes e_{\phi_i(x)}.
\end{array}\right.$$
It is injective (respectively surjective, bijective) if, and only if, $\phi_i$ is injective (respectively surjective, bijective) for all $i\in I$.
\end{prop}

\subsection{Dilatation of a Lie algebra}

Let $\Set$ be the category of sets, $\Vect$ be the category of Vector spaces, and $\Lie$ the category of Lie algebras.

\begin{defi}\textnormal{
Let $V$ be a vector space. We define a function $F_V$ from $\Set$ to $\Vect$ in the following way:
\begin{enumerate}
\item If $I$ is a set:
$$F_V(I)=\bigoplus_{i\in I} V.$$
The element $v\in V$ in the copy of $V$ corresponding to the index $i\in I$ will be denoted by $v_i$.
\item If $\sigma:I\longrightarrow J$ is a map:
$$ F_V(\sigma) : \left\{ \begin{array}{rcl}
F_V(I)&\longrightarrow & F_V(J)\\
v_i&\longrightarrow &v_{\sigma(i)}.
\end{array}\right.$$
\end{enumerate}}\end{defi}

\begin{defi}\textnormal{
Let $\g$ be a Lie algebra. A {\it dilatation} of $\g$ is functor $F:\Set \longrightarrow \Lie$ such that $F(\{1\})=\g$ and making the following diagram commuting:
$$\xymatrix{\Set\ar[rr]^{F}\ar[rd]_{F_\g}&&\Lie \ar[ld]\\
&\Vect&}$$
where the functor from $\Lie$ to $\Vect$ is the forgetful functor.
}\end{defi}

\begin{prop}
Let $\g$ be a Lie algebra. There is a bijection between the set of dilatations of $\g$ and the set of pre-Lie product inducing the Lie bracket of $\g$.
\end{prop}

\begin{proof} {\it First step.} Let $\star$ be a pre-Lie product inducing the Lie bracket of $\g$. Let $I$ be a set.
We identify $v\otimes e_i\in \g \otimes A_I$ and $v_i \in F_\g(I)$. So $F_\g(I)$ is given a structure of pre-Lie algebra by:
$$v_i \star w_j=(v \star w)_j.$$
The induced Lie bracket is given by:
$$[v_i,w_j]=(v \star w)_j - (w \star v)_i.$$
It is then easy to prove that this structure of pre-Lie algebra on $F_\g(I)$ for all $I$ gives a dilatation of $\g$.\\

{\it Second step.} Let $F$ be a dilatation of $\g$. So for any set $I$, $F_\g(I)$ is now a Lie algebra. Moreover, if $\sigma:I \longrightarrow J$ is any map,
then $F_\g(\sigma):F_\g(I) \longrightarrow F_\g(J)$ is a Lie algebra morphism. We first consider $F_\g(\{1,2\})$.  Let $\pi_2$ be the projection 
on $F_\g(\{2\})$ which vanishes on $F_\g(\{1\})$ in $F_\g(\{1,2\})$. We define $\star$ on $\g$ in the following way: if $v,w \in V$,
$$(v \star w)_2=\pi_2([v_1,w_2]).$$
Let $\sigma:\{1,2\}\longrightarrow \{1,2\}$, permuting $1$ and $2$. Then $F_\g(\sigma)$ permutes the two copies of $\g$ in $\F_\g(\{1,2\})$,
so $F_\g(\sigma)\circ \pi_1=\pi_2 \circ F_\g(\sigma)$. Moreover, $F_\g(\sigma)$ is a morphism of Lie algebras, so for all $v,w \in V$:
\begin{eqnarray*}
F_\g(\sigma)\circ \pi_2([w_1,v_2])&=&\pi_1 \circ F_\g(\sigma)([w_1,v_2]),\\
F_\g(\sigma)((w\star v)_2)&=&\pi_1([w_2,v_1])\\
(w \star v)_1&=&-\pi_1([v_1,w_2]).
\end{eqnarray*}
So, in $F_\g(\{1,2\})$:
$$[v_1,w_2]=\pi_1([v_1,w_2])+\pi_2([v_1,w_2])=(v \star w)_2 -(w\star v)_1.$$

Let us now consider any set $I$ and $i,j \in I$, not necessarily distinct. Considering $\tau: \{1,2\} \longrightarrow \{i,j\}$ sending $1$ to $i$ 
and $2$ to $j$, as $F_\g(\tau)$ is a morphism of Lie algebras, for all $v,w \in \g$, in $F_\g(I)$:
\begin{eqnarray*}
[v_i,w_j]&=&[F_\g(\tau)(v_1),F_\g(\tau)(w_2)]\\
&=&F_\g(\tau)([v_1,w_2])\\
&=&F_\g(\tau)((v \star w)_2-(w \star v)_1)\\
&=&(v \star w)_j -(w \star v)_i.
\end{eqnarray*}
In particular, if $i=j=1$, in $F_\g(\{1\})=\g$, $[v,w]=v \star w-w \star v$: the product $\star$ induces the Lie bracket of $\lies$.

Let $x,y,z \in \g$. In $F_\g(\{1,2,3\})$:
\begin{eqnarray*}
0&=&[x_1,[y_2,z_3]]+[y_2,[z_3,x_1]]+[z_3,[x_1,y_2]]\\
&=&(x \star (y \star z))_3- (x \star (z \star y))_2 - ((y\star z) \star x)_1+((z\star y) \star x)_1\\
&&+(y \star (z \star x))_1- (y \star (x \star z))_3 - ((z\star x) \star y)_2+((x\star z) \star y)_2\\
&&+(z \star (x \star y))_2- (z \star (y \star x))_1 - ((x\star y) \star z)_3+((y\star x) \star z)_3.
\end{eqnarray*}
Considering the terms in the third copy of $\g$:
$$(x \star (y \star z))_3- (y \star (x \star z))_3- ((x\star y) \star z)_3+((y\star x) \star z)_3=0.$$
So $\star$ is pre-Lie.\\

{\it Last step.} We define in the first step a correspondance sending a pre-Lie product on $\g$ to a dilatation of $\g$, and in the second step 
a correspondance sending a dilatation of $\g$ to a pre-Lie product on $\g$. It is clear that they are inverse one from the other. \end{proof}

\bibliographystyle{amsplain}
\bibliography{biblio}

\end{document}